\documentclass{scrartcl}
\usepackage{amsfonts,amsthm,amssymb}
\usepackage{mathtools}

\usepackage[T1]{fontenc}
\usepackage{lmodern}
\usepackage{mathrsfs}
\usepackage{stmaryrd}

\usepackage{url}
\usepackage{hyperref}

\usepackage{enumitem}
\usepackage[capitalize]{cleveref}
\usepackage{crossreftools}
\usepackage[section]{placeins}

\usepackage[all]{xy} 
\usepackage{tikz}

\usepackage{diagbox}

\newtheorem{thm}{Theorem}[section]
\newtheorem{lem}[thm]{Lemma}
\newtheorem{coro}[thm]{Corollary}
\newtheorem{prop}[thm]{Proposition}

\theoremstyle{definition}
\newtheorem{defn}[thm]{Definition}
\newtheorem{ex}[thm]{Example}

\theoremstyle{remark}
\newtheorem{rem}[thm]{Remark}

\crefname{coro}{Corollary}{Corollaries}%
\crefname{lem}{Lemma}{Lemmas}%
\crefname{ex}{Example}{Examples}%
\crefdefaultlabelformat{#2\textup{#1}#3}

\setlist[enumerate,1]{label={\textup{\arabic*.}}}
\setlist[enumerate,2]{label={\textup{\alph*)}}}

\def\Set#1{\Setdef#1\Setdef}
\def\Setdef#1|#2\Setdef{\left\{#1\,\;\mathstrut\vrule\,\;#2\right\}}%inoue-san

\newcommand{\cB}{\mathcal{B}}
\newcommand{\cK}{\mathcal{K}}
\newcommand{\cR}{\mathcal{R}} 
\newcommand{\cX}{\mathcal{X}}

\newcommand{\bS}{\mathbb{S}}

\newcommand{\los}[2][X]{U_{#2}}

\newcommand{\card}[1]{\left|#1\right|}
\newcommand{\Card}[1]{\card{#1}}

\newcommand{\OrdCpx}{\cK}
\newcommand{\Fpst}{\cX}

\newcommand{\geomr}[1]{#1}
\newcommand{\ordcpx}[1]{\OrdCpx\left( #1 \right)}
\newcommand{\geop}[1]{\ordcpx{#1}}

\newcommand{\fpst}[1]{\Fpst\left( #1 \right)}

\newcommand{\Sd}[2][]{{#2}'}
\newcommand{\exch}[1]{\widetilde{{#1}'}}

\newcommand{\susp}{S} % unreduced suspension 

\DeclareMathOperator{\ost}{ost} %open star
\DeclareMathOperator{\st}{st} %star
\DeclareMathOperator{\lk}{lk} %link

\DeclareMathOperator{\supp}{supp} %support

\DeclareMathOperator{\mxl}{mxl} 
\DeclareMathOperator{\mnl}{mnl} 
\DeclareMathOperator{\hgt}{h} 

\newcommand{\len}[1]{l\left(#1\right)}

\allowdisplaybreaks

\title{On the weak homotopy types of \\ small finite spaces}
\author{Kango Matsushima \and Shuichi Tsukuda}

\begin{document}

\maketitle

\renewcommand{\thefootnote}{}
\footnotetext{2020 \emph{Mathematics Subject Classification}: Primary 55P15; Secondary 06A99.}
\footnotetext{\emph{Key words and phrases}: Finite topological spaces, minimal finite models, posets, order complexes.}

\begin{abstract}
We show that a connected finite topological space with $12$ or less points has a weak homotopy type of a wedge of spheres. 
In other words, we show that the order complex of a connected finite poset with $12$ or less points has a homotopy type of a wedge of spheres. 
\end{abstract}

\section{Introduction}

Given a topological space $X$, a finite space $Y$, that is, a topological space with finitely many points, 
is a \emph{finite model} of $X$ 
if it is weak homotopy equivalent to $X$,  
and it is called a \emph{minimal finite model} of $X$ if it is a finite model of minimal cardinality. 
McCord \cite{McCord} showed that given a finite simplicial complex $K$, 
there exists a finite topological space $\fpst{K}$ and 
a weak homotopy equivalence $\mu_K\colon \geomr{K} \to \fpst{K}$, 
whence any compact CW-complex has a minimal finite model. 
Cianci and Ottina \cite{CianciOttina} gave a complete characterization of minimal  finite models of 
the real projective plane, the torus, and the Klein bottle. 
In particular, minimal finite models of the real projective plane have $13$ points.

In that paper, they showed that if $X$ is a connected finite space with $\Card{X}\leq 12$, 
then $\pi_1(X)$ is a free group and its integral homology is torsion free, 
where $\Card{X}$ denotes the cardinality of $X$.
In this article, we refine their study and show a stronger result:
\begin{thm}
 \label{mainthm}
 If $X$ is a connected finite space with $\Card{X}\leq 12$, then 
 $X$ has a weak homotopy type of a wedge of spheres, 
  where we consider one point space as a wedge of $0$-copies of spheres.
 %
 % In other words, the order complex of a connected finite poset with $12$ or less points has a homotopy type of a wedge of spheres. 
\end{thm}

The proof is given by induction on the cardinality of the space. 

\begin{defn}
 Let $X$ be a finite connected space. 
 We say that $X$ \emph{splits into smaller spaces}
 % or \emph{splits into spaces smaller than} $X$
 if there exist finite spaces $A_i$ with $\Card{A_i}< \Card{X}$ and 
 non negative integers $n_i\geq 0$ such that 
 \[
 X\simeq_w \bigvee_i \bS^{n_i}A_i
 \]
 where $\bS$ denotes the Non-Hausdorff suspension (see \cref{nHsuspension}).

 When $X$ is not connected, we say that $X$ splits into smaller spaces 
 if each connected component does.

 More generally, if $\Card{A_i}< \Card{B}$ for some finite space $B$, 
 we say that $X$ 
 \emph{splits into spaces smaller than} $B$. 
 Note that, if $X$ splits into spaces smaller than $B$, then so does $\bS X$ 
 (see \cref{suspdisjoint,invariance-of-susp-and-wedge}).

 We also note that the weak homotopy type of a wedge sum of connected finite $T_0$ spaces 
 is independent of base points and is weak homotopy invariant
 %  Since what we actually treat are order complexes of finite spaces, 
 % we can be loose about the base points when taking wedge sum 
 % we take one of apices as the base point of non-Hausdorff suspension
 (see \cref{invariance-of-susp-and-wedge}). 
\end{defn}

\begin{ex}
 \label{splitexpl}
 Since $\bS^0A=A$, $\bS^n\emptyset \simeq_w S^{n-1}$, 
 if $X$ is weak homotopy equivalent to a wedge of some spheres 
 and a space $A$ with $\Card{A}<\Card{X}$;
 \[
  X\simeq_w A \vee \left(\bigvee_i S^{n_i} \right),
 \]
 then $X$ splits into smaller spaces. 
 % Note that if $X$ is connected, then so is $A$.
\end{ex}

Finite $T_0$ topological spaces and finite posets are essentially the same objects, 
and Cianci-Ottina \cite{CianciOttina} observed that,  
in most cases, a small finite $T_0$ space $X$ can be decomposed into 
the form $U_a \cup F_b \cup \mxl(X) \cup \mnl(X)$ 
where $U_a$ is the set of elements smaller than or equal to $a\in X$, 
$F_b$ is the set of elements greater than or equal to $b\in X$, 
$\mxl(X)$ is the set of maximal elements, and 
$\mnl(X)$ is the set of minimal elements. 
We show that this decomposition splits $X$ into wedge sum of suspensions 
and that, 
if $X$ is a connected finite $T_0$ space with 
$1<\Card{X}\leq 12$, then $X$ splits into smaller spaces, %in \cref{sec-mleq3,sec-lleq5,sec-card12},  
and from which, \cref{mainthm} follows by induction.

The paper is organized as follows:
In \cref{sec-preliminaries}, we recall some basic results on simplicial complexes and finite spaces
and fix notations. 
In \cref{sec-posetsplitting}, we reformulate the poset splitting of Cianci-Ottina 
and show our fundamental splitting results \cref{suspensionposetlemma,UaFbmxlmnlsplit,Uamxlmnlsplit}. 
In most cases, \cref{UaFbmxlmnlsplit,Uamxlmnlsplit} give splittings of small finite spaces into smaller spaces. 
To handle exceptional cases, 
we study the weak homotopy types of posets of intervals in \cref{sec-interval}. 
We describe some very small finite spaces in \cref{sec-verysmallfinite}. 
In \cref{sec-mleq3,sec-lleq5,sec-card12}, we show that, if $X$ is a connected finite $T_0$ space with 
$1<\Card{X}\leq 12$, then $X$ splits into smaller spaces, 
and we give a proof of \cref{mainthm} in \cref{sec-proof}.

\section{Preliminaries}
\label{sec-preliminaries}

In this section, we recall some basic results on simplicial complexes and finite spaces 
and fix notations. 
If two topological spaces $X$ and $Y$ are homeomorphic, we write $X\cong Y$, 
if they are homotopy equivalent, we write $X \simeq Y$, 
and if they are weak homotopy equivalent, we write $X \simeq_w Y$. 
We denote the double mapping cylinder of maps 
$\xymatrix@1{X & A \ar[r]^g \ar[l]_f & Y}$ by $M_{f,g}$. 
If $f_0\simeq f_1$ and $g_0\simeq g_1$, then $M_{f_0,g_0}\simeq M_{f_1,g_1}$.

\subsection{Simplicial complexes}

All simplicial complexes are finite in this section.
We will not distinguish between a simplicial complex and its geometric realization. 
The following are very basic facts of homotopy theory of simplicial complexes:
\begin{itemize}
 \item Weak homotopy equivalent simplicial complexes are homotopy equivalent. 
 \item If a simplicial complex $K$ is a union of two subcomplexes 
       $K=K_1\cup K_2$, then 
       $K/K_2$ is homeomorphic to  $K_1/(K_1\cap K_2)$.
 \item If $L$ is a subcomplex of $K$ (more generally, if $(K,L)$ is a CW pair), 
       then $K/L$ is homotopy equivalent to the double mapping cylinder of 
       $\xymatrix@1{{*} & L \ar[r]^-{i} \ar[l] & K}$ where $i\colon L \to K$ is the inclusion. 
       
 \item If $K=K_1\amalg K_2$, $L_i\subset K_i$ and $L=L_1\amalg L_2$, 
       then $K/L \cong K_1/L_1\vee K_2/L_2$ 
       (If $L_i=\emptyset$, then $K_i/L_i = K_{i}^{+}$, which is $K_i$ with disjoint base point added).
\end{itemize}

If $L_1$ and $L_2$ are subcomplexes of $K$ such that $L_1\cap L_2=\emptyset$, we denote the space 
obtained by collapsing $L_1$ to a single point and $L_2$ to a single point 
by $K/(L_1,L_2)$. 
The unreduced suspension of $K$ is denoted by $\susp K$, 
that is, 
\[
 SK = \frac{K\times I}{(K\times \{0\}, K\times \{1\})}.
\]

\begin{ex}
 \begin{enumerate}
  \item  If the inclusion $L\to K$ is null homotopic, then 
	 \[
	 K/L \simeq K \vee \susp L.
	 \]
	 In particular, 	 
	 \begin{enumerate}
	  \item if $K\simeq *$, then $K/L \simeq \susp L$, 
	  \item if $L\simeq *$, then $K/L\simeq K$.
	 \end{enumerate}
  \item  If $K=K_1\cup K_2$, $K_2 \simeq *$, and $K_1\cap K_2 \to K_1$ is null homotopic, then 
	  \begin{align*}
	   K \simeq K/K_2 \cong K_1/(K_1\cap K_2) \simeq K_1\vee \susp(K_1 \cap K_2).
	  \end{align*}
  \item  If $K=K_1\cup K_2$ and $K_1, K_2 \simeq *$, then 
	 \begin{align*}
	  K \simeq K/K_2 \cong K_1/(K_1\cap K_2) \simeq \susp(K_1 \cap K_2).
	 \end{align*}
 \end{enumerate}
\end{ex}

\begin{ex}
  For a vertex $v\in V(K)$, we denote the subcomplex spanned by $V(K)-\{v\}$ by $K\setminus\{v\}$. 
 Since 
 \begin{align*}
  \st(v)&=\Set{\sigma\in K | \sigma\cup \{v\}\in K} \simeq *, \\
  \lk(v)&=\Set{\sigma\in \st(v) | v\not \in \sigma}, \\ 
  K\setminus\{v\}&=\Set{\sigma \in K | v\not\in \sigma}, \\ 
  K&=\left(K\setminus\{v\}\right) \cup \st(v), \\
  \lk(v) &=\left(K\setminus\{v\}\right)\cap \st(v),
 \end{align*}
 if the inclusion $\lk(v) \to K\setminus \{v\}$ is  null homotopic, 
 we have 
 \begin{align*}
  K&%\simeq K/\st(v) \cong \left(K\setminus\{v\}\right)/\lk(v) 
  \simeq \left(K\setminus \{v\}\right) \vee S(\lk(v)).
 \end{align*}
\end{ex}

\begin{ex}
 \label{suspdisjoint}
 If $K$ is a disjoint union of subcomplexes 
 $K=\coprod_{i=1}^{n}K_i$ ($K_i\ne \emptyset$), 
 then 
 $\susp K$ is homotopy equivalent to the wedge sum of $\susp K_i$ and $n-1$ copies of  $S^1$: 
 \[
  \susp K \simeq \left(\bigvee_{i=1}^n \susp K_i\right)\vee \left(\bigvee_{n-1} S^1\right).
 \]
 In particular, $\susp K$ is a wedge of spheres 
 when each connected components of $K$ is a wedge of spheres. 
\end{ex}

One can easily show the following (We give a proof in \cref{a-suspensionlemma}.):
\begin{lem}
 \label{suspensionlemma}
 Let $K$ be a connected simplicial complex, $L$, $M$ be subcomplexes and $K=L\cup M$. 
 If all the connected components of $L$ and $M$ are contractible;
 $L=\coprod_{i=1}^l L_i$,  $M=\coprod_{i=1}^m M_i$, $L_i\simeq *$, and $M_i\simeq *$, 
 then 
\[
  K \simeq 
 \left(\bigvee_{\substack{i,j \\ L_i\cap M_j \ne \emptyset}} \susp \left(L_i\cap M_j\right)  \right)  
 \vee \left(\bigvee_{n} S^1\right) 
\]
 where
 \[
  n=\card{\Set{(i,j) | L_i\cap M_j \ne\emptyset}} - l - m +1.
 \]
\end{lem}

Finally, we note that, since any point of a (geometric realization of) simplicial complex 
is a nondegenerate base point, the homotopy type of a wedge sum of connected simplicial complexes 
is independent of base points and is homotopy invariant, that is, 
if $K_1$, $K_2$, $L_1$, and $L_2$ are connected simplicial complexes such that $K_i\simeq L_i$, 
 then $K_1\vee K_2\simeq L_1\vee L_2$.

\subsection{Finite spaces}

In this subsection, we collect some basic facts on finite spaces. 
See \cite{MayDraft} and \cite{Barmak} for details.

% \redcom{It would be better to use Alexandroff}

\paragraph{Finite spaces}
Alexandroff \cite{Alexandroff} showed that finite topological spaces and finite preordered sets are 
essentially the same objects: 
Given a preordered set $X$, the set of all down sets of $X$ gives a topology on $X$, 
and if $X$ is finite, this correspondence gives a bijection between preorders on $X$ and 
topologies on $X$. 
We call a finite topological space ($=$ finite preordered set) a \emph{finite space}. 
Moreover, a map between finite spaces are continuous if and only if it is monotone.

All the maps between finite spaces are continuous unless otherwise stated.

\begin{prop}
 Let $f,g\colon X \to Y$ be maps between finite spaces. 
 Then $f\simeq g$ if and only if there exists a sequence of maps 
 $f_0, \dots, f_n\colon X \to Y$ satisfying 
 $f=f_0\leq f_1 \geq f_2\leq \dots f_n=g$.

 In particular, if $f\leq g$, then $f\simeq g$.
\end{prop}

McCord showed that a finite space is $T_0$ if and only if it is a poset. 
If $X$ is a finite preordered set, then the projection to its maximal quotient poset is a homotopy equivalence. 
Therefore, when considering homotopy types of finite spaces, one may consider only finite $T_0$ spaces.

Let $X$ be a finite space and $a\in X$. We denote
\begin{align*}
 U_a^X&=\Set{x\in X | x\leq a}, & \widehat{U}_a^X&=\Set{x\in X | x<a}, \\
 F_a^X&=\Set{x\in X | x\geq a}, & \widehat{F}_a^X&=\Set{x\in X | x>a}, \\
 C_a^X&=U_a\cup F_a, & \widehat{C}_a^X&=C_a-\{a\}, \\
 \mxl(X)&=\Set{x\in X | \text{$x$ is maximal.}}, & 
 \mnl(X)&=\Set{x\in X | \text{$x$ is minimal.}}. 
\end{align*}
We often omit the superscript $X$ and write $U_a$ instead of $U_a^X$ etc.

It is easy to see the following:
\begin{lem}
 \label{elementaryfactsonUaFb}
 Let $X$ be a finite $T_0$ space and $a,b \in X$. The following holds:
 \begin{enumerate}
  \item  $a\leq b \Leftrightarrow a\in U_b \Leftrightarrow b\in F_a  \Leftrightarrow U_a\subset U_b \Leftrightarrow F_a\supset F_b\Leftrightarrow F_a\cap U_b\ne \emptyset$.
  \item % If $U_a\cap U_b\ne \emptyset$, 
	% then either 
	$U_a\cap U_b\simeq *$ or $\card{U_a\cap U_b}<\card{X}$. 
	
	In fact, if $a\in U_b$, then $U_a\subset U_b$ and $U_a\cap U_b=U_a\simeq *$. 
	If $a\not\in U_b$, then $a\not \in U_a\cap U_b$ and $U_a\cap U_b \subsetneq U_a\subset X$.
  \item $U_a\cup U_b$ is connected $\Leftrightarrow$ $U_a\cap U_b\ne \emptyset$.
  \item Let $A\subset X$ be a subspace and $a\in A$.   
	\begin{enumerate}
	 \item $U^A_a=U^X_a \cap A$.
	 \item $F^A_a=F^X_a \cap A$.
	 \item $\mxl(X)\cap A\subset \mxl(A)$.
	 \item $\mnl(X)\cap A \subset \mnl(A)$.
	\end{enumerate}
 \end{enumerate}
\end{lem}
% \begin{proof}
%  \begin{enumerate}
%   \item Clear.
%   \item As shown.
%   \item Clearly, if $U_a\cap U_b\ne \emptyset$, then $U_a\cup U_b$ is connected. 

% 	Suppose $U_a\cup U_b$ is connected. 
% 	If $a$ and $b$ are comparable, say $a\leq b$, then 
% 	$U_a\cap U_b=U_a\ne\emptyset$. 
% 	If $a$ and $b$ are incomparable, then 
% 	$a$ and $b$ are maximal elements of $U_a\cup U_b$. 
% 	Since $U_a\cup U_b$ is connected, 
% 	there exists a fence 
% 	$a=x_0\geq y_0 \leq x_1 \geq y_1 \leq \dots \geq y_{n-1} \leq x_n=b$ 
% 	in $U_a\cup U_b$ connecting $a$ and $b$. 
% 	Since $x_n=b\not\in U_a$, $\Set{i | x_i\not\in U_a}\ne\emptyset$. 
% 	Put $j=\min\Set{i | x_i\not\in U_a}$. 
% 	Since $x_j\not\in U_a$, we have $x_j\in U_b$. 
% 	Since $x_0=a\in U_a$, we have $j\geq 1$ and $x_{j-1}\in U_a$. 
% 	Therefore we have 
% 	$a\geq x_{j-1}\geq y_{j-1}\leq x_j\leq b$, and hence $y_{j-1}\in U_a\cap U_b$.
%   \item Clear.
%  \end{enumerate}
% \end{proof}

\paragraph{Homotopy types}
Stong \cite{Stong} studied homotopy types of finite spaces. 

\begin{defn}
 Let $X$ be a finite $T_0$ space and $x\in X$. 
 We call $x$ a \emph{down beat point} if $x$ covers one and only one element of $X$. 
 In other words, $x$ is a down beat point if $\widehat{U}_x$ has a maximum. 
 Dually, $x$ is an \emph{up beat point} if $\widehat{F}_x$ has a minimum. 
\end{defn}

\begin{prop}[Stong \cite{Stong}]
 Let $X$ be a finite $T_0$ space and $x\in X$ a beat point.  
 Then $X\setminus \{x\}$ is a strong deformation retract of $X$. 
\end{prop}

Actually, we can remove down beat points at once. 
\begin{lem}
 Let $X$ be a finite $T_0$ space, $x\in X$ a down beat point and $x\ne y\in X$.  
 If $y$ is a down beat point of $X$, then $y$ is a down beat point of $X\setminus \{x\}$.
\end{lem}
\begin{proof}
 Since $x$ is a down beat point of $X$, $\widehat{U}^X_x$ has the maximum . 
 We put $\underline{x}=\max\widehat{U}^X_x$. 
 We have $\widehat{U}^X_x=U^X_{\underline{x}}$. 

 Assume that $y$ is a down beat point of $X$ and put
 $\underline{y}=\max \widehat{U}^X_y$. 

 Since $\widehat{U}^{X\setminus \{x\}}_y=\widehat{U}^X_y-\{x\}$, if 
 $x\not\in U^X_y$, then $\underline{y}=\max \widehat{U}^{X\setminus \{x\}}_y$. 
 
 Consider the case when $x\in U^X_y$, that is, $x\leq y$. 
 Since $x\ne y$, we have $x\in \widehat{U}^X_y$.

 If $x\ne \underline{y}=\max \widehat{U}^X_y$, we have   
 $\underline{y}=\max\left(\widehat{U}^X_y-\{x\}\right)=\max \widehat{U}^{X\setminus \{x\}}_y$. 

 If $x=\underline{y}$, we have $\widehat{U}^X_y=U_x^X$ whence 
 \begin{align*}
  \widehat{U}^{X\setminus \{x\}}_y&=\widehat{U}^X_y-\{x\} 
  =U^X_x - \{x\}  =\widehat{U}^X_x = U^X_{\underline{x}}
 \end{align*}
 Therefore $\underline{x}=\max \widehat{U}^{X\setminus \{x\}}_y$.
\end{proof}

\begin{coro}
 Let $X$ be a finite $T_0$ space, $A\subset X$ a subset and assume that 
 all points in $X-A$ are down beat points. 
 Then $A$ is a strong deformation retract of $X$. 
\end{coro}

\begin{rem}
 Removing down beat points may affect up beat points. 
 % Let $x,y\in X$, $x\ne y$, $x$ a down beat point of $X$ and $y$ an up beat point of $X$. 
 % Then $y$ may not be an  up beat point of $X\setminus \{x\}$. 
 Therefore, removing up and down beat points at once could change the weak homotopy type.
 See \cref{fig-rmv-up-down}.
 % $a$ is an up beat point, $b$ is a down beat point. 
 % Removing both $a,b$ at once, we get $S^1$.
 % $b$ is not a beat point after removing $a$.
 % $a$ is not a beat point after removing $b$.

 \begin{figure}[h]
  \centering
  \begin{tikzpicture}

   \begin{scope}[xshift=4cm]
   \node (-1) at (-1,0) {$\bullet$};
   \node (1) at (1,0) {$\bullet$}; 
   \node (a) at (0,1) {$a$};
   \node (b) at (0,2) {$b$};
   \node (-13) at (-1,3) {$\bullet$};
   \node (13) at (1,3) {$\bullet$}; 

   \draw (-13) -- (b) -- (13);
   \draw (-1) -- (a) -- (1); 
   \draw (a) -- (b);
   \end{scope}

   \begin{scope}[xshift=8cm, yshift=-4cm]
   \node (-1) at (-1,0) {$\bullet$};
   \node (1) at (1,0) {$\bullet$}; 
   % \node (a) at (0,1) {$a$};
   \node (b) at (0,2) {$b$};
   \node (-13) at (-1,3) {$\bullet$};
   \node (13) at (1,3) {$\bullet$}; 

   \draw (-13) -- (b) -- (13);
   \draw (-1) -- (b) -- (1);
   \end{scope}

   \begin{scope}[yshift=-4cm]
   \node (-1) at (-1,0) {$\bullet$};
   \node (1) at (1,0) {$\bullet$}; 
   \node (a) at (0,1) {$a$};
   % \node (b) at (0,2) {$b$};
   \node (-13) at (-1,3) {$\bullet$};
   \node (13) at (1,3) {$\bullet$}; 

   \draw (-13) -- (a) -- (13);
   \draw (-1) -- (a) -- (1);
   \end{scope}

   \begin{scope}[xshift=4cm, yshift=-4cm]
   \node (-1) at (-1,0) {$\bullet$};
   \node (1) at (1,0) {$\bullet$}; 
   % \node (a) at (0,1) {$a$};
   % \node (b) at (0,2) {$b$};
   \node (-13) at (-1,3) {$\bullet$};
   \node (13) at (1,3) {$\bullet$}; 

   \draw (-1) -- (13) -- (1);
   \draw (-1) -- (-13) -- (1);
   \end{scope}

  \end{tikzpicture}
  \caption{Removing up and down beat points.}
  \label{fig-rmv-up-down}
 \end{figure}
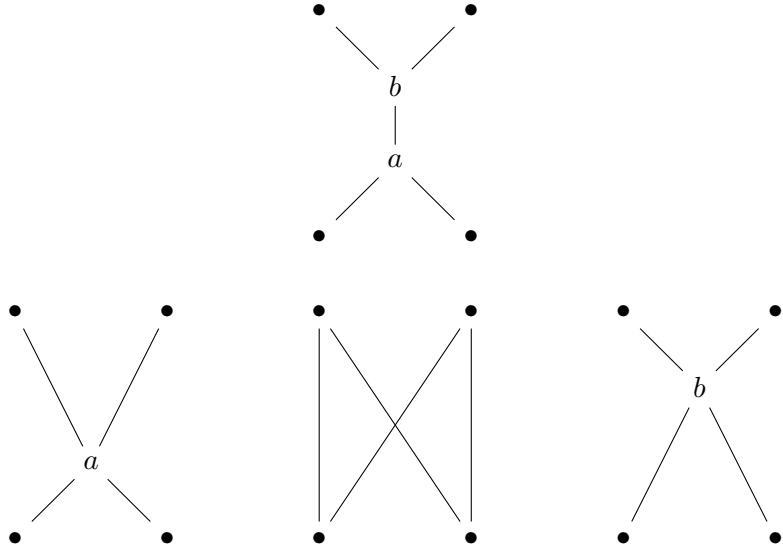
\end{rem}

\begin{defn}
 A finite $T_0$ space is called a \emph{minimal finite space} if it has no beat points. 
 A \emph{core} of a finite space is a strong deformation retract of the space 
which is a minimal finite space.
\end{defn}

\begin{thm}[Stong \cite{Stong}]
 Any finite space has a core. 

 Two finite spaces are homotopy equivalent if and only if they have homeomorphic cores.
\end{thm}

Observe the following:
\begin{lem}
 \label{Fb-Uanonempty}
 Let $X$ be a minimal finite space. 
 If $b\in X$ is not a maximal element, then for all $a\in X$, $F_b-U_a \ne \emptyset$.
\end{lem}
\begin{proof}
 We show that if there exists a point $a\in X$ such that $F_b-U_a=\emptyset$, then $X$ has a beat point. 

 If $F_b-U_a=\emptyset$, then $F_b\subset U_a$. 
 Since $b$ is not maximal, 
 there exists an element $b'\in X$ such that $b<b'$. 
 Since $b'\in F_b\subset U_a$, we have  $b<b'\leq a$, in particular, $b\ne a$. 
 Hence $F_b-\{a\}\ne \emptyset$. 
 Since $F_b-\{a\}$ is a nonempty finite set, there exists a maximal element 
 $x\in \mxl(F_b-\{a\})$. 
 We have $\widehat{F}_x \cap \left(F_b-\{a\}\right)=\emptyset$ and $\widehat{F}_x\subset F_b$, whence 
 $\widehat{F}_x - \{a\}=\emptyset$, that is, $\widehat{F}_x\subset \{a\}$. 
 Since $x\in F_b-\{a\}\subset U_a-\{a\}$, we have $x<a$. 
 Therefore $\widehat{F}_x=\{a\}$ whence 
 $x$ is an up beat point.
\end{proof}

\paragraph{Weak homotopy types} 
Given a finite $T_0$ space $X$, we denote its \emph{order complex}, 
that is, the simplicial complex of nonempty chains of $X$, by $\ordcpx{X}$. 
We denote the \emph{face poset} of a simplicial complex $K$ by $\fpst{K}$. 

\begin{thm}[McCord \cite{McCord}] 
 Let $X$ be a finite $T_0$ space. 
 The map $\mu_X\colon \ordcpx{X} \to X$ given by $\mu_X(\alpha)=\min(\supp(\alpha))$ is a 
 weak homotopy equivalence. 

 The map $\mu_X$ is natural with respect to $X$, that is, 
 if $f\colon X \to Y$ is a map between finite spaces, then the following diagram is commutative:
\[
 \xymatrix{
 \geop{X} \ar[r]^-{\geop{f}} \ar[d]_-{\mu_X}^-{\simeq_w} & \geop{Y} \ar[d]^-{\mu_Y}_-{\simeq_w} \\
 X \ar[r]_-{f} & Y
}
\] 
\end{thm}

\begin{coro}
 Let $X$, $Y$ be finite $T_0$ spaces. Then, 
 $X\simeq_w Y$ if and only if $\geop{X}\simeq \geop{Y}$.

 In particular, $X \simeq_w X^o$, where $X^o$ is the opposite of $X$.
\end{coro}

A finite $T_0$ space $X$ is said to be \emph{homotopically trivial} if 
$\geop{X}\simeq *$, equivalently if $X \simeq_w *$.

\begin{prop}
 Let $f,g\colon X \to Y$ be maps between finite $T_0$ spaces. 
 If $f\simeq g$, then $\geop{f}\simeq \geop{g}$.
\end{prop}

When one considers the weak homotopy types of finite spaces, 
Quillen's Theorem A is quite a powerful tool. 

\begin{thm}[Quillen \cite{Quillen}]
 Let $X$, $Y$ be finite $T_0$ spaces and $f\colon X \to Y$ a map. 
 
 If $f^{-1}(U_y) \simeq_w *$ for every $y\in Y$, then $f$ is a weak homotopy equivalence. 

 Dually, if $f^{-1}(F_y) \simeq_w *$ for every $y\in Y$, then $f$ is a weak homotopy equivalence. 
\end{thm}

\begin{defn}
 \label{nHsuspension}
 Let $S^0$ be the $0$-dimensional sphere, that is, 
 the discrete space with $2$ points.

 The ordinal sum $X * S^0$ of a finite $T_0$ space $X$ and $S^0$ is 
 called the \emph{non-Hausdorff suspension} of $X$ and 
 is denoted by $\bS X$. 
 We define the $n$-fold non-Hausdorff suspension inductively by 
 $\bS^0X=X$, 
 $\bS^nX=\bS(\bS^{n-1}X)$. 
\end{defn}

\begin{prop}[McCord \cite{McCord}, Barmak \cite{Barmak}]
 \label{suspension-wedge-sum}
Let $X$, $Y$ be finite $T_0$ spaces. The following holds:
\begin{align*}
 \bS X&\simeq_w \susp\ordcpx{X}, \\
 X\vee Y&\simeq_w \ordcpx{X}\vee \ordcpx{Y}.
\end{align*}
\end{prop}

\begin{coro}
 \label{invariance-of-susp-and-wedge}
 Let $X$, $Y$, $X_1$, $X_2$, $Y_1$, and $Y_2$ be finite $T_0$ spaces 
 and $K$, $K_1$ and $K_2$ be simplicial complexes. The following holds:
 \begin{enumerate}
  \item If $X\simeq_w K$, then $\bS X \simeq_w SK$. 
  \item If $X\simeq_w Y$, then $\bS X \simeq_w \bS Y$. 
  \item The weak homotopy type of a wedge sum of connected finite $T_0$ spaces is independent of basepoints.
  \item If $X_1$ and $X_2$ are connected and $X_i\simeq_w K_i$, then $X_1\vee X_2 \simeq_w K_1\vee K_2$.
  \item If $X_1$ and $X_2$ are connected and $X_i\simeq_w Y_i$, then $X_1\vee X_2 \simeq_w Y_1\vee Y_2$.
  \item $\bS(X\vee Y)\simeq_w \bS X \vee \bS Y$.
 \end{enumerate}
\end{coro}
\begin{proof}
 Recall that the homotopy type of a wedge sum of connected simplicial complexes 
 is independent of base points and is homotopy invariant.
 \begin{enumerate}
  \item If $X\simeq_w K$, then $\ordcpx{X}\simeq_w X \simeq_w K$ whence $\ordcpx{X}\simeq K$. 
	Therefore $\bS X \simeq_w \susp\ordcpx{X} \simeq \susp K$. 
  \item If $X\simeq_w Y$, then $\susp\ordcpx{X}\simeq \susp\ordcpx{Y}$ whence $\bS X \simeq_w \bS Y$.
  \item This is because $X\vee Y\simeq_w \ordcpx{X}\vee \ordcpx{Y}$ and the homotopy type of the right hand side is 
	independent of base points.
	% Suppose $X$ and $Y$ are connected, $x_0,x_1\in X$, and $y_0,y_1\in Y$. 
	% Since the homotopy type of wedge sum of connected simplicial complexes are independent of base points,	% we have
	% \begin{align*}
	% \left(X, x_0\right) \vee \left(Y, y_0\right) 
	%  &\simeq_w
	% \left(\ordcpx{X}, x_0\right) \vee \left(\ordcpx{Y}, y_0\right)  \\
	%  &\simeq 
	% \left(\ordcpx{X}, x_1\right) \vee \left(\ordcpx{Y}, y_1\right) \\ 
	%  &\simeq_w
	% \left(X, x_1\right) \vee \left(Y, y_1\right) .
	% \end{align*}
  \item If $X_i\simeq_w K_i$, then $\ordcpx{X_i}\simeq K_i$. Therefore 
	$X_1\vee X_2 \simeq_w \ordcpx{X_1}\vee \ordcpx{X_2}  \simeq K_1\vee K_2$.
  \item If $X_i\simeq_w Y_i$, then $\ordcpx{X_i}\simeq \ordcpx{Y_i} \simeq_w Y_i$. Therefore 
	$X_1\vee X_2 \simeq_w \ordcpx{X_1}\vee \ordcpx{X_2}  \simeq_w Y_1\vee Y_2$.
  \item  \begin{align*}
	  \bS(X\vee Y) \simeq_w \susp\left(\ordcpx{X}\vee \ordcpx{Y}\right)
	  &\simeq \susp\ordcpx{X}\vee \susp\ordcpx{Y} \\
	 &\simeq \ordcpx{\bS X} \vee \ordcpx{\bS Y}
	  \simeq_w \bS X \vee \bS Y.
	 \end{align*}
 \end{enumerate}
  % In particular, 
 % and 
 % We also note that, since 
 % $\susp\left(\ordcpx{X}\vee \ordcpx{Y}\right)
 % \simeq \susp\ordcpx{X}\vee \susp\ordcpx{Y}
 % \simeq \ordcpx{\bS X} \vee \ordcpx{\bS Y}$, 
 % we have $\bS(X\vee Y)\simeq_w \bS X \vee \bS Y$.
 % $X\simeq_w L \Rightarrow \ordcpx{X}\simeq_w X \simeq_w L \Rightarrow \ordcpx{X}\simeq L$ 
 % $\Rightarrow \susp\ordcpx{X}\simeq \susp L \Rightarrow \bS X \simeq_w \susp L$.
\end{proof}

The following generalizations of beat points are given by Barmak and Minian \cite{MR2448871}.

\begin{defn}
 Let $X$ be a finite $T_0$ space and $x\in X$. 
 We call $x$ a \emph{down weak beat point} if $\widehat{U}_x$ is contractible. 
 Dually, $x$ is an \emph{up weak beat point} if $\widehat{F}_x$ is contractible. 
 We call $x$ a \emph{weak point} if it is either a down or up weak beat point. 

 As remarked in \cite[4.2.3]{Barmak}, 
 a point $x$ is a weak point if and only if $\widehat{C}_x$ is contractible. 

 We call $x$ a \emph{$\gamma$-point} if $\widehat{C}_x$ is homotopically trivial.
\end{defn}

\begin{prop}[Barmak and Minian \cite{MR2448871}]
 Let $X$ be a finite $T_0$ space and $x\in X$ a $\gamma$-point. 
 Then, the inclusion $i\colon X\setminus\{x\} \to X$ is a weak homotopy equivalence.
\end{prop}

A little bit more generally, the following holds:

\begin{prop}
 \label{weakgamma}
 Let $X$ be a finite $T_0$ space. 
 If the inclusion 
 $\geop{\hat{C}_x} \to \geop{X-\{x\}}$ is null homotopic, 
 then 
 \[
  X \simeq_w (X-\{x\}) \vee \bS\hat{C}_x
 \]
 In particular, $X$ splits into smaller spaces.
\end{prop}
\begin{proof}
 Since 
 \begin{align*}
  \ordcpx{X}\setminus \{x\}&=\ordcpx{X-\{x\}},  \\
  \lk\left(x\right)&=\Set{\sigma \subset X-\{x\} | \sigma\ne\emptyset, \;\sigma\cup \{x\} \text{ is a chain}} 
  =\ordcpx{\hat{C}_x}, 
 \end{align*}
 the inclusion $\lk(x) \to    \ordcpx{X}\setminus \{x\}$ is null homotopic by the assumption. 
 Therefore
 \begin{align*}
  \ordcpx{X} &\simeq \left(\ordcpx{X}\setminus \{x\}\right) \vee \susp\left(\lk(x)\right)  \\
  &= \ordcpx{X-\{x\}} \vee \susp \ordcpx{\hat{C}_x}
 \end{align*}
\end{proof}

\begin{coro}
 \label{weakgammamaximal}
 Let $X$ be a finite $T_0$ space. 
 If there exists a maximal element $a\in \mxl(X)$ 
 such that  $X-\{a\}$ is connected and $\hat{U}_a\simeq_w \bigvee_n S^0$, 
 then  
 \[
  X \simeq_w (X-\{a\}) \vee \bigvee_n S^1.
 \]
 In particular, $X$ splits into smaller spaces.
\end{coro}
\begin{proof}
 Since $a$ is maximal, $\widehat{C}_a=\widehat{U}_a$, whence
 $\geop{\widehat{C}_a}=\geop{\widehat{U}_a}\simeq \bigvee_n S^0$. 
 Since $\geop{X-\{a\}}$ is connected, the inclusion 
 $\geop{\widehat{C}_a} \to \geop{X-\{a\}}$ is null homotopic.
\end{proof}

\section{Poset splitting} 
\label{sec-posetsplitting}

% We do not need the finiteness assumption in most of this section.

In this section, we reformulate the poset splitting of Cianci-Ottina 
and show our fundamental splitting results \cref{suspensionposetlemma,UaFbmxlmnlsplit,Uamxlmnlsplit}. 

\begin{defn}
 Given a poset $X$, the poset
 \[
  \Sd{X}\coloneqq  \fpst{\ordcpx{X}}
 \]
 is called the \emph{barycentric subdivision} of $X$. 
 $\Sd{X}$ is the set of nonempty finite chains of $X$ ordered by the inclusion order.

 $\Sd{X}$ is weak homotopy equivalent to $X$.
\end{defn}

Clearly the following holds:
\begin{lem}
 Let $X$ be a poset and $A_i\subset X$.
 \begin{enumerate}
  \item \begin{enumerate}
	 \item $\bigcap_i \ordcpx{A_i}=\ordcpx{\bigcap_i A_i}$.
	 \item $\bigcup_i \ordcpx{A_i}\subset \ordcpx{\bigcup_i A_i}$.

	       If every $A_i$ is a down set or every $A_i$ is an up set, then 
	       $\bigcup_i \ordcpx{A_i} = \ordcpx{\bigcup_i A_i}$.
	\end{enumerate}
  \item \begin{enumerate}
	 \item $\Sd{A_i}$ is a down set of $\Sd{X}$. 
	 \item $\bigcap_i \Sd{A_i} = \Sd{\left(\bigcap_i A_i\right)} $.
	 \item $\bigcup_i \Sd{A_i} \subset \Sd{\left(\bigcup_i A_i\right)}$.
	       
	       If every $A_i$ is a down set or every $A_i$ is an up set, then 
	       $\bigcup_i \Sd{A_i} = \Sd{\left(\bigcup_i A_i\right)}$.
	\end{enumerate}
 \end{enumerate}
\end{lem}

In general, the inclusion 
\[
 \bigcup_i \Sd{A_i} \subset \Sd{\left(\bigcup_i A_i\right)}
\]
is not a weak homotopy equivalence. 
To remedy this, Cianci-Ottina \cite{CianciOttina} used open stars
and the second barycentric subdivisions.
% and introduced a technique which they call the poset splitting. 

\begin{defn}
 \begin{enumerate}
  \item  Let $K$ be a simplicial complex and $V$ its set of vertices.
	 \begin{enumerate}
	  \item For a vertex $v\in V$, the subset 
		\[
		\ost_K(v)\coloneqq \Set{\sigma\in K | v\in \sigma}
		\]
		of $K$ is called the \emph{open star} of $v$.
	  \item For a subset $A\subset V$, 
	  \begin{align*}
	 \ost_K(A)&\coloneqq \bigcup_{v\in A}\ost_K(v) 
	   % \\
	   % &=\Set{\sigma\in K | \sigma\cap A \ne \emptyset} 
	  \end{align*}
		is called the \emph{open star} of $A$.
		
		Clearly, we have
	 \begin{align*}
	  \ost_K\left(\bigcup_iA_i\right)&=\bigcup_i\ost_K(A_i), \\
	  \ost_K(A)&=\Set{\sigma\in K | \sigma\cap A \ne \emptyset} .
	 \end{align*}
	 \end{enumerate}
  \item  Let $X$ be a poset and $A\subset X$ a subset. 
	 We define a subposet $\exch{A}\subset \Sd{X}$ by 
	 \[
	  \exch{A}\coloneqq \fpst{\ost_{\ordcpx{X}}(A)}
	  =\Set{ \sigma\in \Sd{X} | \sigma \cap A \ne \emptyset} \subset \Sd{X}.
	 \]
 \end{enumerate}
\end{defn}

\begin{rem}
We have $\exch{A}=\Sd{X}-\Sd{(X-A)}$.
 % For $\sigma\in \Sd{X}$, 
 % \[
 %  \sigma\not\in \exch{A} \Leftrightarrow \sigma \cap A = \emptyset 
 % \Leftrightarrow \sigma \subset X-A 
 % \Leftrightarrow \sigma \in \Sd{(X-A)}
 % \]
\end{rem}

\begin{lem}
 \label{whtpofexch}
  \begin{enumerate}
  \item $\exch{A}$ is an up set of $\Sd{X}$ and $\Sd{A} \subset \exch{A}$.
  \item The inclusion $i\colon \Sd{A}\to \exch{A}$ 
	and the map $r\colon \exch{A} \to \Sd{A}$ given by 
	$r(\sigma)=\sigma\cap A$ are mutually inverse homotopy equivalences.
  \item The map $\max\colon \Sd{A} \to A$ is a weak homotopy equivalence.
  \item If $A$ is an up set, then $\max \sigma = \max\left(\sigma\cap A\right)\in A$ 
	for any $\sigma\in \exch{A}$,  
	and $\max\colon \exch{A} \to A$ is a weak homotopy equivalence. 
	\label{whtpofexch-up}
  \item If $A$ is a down set, then $\min \sigma =\min\left(\sigma\cap A\right) \in A$ 
	for any $\sigma\in \exch{A}$,  
	and $\min\colon \exch{A} \to A^o$ is a weak homotopy equivalence.
	\label{whtpofexch-down}
 \end{enumerate}
\end{lem}
\begin{proof}
 \begin{enumerate}
  \item If $\sigma\cap A\ne\emptyset$ and $\sigma\subset \tau$, then $\tau\cap A\ne\emptyset$. 

	If $\emptyset\ne \sigma\subset A$, then $\sigma\cap A\ne\emptyset$.
  \item This is essentially shown in \cite[II.9]{MR0050886}. 
	If $\sigma\in \exch{A}$, then $\sigma\cap A \ne \emptyset$ hence $\sigma\cap A\in \Sd{A}$. 
	Clearly, the map $r$ is monotone and $ri(\tau)=\tau$. 
	We have $ir(\sigma)\leq \sigma$ for $\sigma\in \exch{A}$. Therefore $ir\simeq 1_{\exch{A}}$. 
  \item Well known. See \cite{MR1240364} for example. 
	% (cf. \cref{intersectioninterval}). 
       % 	For an element $a\in A$, we have $\max^{-1}(U_a)\ne \emptyset$ 
       % 	since $\{a\}\in \Sd{A}$ and $\max(\{a\})=a$. 

       % 	If $\sigma\in \max^{-1}(U_a)$, namely, if $\max \sigma \leq a$, 
       % 	then we see that $\sigma\cup \{a\}\in \max^{-1}(U_a)$ 
       % 	and $\sigma\leq \sigma\cup \{a\}\geq \{a\}$, 
       % 	and hence $\max^{-1}(U_a)\simeq *$. 
	
       % Therefore, by Quillen's Theorem A, the map $\max$ is a weak homotopy equivalence.

       % 	I learned this from \cite{45543}, where Strickland showed that $\ordcpx{\max}$ is homotopic to 
       % 	the standard homeomorphism given by barycentric subdivision.
  \item Suppose $A$ is an up set and $\sigma\in \exch{A}$.
	Since $\sigma$ is a finite chain and $\emptyset\ne\sigma\cap A \subset \sigma$, 
	we have $\max\sigma\geq \max (\sigma\cap A) \in A$, 
	and since $A$ is an up set, we have $\max \sigma \in A$. 
	Therefore $\max \sigma \in \sigma \cap A$ and 
	$\max \sigma \leq \max(\sigma\cap A)$. 

	Since $\max \circ i = \max \colon \Sd{A} \xrightarrow{i} \exch{A} \xrightarrow{\max} A$, 
	% Since the following diagram is commutative, 
	we see that $\max\colon \exch{A} \to A$ is a weak homotopy equivalence.
	% \[
	%  \xymatrix{
	% \Sd{A} \ar[rd]_-{\max}^-{\simeq_w} \ar[rr]^-{i}_-{\simeq} & & \exch{A} \ar[ld]^-{\max} \\
	% & A &
	% }
	% \]
  \item This is the dual of part 4.
 \end{enumerate}
\end{proof}

\begin{coro}
 Let $X$ be a poset, $A_i\subset X$ and $A=\bigcup_i A_i$.

 We have 
\begin{align*}
 \exch{A}&=\bigcup_i \exch{A_i},  &
 A &\simeq_w  %\geop{\exch{A}} = 
 \bigcup_i \geop{\exch{A_i}}. 
 \intertext{In particular, if $X=\bigcup_i A_i$, we have}
 \Sd{X}&=\bigcup_i \exch{A_i}, &
 X &\simeq_w %\geop{\Sd{X}} =
 \bigcup_i \geop{\exch{A_i}}. 
\end{align*}
\end{coro}
\begin{proof}
 \begin{align*}
  \bigcup_i \exch{A_i} 
  &=\bigcup_i \fpst{\ost_{\ordcpx{X}}(A_i) } \\
  &=\fpst{\bigcup_i \ost_{\ordcpx{X}}(A_i) } \\
  &=\fpst{\ost_{\ordcpx{X}}\left(\bigcup_iA_i\right)}  
  =\fpst{\ost_{\ordcpx{X}}\left(A\right) } 
  =\exch{A}. 
 \end{align*}
 Since $\exch{A_i}$ is an up set, we have
 \[
 A
 \simeq_w \exch{A} 
 \simeq_w  \geop{\exch{A}} 
 = \geop{\bigcup_i \exch{A_i}} 
 = \bigcup_i \geop{\exch{A_i}}.
 \]
\end{proof}

Cianci-Ottina called the following the poset splitting technique:
Let $X$ be a finite $T_0$ space, $A_1$ and $A_2$ be subspaces of $X$. 
Then, $A_1\cup A_2\simeq_w \ordcpx{\exch{A_1}} \cup \ordcpx{\exch{A_2}}$ and 
$\ordcpx{\exch{A_i}}\simeq_w A_i$, 
whence one may obtain the information of the weak homotopy type of $A_1\cup A_2$ 
from those of $A_1$ and $A_2$.

Moreover, since 
\[
 \ordcpx{\exch{A_1}} \cap  \ordcpx{\exch{A_2}} 
 =  \ordcpx{\exch{A_1}\cap \exch{A_2}} \simeq_w \exch{A_1}\cap \exch{A_2}, 
\]
the information on the weak homotopy type of $\exch{A_1}\cap \exch{A_2}$ 
and the inclusion $\exch{A_1}\cap \exch{A_2} \to \exch{A_i}$ would give us more information.

\begin{ex}
 \label{unionhtptrivsusp}
 Let $X=A\cup B$ be a finite $T_0$ space. 
 If $A$ and $B$ are homotopically trivial, then we have
 \[
 X \simeq_w \bS(\exch{A}\cap \exch{B})
 \]
\end{ex}
\begin{proof}
 Since $A$ and $B$ are homotopically trivial, 
 $\ordcpx{\exch{A}}\simeq  \ordcpx{A}\simeq *$, $\ordcpx{\exch{B}}\simeq *$. 
 Therefore
 \begin{align*}
  X&\simeq_w \ordcpx{\exch{A}} \cup \ordcpx{\exch{B}} \\
  &\simeq \susp\left(\ordcpx{\exch{A}} \cap  \ordcpx{\exch{B}}\right) \\
  &= \susp\left(\ordcpx{\exch{A} \cap \exch{B}}\right) \\
   &\simeq_w \bS(\exch{A}\cap \exch{B}).
 \end{align*}
\end{proof}

\begin{lem}
 \label{intersectioncomparable}
 Let $X$ be a poset and $A$ and $B$ nonempty subsets. 
 Then, 
 $\exch{A}\cap \exch{B}\ne\emptyset$ if and only if 
 there exist an element in $A$ and an element in $B$ which are comparable.
\end{lem}
\begin{proof}
 % Clear. 
 If $\exch{A}\cap \exch{B}\ne\emptyset$, pick an element $\sigma\in \exch{A}\cap \exch{B}$. 
 Since $\sigma \in \exch{A}$, 
 we have $\sigma\cap A \ne \emptyset$, that is, there exists an element $a\in A$ such that $a\in \sigma$. 
 Similarly, there exists an element $b\in B$ such that $b\in \sigma$. 
 Since $\sigma$ is a chain, $a$ and $b$ are comparable. 

 If $a\in A$ and $b\in B$ are comparable, then 
 $\{a,b\}$ is a chain, $\{a,b\}\cap A\ne\emptyset$ and $\{a,b\}\cap B\ne \emptyset$. 
 Hence  $\{a,b\}\in \exch{A} \cap \exch{B}$.
\end{proof}

\begin{defn}
 Let $X$ be a poset and $A,B\subset X$. 

 We say that $A$ and $B$ are \emph{comparable} if 
 there exist an element in $A$ and an element in $B$ which are comparable.

 Otherwise, that is, if any element of $A$ and any element of $B$ are incomparable, 
 we say that $A$ and $B$ are \emph{incomparable}.
\end{defn}

\begin{lem}
 \label{comparedownsetandb}
 Let $A\subset X$ be a  down set and $b\in X$. Then 
 $A$ and $\{b\}$ are comparable if and only if $A \cap U_b\ne \emptyset$.
\end{lem}
\begin{proof}
 % Clear.

 If $A\cap U_b\ne\emptyset$, then clearly $A$ and $\{b\}$ are comparable. 

 If $A$ and $\{b\}$ are comparable,  there exists $a\in A$ such that $a$ and $b$ are comparable. 
 If $a\leq b$, then $a\in A \cap U_b$. 
 If $b\leq a$, since $A$ is a down set, $b\in A$ and so $b\in A \cap U_b$. 
 In both cases, $A\cap U_b\ne \emptyset$.  
\end{proof}

\begin{lem}
 \label{intersectiononepoint}
 Let $X$ be a poset.
 \begin{enumerate}
  \item Let $A\subset X$ be a down set and $b\in X$. 
	
	There exists a homotopy equivalence 
	$\xymatrix@1{q\colon\exch{A}\cap \exch{\{b\}} \ar[r]^-{\simeq} & \Sd{(A\cap U_b)}}$
	which makes the following diagram homotopy commutative:
	\[
	 \xymatrix{
	\exch{A}\cap \exch{\{b\}} \ar[r]^-{q}_-{\simeq} \ar@{}[d]|-*{\bigcap} 
	& \Sd{(A\cap U_b)} \ar@{}[d]|-*{\bigcap} \ar[r]^-{\max}_-{\simeq_w}
	& A\cap U_b \ar@{}[d]|-*{\bigcap} \\
	\exch{A} \ar[r]_-{r}^-{\simeq} & \Sd{A} \ar[r]_-{\max}^-{\simeq_w} & A
	}
	\]
  \item Let $a\in X$ and $B\subset X$ an up set. 
	
	There exists a homotopy equivalence 
	$\xymatrix@1{q\colon\exch{\{a\}}\cap \exch{B} \ar[r]^-{\simeq} & \Sd{(F_a\cap B)}}$
	which makes the following diagram homotopy commutative:
	\[
	 \xymatrix{
	\exch{\{a\}}\cap \exch{B} \ar[r]^-{q}_-{\simeq} \ar@{}[d]|-*{\bigcap} 
	& \Sd{(F_a\cap B)} \ar@{}[d]|-*{\bigcap} \ar[r]^-{\max}_-{\simeq_w}
	& F_a\cap B \ar@{}[d]|-*{\bigcap} \\
	\exch{B} \ar[r]_-{r}^-{\simeq} & \Sd{B} \ar[r]_-{\max}^-{\simeq_w} & B
	}
	\]
\end{enumerate}
\end{lem}
\begin{proof}
 We show part 1. Part 2 is the dual.

 Let $A$ be a down set and 
 $\sigma\in \exch{A}\cap \exch{\{b\}}$. 
 As noted in \cref{whtpofexch}, $\min\sigma\in A$ and since 
 $b\in \sigma$, $\min\sigma \leq b$.
 Hence $\min \sigma \in A\cap U_b$ and so $\sigma\cap A\cap U_b\ne\emptyset$. 
 We define a map $q\colon \exch{A}\cap \exch{\{b\}}\to \Sd{(A\cap U_b)}$ 
 by $q(\sigma)=\sigma\cap A\cap U_b$. 

 For $\tau\in \Sd{(A\cap U_b)}$, clearly we have $\tau\cup \{b\}\in \exch{A}\cap \exch{\{b\}}$. 
 We define a map $i_b\colon \Sd{(A\cap U_b)} \to \exch{A}\cap \exch{\{b\}}$ 
 by $i_b(\tau)=\tau \cup \{b\}$. 

 Clearly, $q$ and $i_b$ are order preserving.

 For $\sigma\in \exch{A}\cap \exch{\{b\}}$, 
 because we have $b\in \sigma$ and $\sigma\cap A\cap U_b\subset \sigma$, 
 \[
 i_bq(\sigma)=\left(\sigma\cap A \cap U_b\right)\cup \{b\}\leq \sigma.
 \]
 On the other hand, 
 for $\tau\in \Sd{(A\cap U_b)}$,  since $\tau\subset A\cap U_b$, we see that 
 \[
 qi_b(\tau)=\left(\tau \cup \{b\}\right) \cap (A\cap U_b )
 % =\tau \cup \left(\{b\}\cap A\right) 
 \geq \tau.
 \]
 Therefore, $q$ and $i_b$ are mutually inverse homotopy equivalences.

 For $\sigma\in \exch{A}\cap \exch{\{b\}}$, we have 
 \[
 q(\sigma)=\sigma\cap A \cap U_b \leq \sigma \cap A = r(\sigma)
 \]
 whence the diagram is homotopy commutative.
\end{proof}

The next proposition and \cref{UaFbmxlmnlsplit,Uamxlmnlsplit} are our poset splitting results.
\begin{prop}
 \label{suspensionposetlemma}
 Let $X=A\cup B$ be a connected finite $T_0$ space. 
 Assume that  $A=\coprod_{i=1}^l A_i$, $B= \coprod_{i=1}^m B_i$, 
 $A_i\simeq_w *$, $B_i \simeq_w *$,  
 and if $i\ne j$, $A_i$ and $A_j$, $B_i$ and $B_j$ are incomparable.
 Then
 \[
  X\simeq_w \left(\bigvee_{\substack{i,j \\ \text{$A_i$ and $B_j$ are comparable}}} 
 \bS\left(\exch{A_i}\cap \exch{B_j}\right) \right)
\vee \left(\bigvee S^1\right)
 \]
 and $\bS\left(\exch{A_i}\cap \exch{B_j}\right) \simeq_w  A_i\cup B_j $.
\end{prop}
\begin{proof}
 Since $X=A\cup B$, we have 
 $X\simeq_w \ordcpx{\exch{A}} \cup \ordcpx{\exch{B}}$. 

 We set 
 \begin{align*}
  K&=\ordcpx{\exch{A}} \cup \ordcpx{\exch{B}},  \\
  L&=\ordcpx{\exch{A}}, & M &= \ordcpx{\exch{B}}, \\ 
  L_i&=\ordcpx{\exch{A_i}}, & M_i &= \ordcpx{\exch{B_i}}. 
  % K&=L\cup M & & 
 \end{align*}

 If $i\ne j$, $A_i$ and $A_j$ are incomparable whence 
 $\exch{A_i}\cap \exch{A_j}=\emptyset$ by \cref{intersectioncomparable}. 
 Therefore $\ordcpx{\exch{A_i}}\cap \ordcpx{\exch{A_j}}=\emptyset$ and 
 \[
  L=\ordcpx{\exch{A}}=\bigcup_i\ordcpx{\exch{A_i}}=\coprod_i\ordcpx{\exch{A_i}} =\coprod_i L_i .
 \]
 Similarly, we see that $M=\coprod_i M_i$. 
 By the assumption, we have $L_i\simeq *$, $M_i\simeq *$. 

 Therefore by \cref{suspensionlemma}, 
 $K$ is homotopy equivalent to the wedge sum of 
 $\susp{(L_i\cap M_j)}$ for those $L_i\cap M_j\ne\emptyset$ and 
 some copies of $S^1$.
 We see that 
 \begin{align*}
  L_i\cap M_j \ne \emptyset 
  &\Leftrightarrow \ordcpx{\exch{A_i}} \cap \ordcpx{\exch{B_j}} \ne\emptyset \\
  &\Leftrightarrow \exch{A_i}\cap \exch{B_j} \ne\emptyset \\ 
  &\Leftrightarrow \text{$A_i$ and $B_j$ are comparable}.
 \end{align*}
 Finally,  by \cref{unionhtptrivsusp}, 
 $\susp{(L_i\cap M_j)} \simeq_w \bS\left(\exch{A_i}\cap \exch{B_j}\right) \simeq_w  A_i\cup B_j $.
\end{proof}

% We do not use this.
% \begin{coro}
%  If $(l,m)\ne (1,1)$ in \cref{suspensionposetlemma}, 
%  then $A_i\cup B_j \subsetneq X$ and 
%  $X$ splits into smaller spaces. 
% \end{coro}

A crucial observation of Cianci-Ottina \cite{CianciOttina} is that 
many small finite spaces can be decomposed into 
the form $U_a \cup F_b \cup \mxl(X) \cup \mnl(X)$. 
We show that this decomposition splits $X$ into wedge sum of suspensions.

\begin{coro}
 \label{UaFbmxlmnlsplit}
 Let $X$ be a connected finite $T_0$ space 
 and assume that there exist $a,b\in X$ such that 
 \[
  X=U_a \cup F_b \cup \mxl(X) \cup \mnl(X).
 \]
 We put 
 \begin{align*}
  A&=\mnl(X)-U_a - \{b\}, &
  A_{\sim b}&=\Set{x\in A | F_b\cap F_x \ne \emptyset}, \\
  B&=\mxl(X)-F_b - \{a\}, &
  B_{\sim a}&=\Set{x\in B | U_a \cap U_x \ne \emptyset} .
 \end{align*}

 If $U_a$ and $F_b$ are comparable, 
 then 
 \[
  X \simeq_w \bS(\exch{U_a}\cap \exch{F_b}) \vee \left(\bigvee_{x\in B_{\sim a}} \bS(U_a\cap U_{x})\right) 
 \vee \left(\bigvee_{x\in A_{\sim b}} \bS(F_b\cap F_{x})\right) \vee \left(\bigvee_k S^1\right) 
 \]
 for some $k\geq 0$, and $\bS(\exch{U_a}\cap \exch{F_b})\simeq_w U_a \cup F_b$.

 If $U_a$ and $F_b$ are incomparable, 
 \[
  X\simeq_w \left(\bigvee_{x\in B_{\sim a}} \bS(U_a\cap U_{x})\right) 
 \vee \left(\bigvee_{x\in A_{\sim b}} \bS(F_b\cap F_{x})\right) \vee \left(\bigvee_k S^1\right)
 \]
 for some $k\geq 0$.

 % \begin{align*}
 %  \bS(\exch{U_a}\cap \exch{F_b})&\simeq_w U_a \cup F_b, &
 %  \bS(U_a\cap U_x)&\simeq_w U_a \cup U_x,  &
 %  \bS(F_b\cap F_x)&\simeq_w F_b \cup F_x . 
 % \end{align*}

 In particular, if $U_a$ and $F_b$ are incomparable or 
 $U_a\cup F_b\subsetneq X$ or $U_a\cup F_b$ splits into smaller spaces, 
 then $X$ splits into smaller spaces.
\end{coro}
\begin{proof}
 Apply \cref{suspensionposetlemma} to 
 $U_a \cup A=U_a\amalg \coprod_{x\in A} \{x\}$ and $F_b\cup B=F_b\amalg \coprod_{x\in B}\{x\}$.
 
 For the reader's convenience, we record the details.

 Clearly, $X=U_a\cup A \cup F_b \cup B$. 

 Since $A\subset \mnl(X)$, different elements in $A$ are incomparable. 
 Let $x\in A$. Since $x$ is a minimal element and $x\not\in U_a$,  
 $U_a\cap U_x=U_a\cap \{x\}=\emptyset$. Hence by \cref{comparedownsetandb}, $\{x\}$ and 
 $U_a$ are incomparable.
 Moreover, $U_a$ is contractible. 

 Similarly, different elements of $B$ are incomparable, 
 $F_b$ and elements of $B$ are incomparable and $F_b\simeq *$. 

 Therefore, we can apply \cref{suspensionposetlemma}.

  By \cref{comparedownsetandb}, 
 $U_a$ and $x\in B$ is comparable if and only if $U_a\cap U_x\ne \emptyset$, namely, $x\in B_{\sim a}$. 
 In this case, we have 
 $\exch{U_a}\cap \exch{\{x\}} \simeq_w U_a\cap U_x$ by \cref{intersectiononepoint} 
 and either $U_a\cap U_x\simeq *$ or $\card{U_a\cap U_x}<\card{X}$ by \cref{elementaryfactsonUaFb}.

 Similarly, 
 $x\in A$ and $F_b$ are comparable if and only if $x\in A_{\sim b}$ and in this case, 
 $\exch{F_b}\cap \exch{\{x\}}\simeq_w F_b \cap F_x$ 
 and either $F_b\cap F_x\simeq *$ or $\card{F_b\cap F_x}<\card{X}$.

 If $x\in A$ and $y\in B$ are comparable, then $x\leq y$ and  
 $\{x,y\}= \min\left(\exch{\{x\}}\cap \exch{\{y\}}\right)$, hence  
 $\exch{\{x\}}\cap \exch{\{y\}} \simeq *$.
\end{proof}

\begin{coro}
 \label{Uamxlmnlsplit}
 Let $X$ be a connected finite $T_0$ space. 
 If there exists $a\in X$ such that 
 \[
    X=U_a\cup \mxl(X)\cup \mnl(X)
 \]
 then 
  \[
  X\simeq_w \left(\bigvee_{\substack{b\in \mxl(X) \\ U_a\cap U_b\ne\emptyset} } \bS(U_a\cap U_b) \right) 
 \vee \left(\bigvee_k S^1 \right)
 \]
 for some $k\geq 0$. 

  In particular, $X$ splits into smaller spaces.
\end{coro}

\section{Weak homotopy types of posets of intervals}
\label{sec-interval}

By \cref{UaFbmxlmnlsplit}, 
if $X=U_a \cup F_b \cup \mxl(X) \cup \mnl(X)$ 
and $U_a\cup F_b\subsetneq X$, then $X$ splits into smaller spaces. 
However, in the case where $X=U_a\cup F_b$, we need to study 
the weak homotopy type of $U_a\cup F_b\simeq_w\bS(\exch{U_a}\cap \exch{F_b})$. 
In this section, we analyse the weak homotopy type of $\exch{U_a}\cap \exch{F_b}$ 
using the poset of intervals.

\begin{defn}
Let $X$ be a poset and $A,B\subset X$. 
Define a subposet $I(A,B)$ of $X^o\times X$ by 
\[
 I^X(A,B) \coloneqq \Set{(a,b) \in A\times B | a\leq b} \subset X^o\times X.
\]
 We often omit the superscript $X$ and write $I(A,B)$ instead of $I^X(A,B)$.

The poset $I(A,B)$ is the set of closed intervals in $X$ 
whose end points are in $A$ and $B$, ordered by inclusion. 

Note that $I^X(A,B) \cong I^{X^o}(B,A)$ as posets.
\end{defn}

\begin{lem}
 \label{intersectioninterval}
Let $X$ be a poset and $A,B\subset X$. 
 \begin{enumerate}
  \item If $A$ is a down set or $B$ is an up set, then 
	there exists a weak homotopy equivalence 
	$\xymatrix@1{e\colon \exch{A}\cap \exch{B} \ar[r]^-{\simeq_w} & I(A,B)}$. 

	If $A$ is a down set, then the following left diagram is commutative, 
	and if $B$ is an up set, then the following right diagram is commutative,  
	where the maps $p_A$ and $p_B$ are projections.
	% which makes the following diagram homotopy commutative:
	\[
	 \xymatrix{
	\exch{A} \ar[d]_-{\min}^-{\simeq_w} 
	& \exch{A}\cap \exch{B} \ar@{}[l]|-*{\supset} 
	\ar[d]^-{e}_-{\simeq_w}
	& \exch{A}\cap \exch{B} \ar@{}[r]|-*{\subset} 
	\ar[d]^-{e}_-{\simeq_w}
	& \exch{B} \ar[d]^-{\max}_-{\simeq_w} \\
	A^o & 	 I(A,B)  \ar[l]^-{p_A}  & 	 I(A,B)   \ar[r]_-{p_B}  & B
	}
	\]
  \item  If both $A$ and $B$ are up sets, 
	 the projection $p_B$ gives a weak homotopy equivalence $p_B\colon I(A,B) \to A\cap B$ 
	 and the following diagram is commutative.
	 \begin{align*}
	  \exch{A}\cap \exch{B}&\simeq_w I(A,B) \simeq_w A\cap B 
	 \end{align*}
	 \[
	 \xymatrix{
	 \exch{A}\cap \exch{B} \ar[rd]_-{\max } \ar[rr]^-{e} & 
	 & I(A,B)  \ar[ld]^-{p_B} \\
	 & A\cap B &
	 }
	 \]
  \item  If both $A$ and $B$ are down sets, 
	 the projection $p_A$ gives a weak homotopy equivalence 
	 $p_A\colon I(A,B) \to \left(A\cap B\right)^o$
	 and the following diagram is commutative.
	 \begin{align*}
	  \exch{A}\cap \exch{B}&\simeq_w I(A,B) \simeq_w A\cap B 
	 \end{align*}
	 \[
	 \xymatrix{
	 \exch{A}\cap \exch{B} \ar[rd]_-{\min } \ar[rr]^-{e} & 
	 & I(A,B)  \ar[ld]^-{p_A} \\
	 & \left(A\cap B\right)^o &
	 }
	 \]
 \end{enumerate}
\end{lem}

\begin{proof}
\begin{enumerate}
 \item We consider the case where $A$ is a down set. 

       Let $\sigma\in \exch{A}\cap \exch{B}$. 
       As we noted in \cref{whtpofexch}, 
       $\min\sigma\in A$. 
       Since $\sigma\cap B\ne\emptyset$ and $\sigma\cap B$ is a finite chain, 
       there exists the maximum element $\max(\sigma\cap B)\in B$. 
       % $\min\sigma\leq \max(\sigma\cap B)$
       % $(\min\sigma, \max(\sigma\cap B))\in I(A,B)$. 
       We define a map 
       $e\colon \exch{A}\cap \exch{B} \to I(A,B)$ 
       by $e(\sigma)=(\min\sigma, \max(\sigma\cap B))$. 
       We have $p_Ae(\sigma)=\min \sigma$, that is, the left diagram is commutative.
       
       If $\sigma\subset \tau$, then 
       \[
	\min\tau \leq \min\sigma \leq \max(\sigma \cap B) \leq \max(\tau\cap B),
       \]
       hence $e$ is monotone.
       % We have
       % \begin{gather*}
       % 	p_Ae(\sigma)=\min \sigma, \\
       % 	p_Be(\sigma)=\max(\sigma\cap B)\leq \max \sigma, 
       % \end{gather*}
       % so the diagram is homotopy commutative (the left square is commutative).
       % $\max\sigma$ may not be in $B$!

       Suppose $(a,b)\in I(A,B)$, namely, $a\in A$, $b\in B$, and $a\leq b$. 

       Clearly, $\{a,b\}\in \exch{A}\cap \exch{B}$ and $e(\{a,b\})=(a,b)$, 
       hence $e^{-1}\left(\los[I(A,B)]{(a,b)}\right)\ne \emptyset$. 

       If $\sigma\in \exch{A}\cap \exch{B}$ and $e(\sigma)\leq (a,b)$, 
       that is, 
       \[
	a\leq \min \sigma \leq \max(\sigma\cap B) \leq b, 
       \]
       then we easily see that
       \[
	\sigma\cup \{a\},\; \left(\sigma\cap B\right)\cup \{a\}, \;
       \left(\sigma\cap B\right)\cup \{a,b\},\; \{a,b\} \in 
       e^{-1}\left(\los[I(A,B)]{(a,b)}\right), 
       \]
       and
       \[
	\sigma\leq \sigma\cup \{a\} \geq \left(\sigma\cap B\right)\cup \{a\} 
       \leq \left(\sigma\cap B\right)\cup \{a,b\} \geq \{a,b\},
       \]
       and hence $e^{-1}\left(\los[I(A,B)]{(a,b)}\right)\simeq *$. 
       
       Therefore, by Quillen's Theorem A, $e$ is a weak homotopy equivalence.
       %
       % \begin{itemize}
       % 	\item Clearly $\sigma\cup \{a\}$ is a chain and 
       % 	      $\sigma\cup \{a\}\in \exch{A}\cap \exch{B}$. 
	      
       % 	      $e(\sigma\cup \{a\})=(a,\max{(\sigma\cap B)})\leq (a,b)$.
       % 	\item $(\sigma\cap B)\cup \{a\}$ is a chain because 
       % 	      $(\sigma\cap B)\cup \{a\}\subset \sigma \cup \{a\}$. 
       % 	      Since $\sigma\cap B\ne\emptyset$ and $a\in A$, we see that
       % 	      $(\sigma\cap B)\cup \{a\}\in \exch{A}\cap \exch{B}$.

       % 	      Since $e$ is monotone, we have
       % 	      $e\left((\sigma\cap B)\cup \{a\}\right)\leq e(\sigma \cup \{a\})\leq (a,b)$.
       % 	\item $\left(\sigma\cap B\right)\cup \{a,b\}=\left(\left(\sigma\cap B\right)\cup \{a\}\right)\cup \{b\}$ is a chain because  $a\leq b$ and $\max(\sigma\cap B)\leq b$.
	      
       % 	      Clearly, we have 
       % 	      $\left(\sigma\cap B\right)\cup \{a,b\}\in \exch{A}\cap \exch{B}$
       % 	      and 
       % 	      $e\left(\left(\sigma\cap B\right)\cup \{a,b\}\right)=(a,b)$.
       % \end{itemize}
       %
 \item   Since both $A$ and $B$ are up sets, 
	 $\max\sigma \in A\cap B$ for any element $\sigma\in \exch{A}\cap \exch{B}$. 
	 Hence we obtain a map $\max\colon \exch{A}\cap \exch{B} \to A\cap B$

	 If $(a,b)\in I(A,B)$, then $a\leq b$, and since $A$ is an up set, 
	 $b\in A\cap B$, namely, $p_B((a,b))\in A\cap B$.
	 Hence the projection gives a map $p_B\colon I(A,B) \to A\cap B$.

	 In this ($B$ being an up set) case, the map $e\colon \exch{A}\cap \exch{B} \to I(A,B)$ 
	 defined by $e(\sigma)=(\min(\sigma \cap A), \max\sigma)$ gives a weak homotopy equivalence, 
	 and we have $p_Be = \max$.

	 Suppose $b\in A\cap B$.  

	 Since $(b,b)\in I(A,B)$ and $p_B((b,b))=b$,  
	 $p_B^{-1}\left(\los[(A\cap B)]{b}\right)\ne \emptyset$.
	 
	 If $(a,b')\in I(A,B)$ and $p_B((a,b'))\leq b$, 
	 then $a\leq b'\leq b$ and 
	 $(a,b')\leq (a,b)\geq (b,b)$. Hence $p_B^{-1}\left(\los[(A\cap B)]{b}\right)\simeq *$. 
	 
	 Therefore $p_B$ is a weak homotopy equivalence.
	 %
	 % Note that the map $\max=p_Be$ is also a weak homotopy equivalence. 
	 % One can directly show that the map $\max$ is a weak homotopy equivalence in a similar way.
 \item This is the dual of part 2.
\end{enumerate}
\end{proof}

\begin{coro}
 \label{unionhtptriv}
 Let $X=A\cup B$ be a finite $T_0$ space. If $A$ and $B$ are homotopically trivial down sets, 
 then we have 
 \[
  X \simeq_w \bS (A\cap B)
 \]
\end{coro}
\begin{proof}
 \[
  X \simeq_w \bS(\exch{A}\cap \exch{B}) \simeq_w \bS(A\cap B)
 \]
 Of course, one can directly show this using Quillen's Theorem A.
\end{proof}

\begin{rem}
 It is well known that 
 $U_a\cup U_b$ is \emph{homotopy equivalent} to $\bS(U_a\cap U_b)$. 
 However, in general, even if both $A$ and $B$ are 
 \emph{contractible} down sets, 
 $A\cup B$ and $\bS(A\cap B)$ may not be homotopy equivalent. 
 For example, consider the space $S^1_3$ in \cref{S12andS13}. 
 $S^1_3=\left(U_{a_0}\cup U_{a_1}\right) \cup U_{a_2}$,  
 and $U_{a_0}\cup U_{a_1}$ and $U_{a_2}$ are contractible. 
 Since $\left(U_{a_0}\cup U_{a_1}\right) \cap U_{a_2}=\{b_0,b_1\}$, 
 $\bS\left(\left(U_{a_0}\cup U_{a_1}\right) \cap U_{a_2}\right)$ is 
 homeomorphic to $S^1_2$. 
 Both $S^1_3$ and $S^1_2$ are minimal and they are not homeomorphic, 
 so they are not homotopy equivalent.
\end{rem}

\begin{coro}
 Let $X$ be a finite $T_0$ space and 
 $a,b\in X$. We have
 \[
  U_a\cup F_b \simeq_w \bS(I(U_a,F_b)).
 \]
\end{coro}
\begin{proof}
 \begin{align*}
  U_a\cup F_b \simeq_w \bS(\exch{U_a}\cap \exch{F_b}) \simeq_w \bS(I(U_a,F_b)).
 \end{align*}
\end{proof}

We consider the height of the interval.

\begin{defn}
 Let $X$ be a finite $T_0$ space.
 
 The \emph{length} $\len{c}$ of a chain $c$ of $X$ is 
 one less than the cardinality of $c$: $\len{c}\coloneqq \card{c}-1$.
 The number
 \[
  \hgt(X)\coloneqq \max\Set{\len{c} | c \text{ is a chain of $X$}}
 \]
 is called the \emph{height} of $X$.
\end{defn}

If $\hgt(X)=1$, then $\geop{X}$ is one dimensional simplicial complex, 
whence each connected component of $X$ is weak homotopy equivalent to 
a wedge of circles and $\bS X$ is weak homotopy equivalent to a wedge of spheres of dimension at most $2$.

% \begin{lem}
% \redcom{Use this?}
%  $\hgt(X\times Y)=\hgt(X)+\hgt(Y)$. 
% \end{lem}

\begin{lem}
 Let $X$ be a finite $T_0$ space and $A,B\subset X$. 
 \begin{enumerate}
  \item  $\hgt(I(A,B))\leq \hgt(A\cup B)$. If $A\cap B=\emptyset$, then $\hgt(I(A,B))< \hgt(A\cup B)$. 
  \item If both $A$ and $B$ are up sets or down sets, then $\hgt(A\cup B)=\max\{\hgt(A), \hgt(B)\}$. 
 \end{enumerate}
\end{lem}
\begin{proof}
 % Clear
\begin{enumerate}
 \item  Suppose $(a_0,b_0)< (a_1,b_1)< \dots < (a_k,b_k)$ is a chain of $I(A,B)$ of length $k$. 
	We have 
	\[
	a_k\leq \dots \leq a_1\leq a_0\leq b_0\leq \dots \leq b_k
	\]
	in $A\cup B$ and 
	for each $1\leq i\leq k$, $a_i < a_{i-1}$ or $b_{i-1}<b_i$.
	Therefore the length of the chain 
	$\{a_k,\dots, a_0,\dots, b_k\}$ of $A\cup B$ 
	is greater than or equal to $k$. If $a_0<b_0$, then it is greater than or equal to $k+1$.
	Hence, $\hgt(I(A,B))\leq \hgt(A\cup B)$ and if $A\cap B=\emptyset$, 
	then $\hgt(I(A,B))< \hgt(A\cup B)$. 
 \item Clearly we have $\max\{\hgt(A), \hgt(B)\}\leq \hgt(A\cup B)$.

       Assume that both $A$ and $B$ are up sets. 
       Let $c$ be a nonempty chain of $A\cup B$. 
       If $\min(c)\in A$, then $c\subset A$ for $A$ is an up set, 
       so $\len{c}\leq \hgt(A)$. Similarly, if $\min(c)\in B$, then $\len{c}\leq \hgt(B)$. 
       Therefore, $\hgt(A\cup B)\leq \max\{\hgt(A), \hgt(B)\}$. 
\end{enumerate}
\end{proof}

\begin{defn}
 Let $X$ be a finite $T_0$ space. We set
 \[
  \cB=X-\mxl(X)-\mnl(X).
 \]
\end{defn}

\begin{coro}
 If $\cB$ is an antichain, then 
 $\hgt\left(I(A,B)\right)\leq 2$ for $A,B\subset X$. 
 If $A\cap B=\emptyset$, then  $\hgt\left(I(A,B)\right)\leq 1$.
\end{coro}

\begin{prop}
 \label{unchainsplit}
 Let $X$ be a finite $T_0$ space such that 
 $\cB$ is an antichain and
 $a_0,b_0\in X$.

 Then, each connected component of $I(U_{a_0},F_{b_0})$
 has the weak homotopy type of a wedge of spheres of dimension at most $2$. 
 Hence, $U_{a_0}\cup F_{b_0}$ has the weak homotopy type of 
 a wedge of spheres of dimension at most $3$. 

 In particular, if $X$ is of the form 
 $X=U_{a_0}\cup F_{b_0}\cup \mxl(X) \cup \mnl(X)$ and $\cB$ is an antichain, 
 then $X$ splits into smaller spaces.
\end{prop}
\begin{proof}
 If $a_0\not\geq b_0$, namely, if 
 $U_{a_0}\cap F_{b_0}=\emptyset$, then 
 $\hgt(I(U_{a_0},F_{b_0}))\leq 1$ and each connected component 
 is weak homotopy equivalent to a  wedge of one dimensional spheres.

 If $a_0=b_0$, then 
 $I(U_{a_0},F_{b_0})=U_{a_0}^o \times F_{b_0} \simeq *$.

 Assume that $a_0> b_0$. 
 
 We see that  $(a_0,a_0),\; (b_0,b_0),\; (b_0,a_0)\in I(U_{a_0},F_{b_0})$ 
 because  $a_0,\; b_0\in U_{a_0}\cap F_{b_0}$. 

 We denote $\mxl(I(U_{a_0},F_{b_0}))$ by $\mxl$ and 
 $\mnl(I(U_{a_0},F_{b_0}))$ by $\mnl$.

 We show that if  $ (a,b) \in I(U_{a_0},F_{b_0}) - \left(\mxl \cup \mnl \right)$, 
 then either $(a,b)\in U_{(b_0,a_0)}$ or $(a,b)$ is a down beat point of $I(U_{a_0},F_{b_0})$.

 Since $(a,b)\not\in \mnl$, $a\ne b$, whence $a<b$. 
 Since $(a,b)\not\in \mxl$, $a\in \cB$ or $b\in \cB$. 
 Because $\cB$ is an antichain, 
 we see that either $a\in \cB$ and $b\in \mxl$, or $a\in \mnl$ and $b\in \cB$.

 Suppose $a\in \cB$ and $b\in \mxl$. 
 Since $\cB$ is an antichain, $a\prec b$. 
 Therefore, elements in $I(X, X)$ smaller than $(a,b)$ are only $(a,a)$ and $(b,b)$.
 Since $(a,b)\not\in \mnl$, at least one of them belongs to $I(U_{a_0},F_{b_0})$. 
 If both $(a,a),\; (b,b)\in I(U_{a_0},F_{b_0})$, 
 then $a\in F_{b_0}$ and $b\in U_{a_0}$, that is, $b_0\leq a$ and $b\leq a_0$, 
 hence $(a,b)\leq (b_0,a_0)$. 
 Otherwise, $(a,b)$ is a down beat point.

 The case where $a\in \mnl$ and $b\in \cB$ is similar. 

 Therefore, by removing these down beat points, we have
 \[
  I(U_{a_0},F_{b_0})\simeq U_{(b_0,a_0)}\cup \mxl \cup \mnl.
 \]

 By \cref{Uamxlmnlsplit}, the connected component of the right hand side 
 containing $U_{(b_0,a_0)}$ is weak homotopy equivalent to 
 a wedge of some copies of $S^1$ and 
 $\bS\left(U_{(b_0,a_0)} \cap U_{(a,b)}\right)$ for some $(a,b)$.
 If $U_{(b_0,a_0)} \cap U_{(a,b)}= U_{(b_0,a_0)}$, then  
 $\bS\left(U_{(b_0,a_0)} \cap U_{(a,b)}\right)$ is contractible. 
 If $U_{(b_0,a_0)} \cap U_{(a,b)}\subsetneq U_{(b_0,a_0)}$, then 
 $U_{(b_0,a_0)} \cap U_{(a,b)}\subset \widehat{U}_{(b_0,a_0)}$ and 
 \[
  \hgt\left(U_{(b_0,a_0)} \cap U_{(a,b)}\right)
 \leq \hgt(\widehat{U}_{(b_0,a_0)}) < \hgt(U_{(b_0,a_0)})\leq \hgt(I(U_{a_0},F_{b_0}))\leq 2,
 \]
 therefore $\geop{U_{(b_0,a_0)} \cap U_{(a,b)}}$ is at most $1$ dimensional 
 and hence $\bS\left(U_{(b_0,a_0)} \cap U_{(a,b)}\right)$
 is weak homotopy equivalent to a wedge of some copies of $S^2$ ad $S^1$. 

 The other connected components has height at most $1$.
\end{proof}

We need more general results.

\begin{lem}
 \label{beatptsinterval}
 Let $X$ be a poset and $A,B\subset X$.
 \begin{enumerate}
  \item \begin{enumerate}
	 \item  If $a_0\in \mxl(A)$ and  $b_0\in \mnl(B\cap F_{a_0})$,  
		then $(a_0,b_0)\in I(A,B)$ is a minimal element of $I(A,B)$.
	 \item  If $a_0\in A$ is an up beat point of $A$ and $b_0\in \mnl(B\cap F_{a_0})$, 
		then $(a_0,b_0)\in I(A,B)$ is a minimal element or a down beat point of $I(A,B)$. 
	\end{enumerate}
  \item \begin{enumerate}
	 \item  If $b_0\in \mnl(B)$ and  $a\in \mxl(A\cap U_{b_0})$, 
		then $(a_0,b_0)\in I(A,B)$ is a minimal element of $I(A,B)$.
	 \item  If $b_0\in B$ is a down beat point of $B$ and $a\in \mxl(A\cap U_{b_0})$, 
		then $(a_0,b_0)\in I(A,B)$ is a minimal element or a down beat point of $I(A,B)$.
	\end{enumerate}
 \end{enumerate}
\end{lem}
\begin{proof}
 We show part 1. Part 2 is the dual.
 \begin{enumerate}[label={\alph*})]
  \item  Suppose $a_0\in \mxl(A)$ and $b_0\in \mnl(B\cap F_{a_0})$. 
	 If $(a,b)\in I(A,B)$ and $(a,b)\leq (a_0,b_0)$, 
	 then $a\in A$, $b\in B$ and $a_0\leq a\leq b\leq b_0$. 
	 Since $a_0\in \mxl(A)$, we have $a_0=a$. 
	 Since $b_0\in \mnl(B\cap F_{a_0})$, we have $b=b_0$.  
	 Therefore, $(a,b)=(a_0,b_0)$ and $(a_0,b_0)$ is minimal.
  \item Suppose $a_0\in A$ is an up beat point of $A$ and $b_0\in \mnl(B\cap F_{a_0})$. 
	We put $\hat{a}_0=\min\widehat{F}_{a_0}^A=\min\left(A\cap \widehat{F}_{a_0}\right) $. 

	Assume that $(a_0,b_0)$ is not minimal in $I(A,B)$. 
	We show that $(\hat{a}_0,b_0)=\max \widehat{U}_{(a_0,b_0)}$.
	If $(a,b)\in \widehat{U}_{(a_0,b_0)}$, namely, 
	if $(a,b)\in I(A,B)$ and $(a,b)<(a_0,b_0)$, then
	$a_0\leq a \leq b\leq b_0$ and $a_0<a$ or $b<b_0$. 
	Since $b\in B \cap F_{a_0}$ and $b_0\in \mnl(B\cap F_{a_0})$, we have $b=b_0$. 
	Therefore $a_0<a$ and so $\hat{a}_0\leq a$. 
	Hence $a_0< \hat{a}_0\leq a \leq b=b_0$ and we have $(a,b)\leq (\hat{a}_0,b_0)<(a_0,b_0)$. 
	Therefore, $(\hat{a}_0,b_0)=\max \widehat{U}_{(a_0,b_0)}$ and 
	$(a_0,b_0)$ is a down beat point.
 \end{enumerate}
 % For $A^o,B^o\subset X^o$
 % \begin{align*}
 %  I(B^o,A^o)&=\Set{(b,a)\in B\times A^o | b\leq^o a} \\
 %  &=\Set{(b,a)\in B\times A^o | a\leq b} \\ 
 %  &\cong \Set{(a,b)\in A^o\times B | a\leq b} \\ 
 %  &=I(A,B) 
 % \end{align*}
\end{proof}

\begin{lem}
 \label{specialinterval}
 Let $X$ be a finite $T_0$ space, 
$a_0\in X - \mnl(X)$, $b_0\in X - \mxl(X)$, 
 and $a_0\not\leq b_0$. 
 We put 
 \begin{align*}
  A_0&=\left(U_{a_0} - \mnl(X)\right)- U_{b_0}, &
  A_1&=\left(U_{a_0} - \mnl(X)\right)\cap U_{b_0}, \\
  B_0&=\left(F_{b_0} - \mxl(X)\right)- F_{a_0}, &
  B_1&=\left(F_{b_0} - \mxl(X)\right)\cap F_{a_0}. 
 \end{align*}
 Suppose the following holds:
 \begin{enumerate}
  \item \begin{enumerate}
	 \item All the elements of $A_0\setminus \{a_0\}$ are up beat points of $U_{a_0}$.
	 \item All the elements of $B_0\setminus \{b_0\}$ are down beat points of $F_{b_0}$. 
	\end{enumerate}
  \item $I(A_0,B_0)=\emptyset$.
 \end{enumerate}
 Moreover, when $A_1\ne\emptyset$ or $B_1\ne \emptyset$, we also assume the following:
 \begin{enumerate}[resume]
  \item \begin{enumerate}
	 \item When $A_1\ne \emptyset$, there exists $\max A_1$. We put $a_1= \max A_1$.
	 \item When $B_1\ne \emptyset$, there exists $\min B_1$. We put $b_1=\min B_1$.
	\end{enumerate}
 \end{enumerate}
 Then we have 
 \[
  I(U_{a_0},F_{b_0})
 \simeq F_{(a_1,b_0)} \cup F_{(a_0,b_1)} \cup \mxl(I(U_{a_0},F_{b_0})) \cup \mnl(I(U_{a_0},F_{b_0})) 
 \]
 % \begin{align*}
 %  I(U_{a_0},F_{b_0})
 %  &=F_{(a_1,b_0)} \cup F_{(a_0,b_1)} \cup \mxl \cup \mnl \cup \{\text{ down beat points }\} \\
 %  &\simeq F_{(a_1,b_0)} \cup F_{(a_0,b_1)} \cup \mxl \cup \mnl 
 % \end{align*}
 where 
 % we denote $\mxl(I(U_{a_0},F_{b_0}))$ by $\mxl$ and 
 % $\mnl(I(U_{a_0},F_{b_0}))$ by $\mnl$, 
 % and
 we consider $F_{(a_1,b_0)}=\emptyset$ when $A_1=\emptyset$ 
 and $F_{(a_0,b_1)}=\emptyset$ when  $B_1=\emptyset$. 

 Moreover, the connected component 
 containing 
 $F_{(a_1,b_0)} \cup F_{(a_0,b_1)}$ is weak homotopy equivalent to 
 \[
 \left(\bigvee_{\substack{(a,b)\in \mnl(I(U_{a_0},F_{b_0})) \\ \left(F_{(a_1,b_0)} \cup F_{(a_0,b_1)}\right)\cap F_{(a,b)}\ne \emptyset }} \bS \left(\left(F_{(a_1,b_0)} \cup F_{(a_0,b_1)}\right)\cap F_{(a,b)}\right) \right) 
 \vee \left(\bigvee S^1\right).
 \]
\end{lem}

\begin{proof}
 We put 
 \begin{align*}
  A&=U_{a_0}, & B&=F_{b_0}, \\
  A_m&=A \cap \mnl(X), &  B_m&=B \cap \mxl(X).
 \end{align*}
 Since
 \begin{align*}
  A&=\left(A-\mnl(X)\right) \cup \left(A \cap \mnl(X) \right)   \\
  &=A_0\cup A_1 \cup A_m,  \\ 
  B&= \left(B-\mxl(X)\right) \cup \left( B \cap \mxl(X)\right)  \\ 
  &=  B_0 \cup B_1 \cup B_m 
  \intertext{and, by the assumption, $I(A_0,B_0)=\emptyset$, we have}
  I(A,B)&=I(A_1,B) \cup I(A,B_1) \\
  &\phantom{=}\cup I(A_0,B_m) \cup I(A_m, B_0) \\
  &\phantom{=}\cup I(A_m,B_m).
 \end{align*}

 We show that $I(A_1,B)\subset F_{(a_1,b_0)}$. 

 % If $A_1=\emptyset$, then $I(A_1,B)=\emptyset$. 
 We suppose $A_1\ne \emptyset$. 
 Since 
 \[
  A_1=\cB\cap U_{a_0}\cap U_{b_0}, \quad a_1=\max A_1, \quad B=F_{b_0},
 \]
 if $a\in A_1$ and $b\in B$, then we have $a\leq a_1 \leq b_0\leq b$, 
 henceforth $I(A_1,B)=A_1^o\times B$ and $(a_1,b_0)=\min I(A_1,B)$. 
 Therefore $I(A_1,B)\subset F_{(a_1,b_0)}$. 

 Similarly or dually, we see that $I(A,B_1)\subset F_{(a_0,b_1)}$.

 %  % Consider $I(A_0, B_m)$.
 We show that 
 \[
 I(A_0,B_m)\subset \mnl(I(A,B))\cup \widehat{F}_{(a_0,b_1)} \cup \{\text{ down beat points of $I(A,B)$ }\}.
 \]
 Suppose $(a,b)\in I(A_0,B_m)$, that is, $a\in A_0$, $b\in B_m$ and $a\leq b$. 
 By the assumption 1 (a), $a$ is either the maximum element, namely, $a_0$ or an up beat point of $A=U_{a_0}$. 
 Hence, by \cref{beatptsinterval}, 
 if $b\in \mnl(B\cap F_a)$, then $(a,b)$ is minimal or a down beat point of $I(A,B)$. 

 If $b\not\in \mnl(B\cap F_a)$, then there exists an element $b'\in B$ such that $a\leq b' < b$.
 Since $b'\not\in \mxl(X)$, we have $b'\in B-\mxl(X)=B_0\cup B_1$, 
 but since $a\in A_0$, $a\leq b'$, and $I(A_0,B_0)=\emptyset$, 
 we see that $b'\not\in B_0$ and so $b'\in B_1$, whence $b_1=\min B_1\leq b'$. 
 Therefore, $a\leq a_0\leq b_1\leq b' <b$ and 
 $(a,b)>(a,b')\geq (a_0,b_1)$.

 Similarly, we see that 
 $I(A_m,B_0)\subset \mnl(I(A,B))\cup \widehat{F}_{(a_1,b_0)} \cup \{\text{ down beat points }\}$ 
 %  % Suppose $a\in A_m$, $b\in B_0$, $b\ne b_0$, $a\leq b$. 
 %  % $b$ is a down beat point of $B$ or $b_0$. 
 %  % If $a\in \mxl(A\cap U_b)$, $(a,b)$ is minimal or down beat point. 
 %  % If $a\not\in \mxl(A\cap U_b)$, $\exists a'\in A$, $a< a' \leq  b$.
 %  % $a'\not\in \mnl(X)\cup \mxl(X)$ and so $a'\in \cB\cap A=A_0\cup A_1$, 
 %  % since $b\in B_0$ and $I(A_0,B_0)=\emptyset$, $a'\not\in A_0$ and $a'\in A_1$. 
 %  % Therefore $a<a'\leq a_1 \leq b_0\leq b $ and 
 %  % $(a,b)>(a',b)\geq (a_1,b_0)$.
 and clearly we have $I(A_m,B_m)=\mxl(I(A,B))$.

 Therefore, by removing down beat points,  we have 
 \[
  I(U_{a_0},F_{b_0})
 \simeq F_{(a_1,b_0)} \cup F_{(a_0,b_1)} \cup \mxl(I(U_{a_0},F_{b_0})) \cup \mnl(I(U_{a_0},F_{b_0})). 
 \]
 % \begin{align*}
 %  I(U_{a_0},F_{b_0})
 %  &=F_{(a_1,b_0)} \cup F_{(a_0,b_1)} \cup \mxl \cup \mnl \cup \{\text{ down beat points }\} \\
 %  &\simeq F_{(a_1,b_0)} \cup F_{(a_0,b_1)} \cup \mxl \cup \mnl 
 % \end{align*}
 Since $F_{(a_1,b_0)} \cap F_{(a_0,b_1)}=F_{(a_1,b_1)}$, 
 we have $F_{(a_1,b_0)} \cup F_{(a_0,b_1)} \simeq \bS(F_{(a_1,b_1)}) \simeq *$. 
 Note that $F_{(a_1,b_0)} \cup F_{(a_0,b_1)}$ is an up set. 
 By applying \cref{suspensionposetlemma,intersectiononepoint} to 
 $F_{(a_1,b_0)} \cup F_{(a_0,b_1)} \cup \mxl(I(U_{a_0},F_{b_0}))$ and $\mnl(I(U_{a_0},F_{b_0}))$, 
 we see that 
 the connected component of the right hand side containing 
 $F_{(a_1,b_0)} \cup F_{(a_0,b_1)}$ is weak homotopy equivalent to
 \[
 \left(\bigvee_{\substack{(a,b)\in \mnl(I(U_{a_0},F_{b_0})) \\ \left(F_{(a_1,b_0)} \cup F_{(a_0,b_1)}\right)\cap F_{(a,b)}\ne \emptyset }} \bS \left(\left(F_{(a_1,b_0)} \cup F_{(a_0,b_1)}\right)\cap F_{(a,b)}\right) \right) 
 \vee \left(\bigvee S^1\right).
 \]
\end{proof}

\begin{rem}
 \begin{enumerate}
  \item If $a_0\in \mnl(X)$, then $U_{a_0}=\{a_0\}$ and we have  
	\begin{align*}
	 I(U_{a_0}, F_{b_0})&=I(\{a_0\}, F_{b_0})\cong F_{a_0}\cap F_{b_0}.
	 \intertext{Similarly, if $b_0\in \mxl(X)$, then we have} 
	 I(U_{a_0}, F_{b_0})&=I(U_{a_0}, \{b_0\})\cong \left(U_{a_0}\cap U_{b_0}\right)^o.
	\end{align*}
  \item If $a_0\leq b_0$, then we have 
	\begin{align*}
	 I(U_{a_0},F_{b_0})&=U_{a_0}^o\times F_{b_0} \simeq *.
	\end{align*}
  \item If $a_0\not\leq b_0$, then $a_0\not\in U_{b_0}$ and hence $a_0\not\in A_1$. Therefore $a_1<a_0$. 
	Similarly, $b_0< b_1$.
  \item If $I(A_0,B_0)=\emptyset$, then we see that 
	$U_{a_0}\cap F_{b_0}$ is $\{a_0,b_0\}$ or empty.
 \end{enumerate}
\end{rem}

\begin{lem}
 We consider the same situation as in \cref{specialinterval}. 

 For all $(a,b)\in I(U_{a_0},F_{b_0})$, we have 
 \begin{align*}
  F_{(a_1,b_0)} \cap F_{(a,b)}&= (U_a \cap U_{a_1})^o \times F_b,  \\
  F_{(a_0,b_1)} \cap F_{(a,b)}&= U_a^o \times (F_b\cap F_{b_1}).
 \end{align*}

 If $(a,b) \leq (a_1,b_1)$, then  
 $\left( F_{(a_1,b_0)} \cup  F_{(a_0,b_1)}\right) \cap F_{(a,b)}$ is contractible.
\end{lem}
\begin{proof}
 Note that we have $a_1<a_0\leq b_1$ and $a_1\leq b_0<b_1$.
 Suppose $(c,d)\in I(U_{a_0},F_{b_0})$. Then   
 % $(a_1,b_0)\leq (c,d)$ and $(a,b)\leq (c,d)$ if and only if 
 \begin{align*}
  (c,d) \in F_{(a_1,b_0)} \cap F_{(a,b)} 
  &\Leftrightarrow 
  (a_1,b_0)\leq (c,d) \text{ and } (a,b)\leq (c,d) \\
  &\Leftrightarrow 
  c\leq a_1 \text{ and }  b_0\leq d \text{ and }  c\leq a \text{ and }  b\leq d \\ 
  &\Leftrightarrow c\in U_{a_1}\cap  U_a \text{ and } d\in F_{b}. 
 \end{align*}
 On the other hand, if $(c,d)\in (U_a \cap U_{a_1})^o \times F_b$, then 
 $c\in U_{a_0}$, $d\in F_{b_0}$, and  $c\leq a_1\leq b_0\leq d$, 
 and hence $(c,d)\in I(U_{a_0},F_{b_0})$. 
 Therefore, as we saw, $(c,d)\in F_{(a_1,b_0)} \cap F_{(a,b)}$.
 % \[
 %  \xymatrix{
 % & & b_1& & b \\
 % & a_0 \ar@{-}[ru] & & b_0 \ar@{-}[ru] \ar@{-}[lu]|!{[l];[ru]}\hole \\
 % a \ar@{-}[ru] \ar@{-}[rrrruu]
 % & & a_1 \ar@{-}[lu]|!{[ll];[u]}\hole  \ar@{-}[ru]
 % }
 % \]

 If $(a,b) \leq (a_1,b_1)$, namely, if $a\geq a_1$ and $b\leq b_1$, then  
 $a=a_1$ or $b=b_1$ 
 because,  
 if $a\ne a_1$, then $a>a_1=\max A_1$ hence $a \in A_0$,  
 and since $I(A_0,B_0)=\emptyset$, we have $b\not\in B_0$ and $b=b_1$.
 
 Since $a_1\leq a\leq a_0$ and $b_0\leq b\leq b_1$, 
 we have $U_{a_1}\subset U_a \subset  U_{a_0}$ and 
 $F_{b_0}\supset F_b \supset F_{b_1}$. 
 Therefore, we have 
 \begin{align*}
 \left( F_{(a_1,b_0)} \cup  F_{(a_0,b_1)}\right) \cap F_{(a,b)} 
  &=
  \left((U_a \cap U_{a_1})^o \times F_b \right) \cup 
  \left(U_a^o \times (F_b\cap F_{b_1}) \right) \\
  &= \left(U_{a_1}^o \times F_{b} \right) \cup 
  \left(U_{a}^o \times F_{b_1} \right) \\
  &=\begin{cases}
     U_{a_1}^o \times F_{b} \simeq *,  & a=a_1 \\
     U_{a}^o \times F_{b_1} \simeq *,  & b=b_1. 
    \end{cases}
 \end{align*}
\end{proof}

\begin{coro}
 \label{specialinterval-two-empty}
 We consider the same situation as in \cref{specialinterval}. 
 
 If $A_1=B_1=\emptyset$, then 
 $U_{a_0} \cup F_{b_0}$ is weak homotopy equivalent to a wedge of spheres of dimension at most $2$.
 % $I(U_{a_0},F_{b_0})$ is a disjoint union of wedge of spheres of dimension at most $1$. 

 In particular, if $F_{b_0}\cap F_{a_0}\subset \mxl(X)$ and $U_{a_0}\cap U_{b_0}\subset \mnl(X)$, 
 then this holds.
\end{coro}
\begin{proof}
 If $A_1=B_1=\emptyset$, 
 then $I(U_{a_0},F_{b_0})$ is homotopy equivalent to 
 $\mxl(I(U_{a_0},F_{b_0}))\cup \mnl(I(U_{a_0},F_{b_0}))$, whose height is at most $1$. 
 Therefore, $U_{a_0} \cup F_{b_0} \simeq_w \bS(I(U_{a_0},F_{b_0}))$ 
 is weak homotopy equivalent to a wedge of spheres of dimension at most $2$.

 If $F_{b_0}\cap F_{a_0}\subset \mxl(X)$, then $B_1=\emptyset$. 
 If $U_{a_0}\cap U_{b_0}\subset \mnl(X)$, then $A_1=\emptyset$. 
\end{proof}

\begin{coro}
 \label{specialinterval-one-empty}
 We consider the same situation as in \cref{specialinterval}. 

 If $A_1=\emptyset$, then  
 $I(U_{a_0},F_{b_0})$ splits into spaces smaller than $F_{b_1}$, 
 and hence,  so does $U_{a_0}\cup F_{b_0}$.
 If $B_1=\emptyset$,  then 
 $I(U_{a_0},F_{b_0})$ and $U_{a_0}\cup F_{b_0}$ split into spaces smaller than  $U_{a_1}$.
\end{coro}
\begin{proof}
 We consider the case where $A_1=\emptyset$. 

 In this case, $I(U_{a_0},F_{b_0})$ is homotopy equivalent to 
 $F_{(a_0,b_1)}\cup \mxl\cup \mnl$, 
 and each connected component  
 is weak homotopy equivalent to a wedge of some copies of 
 $S^1$ and $\bS(F_{(a_0,b_1)} \cap F_{(a,b)})$ 
 for some $(a,b)\in I(U_{a_0},F_{b_0})$.
 We have 
 \begin{align*}
  \bS\left(F_{(a_0,b_1)} \cap F_{(a,b)}\right) 
  &= \bS\left(U_a^o \times (F_b\cap F_{b_1})\right) \simeq \bS\left(F_b\cap F_{b_1}\right)
 \end{align*}
 and $\card{F_b\cap F_{b_1}} < \card{F_{b_1}}$ or $F_b\cap F_{b_1} \simeq *$.
\end{proof}

\begin{coro}
 \label{bodyissimple}
 Let $X$ be a connected finite $T_0$ space. 
 Suppose the following holds:
 \begin{enumerate}
  \item All the elements of $\cB-\mxl(\cB)$ are up beat points of $\cB$. 
  \item One of the connected components of $\cB$ is a chain. 
	Let $\cB_0$ be a connected component which is a chain. 
  \item There exists a point $a_0\in \mxl(X)$ such that $\cB_0 - U_{a_0}\ne \emptyset$.  
	We put $b_0=\min(\cB_0-U_{a_0})$. 
 \end{enumerate}
 Then, $U_{a_0}\cup F_{b_0}$ is weak homotopy equivalent to a wedge of spheres of dimension at most $2$.
\end{coro}
\begin{proof}
 % If $U_{a_0}\cup F_{b_0}$ is not connected, then 
 % $I(U_{a_0},F_{b_0})=\emptyset$ and $U_{a_0}\cup F_{b_0} \simeq S^0$
 % (though we do not need to distinguish this case). 
 %
 % Assume that $U_{a_0}\cup F_{b_0}$ is connected. 
 We use \cref{specialinterval}. 
 
 Since $X$ is connected, $\mxl(X)\cap \mnl(X)=\emptyset$, 
 and since $a_0\in \mxl(X)$, we have $a_0\not\in \mnl(X)$.
 Since $b_0\in \cB$, we have $b_0\not\in \mxl(X)$. 
 Clearly, $a_0\not\leq b_0$. 

 Since $a_0\in \mxl(X)$, we have $B_1=\left(F_{b_0}-\mxl(X)\right)\cap F_{a_0}=\emptyset$. 
 Since $b_0\in \cB$, we have 
 \begin{align*}
  \left(U_{a_0}-\mnl(X)\right) \cap U_{b_0} &\subset  U_{a_0} \cap U_{b_0} \cap \cB \\
  &\subset U_{a_0}\cap \cB_0 \\ 
  &\subset  \left(U_{a_0}-\mnl(X)\right) \cap U_{b_0}
 \end{align*}
 and since $b_0=\min(\cB_0-U_{a_0})\in \cB_0$, $\cB-\cB_0$ and $\cB_0$ are incomparable, 
 and $\cB_0-U_{a_0}$ is an up set of  $\cB_0$, we have
 \begin{align*}
  F_{b_0}\cap \cB&=F_{b_0}\cap \cB_0 =\cB_0-U_{a_0},
  \\ 
  F_{b_0}&\subset (\cB_0-U_{a_0} ) \cup \mxl(X).
 \end{align*}
 Therefore, we have
 \begin{align*}
  A_0 &= \left(U_{a_0}-\mnl(X)\right)-U_{b_0}, \\
  A_1&= \left(U_{a_0}-\mnl(X)\right) \cap U_{b_0} = \cB_0\cap U_{a_0}, \\
  B_0&=F_{b_0}-\mxl(X)=\cB_0 - U_{a_0}, \\\
  B_1&=\emptyset .
 \end{align*}

 We show that points of $A_0-\{a_0\}$ are up beat points of $U_{a_0}$. 
 Suppose $x\in A_0-\{a_0\}$. 
 Since $A_0-\{a_0\}\subset \cB$, $x$ is either maximal or up beat point of $\cB$
 by the assumption.
 Note that $\hat{F}_x = \left(\hat{F}_x \cap \cB\right) \cup \left(\hat{F}_x \cap \mxl(X)\right)$. 

 If $x$ is a maximal element of $\cB$, then $\hat{F}_x\subset \mxl(X)$ and 
 \[
  a_0\in \hat{F}_x\cap U_{a_0}\subset \mxl(X)\cap U_{a_0}=\{a_0\},   
 \]
 and hence $\hat{F}_x\cap U_{a_0}=\{a_0\}$. 
 Therefore, $x$ is an up beat point of $U_{a_0}$.

 If $x$ is an up beat point of $\cB$, then we put $\hat{x}=\min(\hat{F}_x\cap \cB)$.  We have 
 \begin{align*}
  \hat{F}_x\cap U_{a_0}&= \hat{F}_x \cap (\cB \cup \mxl(X))\cap U_{a_0} \\
  &=\left(\hat{F}_x \cap \cB \cap U_{a_0}\right) \cup \{a_0\}.
 \end{align*}
 If  $\hat{x}\not\in U_{a_0}$, then $\left(\hat{F}_x\cap \cB\right) \cap U_{a_0}=\emptyset$
 and $\hat{F}_x\cap U_{a_0}=\{a_0\} $.
 If $\hat{x}\in U_{a_0}$, then $\hat{x}=\min(\hat{F}_x \cap U_{a_0})$.
 In any case, $x$ is an up beat point of $U_{a_0}$.

 Because $B_0=F_{b_0}-\mxl(X)=\cB_0-U_{a_0}$ is a chain and $F_{b_0}\subset (\cB_0-U_{a_0} ) \cup \mxl(X)$, 
 every element of $B_0-\{b_0\}$ is a down beat point of $F_{b_0}$. 
 
 We show that $I(A_0,B_0)=\emptyset$.
 Since
 \[
 (A_0-\{a_0\})\cap \cB_{0}
 =\left(U_{a_0}\cap \cB - U_{b_0}\right)\cap \cB_0 \\
 =(U_{a_0}\cap \cB_0) - U_{b_0}=\emptyset ,
 \]
 we see that $A_0-\{a_0\}\subset \cB-\cB_0$, 
 and since $a_0\in \mxl(X)$ and $a_0\not\in B_0$, 
 we have 
 \[
 I(A_0,B_0)=I(A_0-\{a_0\},B_0) \subset I(\cB-\cB_0,\cB_0)=\emptyset. 
 \]
 % \begin{align*}
 %  (A_0-\{a_0\})\cap \cB_{0}&=\left(U_{a_0}\cap \cB - U_{b_0}\right)\cap \cB_0 \\
 %  &=(U_{a_0}\cap \cB_0) - U_{b_0}=\emptyset 
 %  \intertext{hence}
 %  A_0-\{a_0\}&\subset \cB-\cB_0
 %  \intertext{Since $a_0\in \mxl(X)$ and $a_0\not\in B_0$,}
 %  I(A_0,B_0)&=I(A_0-\{a_0\},B_0) \\
 %  &\subset I(\cB-\cB_0,\cB_0)=\emptyset 
 % \end{align*}

 Finally, if $A_1\ne\emptyset$, then  
 $A_1=\cB_0\cap U_{a_0}$ is a nonempty finite chain and so there exists $\max A_1$. 

 Therefore, the assumption of \cref{specialinterval} holds and $B_1=\emptyset$. 

 If $A_1=\emptyset$, then by \cref{specialinterval-two-empty}, 
 $U_{a_0}\cup F_{b_0}$ is weak homotopy equivalent to a wedge of spheres of dimension at most $2$.

 If $A_1\ne\emptyset$, then by \cref{specialinterval-one-empty}, 
 $I(U_{a_0},F_{b_0})$ splits into some copies of $S^1$ and 
 $\bS(U_a\cap U_{a_1})$ for some $a$.  
 Since $U_{a_1}\subset \cB_0\cup \mnl(X)$ and $\cB_0$ is a chain, 
 we see that $U_a\cap U_{a_1}$ is homotopy equivalent to a discrete space. 
 Therefore  $U_{a_0}\cup F_{b_0}$ is weak homotopy equivalent to a wedge of spheres of dimension at most $2$.
\end{proof}

\section{Some small finite spaces}
\label{sec-verysmallfinite}

\begin{defn}
 \label{2npointscircle}
 We denote the finite space of \cref{Fig2npointscircle} by $S^1_n$,  
 that is, the underlying set of 
 $S^1_n$ is the $2n$-element set
 \(
 S^1_n=\{a_0,\dots,a_{n-1},b_0,\dots,b_{n-1}\}
 \)
 and the order is given by 
 $b_i< a_i$ and $b_i < a_{i+1}$, where we consider $a_{n}=a_0$.
 Clearly, $S^1_n\simeq_w S^1$.

\begin{figure}[h]
 \centering
 \begin{tikzpicture}[scale=1.2]
  \node (a0) at (0,0) {$a_0$};
  \node (a1) at (1,0) {$a_1$};
  \node at (2,0) {$\dots$};
  \node (an1) at (3,0) {$a_{n-2}$};
  \node (an) at (4,0) {$a_{n-1}$};
  \node (b0) at (0,-1) {$b_0$};
  \node (b1) at (1,-1) {$b_1$};
  \node at (2,-1) {$\dots$};
  \node (bn1) at (3,-1) {$b_{n-2}$};
  \node (bn) at (4,-1) {$b_{n-1}$};
  \draw (a0) -- (b0) -- (a1) -- (b1) -- (1.5,-.5);
  \draw[dotted] (1.5,-.5) -- (1.8,-.2);
  \draw[dotted] (2.2,-.8) -- (2.5,-.5);
  \draw (2.5,-.5) -- (an1) -- (bn1) -- (an) -- (bn) -- (a0);
 \end{tikzpicture}
 \caption{$S^1_n$}
 \label{Fig2npointscircle}
\end{figure}
\end{defn}

\begin{ex}
\label{S12andS13}
\begin{align*}
 S^1_2&=
 \begin{tikzpicture}[baseline=-.6cm]
  \node (a0) at (0,0) {$a_0$};
  \node (a1) at (1,0) {$a_1$};
  \node (b0) at (0,-1) {$b_0$};
  \node (b1) at (1,-1) {$b_1$};
  \draw (a0) -- (b0) -- (a1) -- (b1) -- (a0);
 \end{tikzpicture} 
 \cong \bS S^0
\\
 S^1_3&=
 \begin{tikzpicture}[baseline=-.6cm]
  \node (a0) at (0,0) {$a_0$};
  \node (a1) at (1,0) {$a_1$};
  \node (a2) at (2,0) {$a_{2}$};
  \node (b0) at (0,-1) {$b_0$};
  \node (b1) at (1,-1) {$b_1$};
  \node (b2) at (2,-1) {$b_2$};
  \draw (a0) -- (b0) -- (a1) -- (b1) -- (a2) -- (b2) -- (a0);
 \end{tikzpicture}
 =
 \begin{tikzpicture}[baseline=-.6cm]
  \node (a0) at (1,0) {$a_0$};
  \node (a1) at (0,0) {$a_1$};
  \node (a2) at (2,0) {$a_{2}$};
  \node (b0) at (0,-1) {$b_0$};
  \node (b1) at (1,-1) {$b_1$};
  \node (b2) at (2,-1) {$b_2$};
  \draw (a0) -- (b0) -- (a1) -- (b1) -- (a2) -- (b2) -- (a0);
 \end{tikzpicture}
 \cong
 \begin{tikzpicture}[baseline=-.6cm]
  \node (a0) at (1,0) {$a_1$};
  \node (a1) at (0,0) {$a_2$};
  \node (a2) at (2,0) {$a_{0}$};
  \node (b0) at (0,-1) {$b_0$};
  \node (b1) at (1,-1) {$b_1$};
  \node (b2) at (2,-1) {$b_2$};
  \draw (a0) -- (b0) -- (a1) -- (b1) -- (a2) -- (b2) -- (a0);
 \end{tikzpicture}
\end{align*} 
\end{ex}

 It is straightforward to see the following:
\begin{lem}
 \label{3by3bipartite}
 Any connected $(3,3)$-bipartite graph with no degree $2$ vertex is isomorphic to 
 one of the graphs in \cref{Fig3by3bipartite}.
 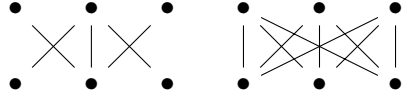
\begin{figure}[h]
  \centering
  \begin{tikzpicture}
   \foreach \i in {0,...,5} {
   \node (a\i) at (\i,1) {$\bullet$};
   \node (b\i) at (\i,0) {$\bullet$};
   }
   \draw (a1) -- (b0);
   \draw (a1) -- (b1);
   \draw (a1) -- (b2);
   \draw (a0) -- (b1) -- (a2);
   \draw (a3) -- (b3) -- (a4) -- (b4) -- (a5) -- (b5) -- (a3) -- (b4);
   \draw (b3) -- (a5);
   \draw (b5) -- (a4);
  \end{tikzpicture}
  \caption{$(3,3)$-bipartite graphs with no degree $2$ vertex.}
  \label{Fig3by3bipartite}
 \end{figure}
\end{lem}
% \begin{proof}
%  It is straightforward to see that if the graph has a degree $1$ vertex, 
%  then it is isomorphic to the left one.
% \end{proof}

\begin{lem}
 \label{4by4bipartite}
 Any $(4,4)$-bipartite graph whose all the vertices have degree $2$ is isomorphic to 
 $S^1_4$ or $S^1_2\amalg S^1_2$.
 % \begin{figure}[h]
 %  \centering
 %  \begin{tikzpicture}
 %   \foreach \i in {0,...,3} {
 %   \node (a\i) at (\i,1) {$\bullet$};
 %   \node (b\i) at (\i,0) {$\bullet$};
 %   }
 %   \foreach \i 
 %   [evaluate=\i as \j using \i+1]
 %   in {4,...,7} {
 %   \node (a\i) at (\j,1) {$\bullet$};
 %   \node (b\i) at (\j,0) {$\bullet$};
 %   }
 %   \draw (a0) -- (b0) -- (a1) -- (b1) -- (a2) -- (b2) -- (a3) -- (b3) -- (a0);
 %   \draw (a4) -- (b4) -- (a5) -- (b5) -- (a4);
 %   \draw (a6) -- (b6) -- (a7) -- (b7) -- (a6);
 %  \end{tikzpicture}
 %  \caption{$2$-regular $(4,4)$-bipartite graphs.}
 %  \label{Fig4by4bipartite}
 % \end{figure}
\end{lem}
% \begin{proof}
%  Clear.
% \end{proof}

We list up connected finite $T_0$ spaces of cardinality $4$ or less.

\begin{table}[h]
 \begin{tabular}{|c|c|c|c|c|}
  \hline
  \diagbox{$\card{\mxl(X)}$\rule[-7pt]{0pt}{0pt}}{$\card{X}$\rule{0pt}{14pt}} & 1 & 2 & 3 &  4 \\ \hline
  1 &  \begin{tikzpicture}[yscale=-1]
       \node at (0,0) {$\bullet$}; 
       \end{tikzpicture}
     & \begin{tikzpicture}[baseline=-.8cm,yscale=-1]
	\node (a) at (0,0) {$\bullet$};
	\node (b) at (0,1) {$\bullet$};
	\draw (a) -- (b);
       \end{tikzpicture}
     & \begin{tikzpicture}[baseline=-1.3cm,yscale=-1,xscale=.8]
   \begin{scope}
    \node (a) at (0,0) {$\bullet$};
    \node (b) at (0,1) {$\bullet$};
    \node (c) at (0,2) {$\bullet$};
    \draw (a) -- (b) -- (c);
   \end{scope}

   \begin{scope}[xshift = 1.2cm]
    \node (a) at (0,0) {$\bullet$};
    \node (b) at (-.5,1) {$\bullet$};
    \node (c) at (.5,1) {$\bullet$};
    \draw (b) -- (a) -- (c);
   \end{scope}
       \end{tikzpicture}
 & \begin{tikzpicture}[baseline=-1.8cm,yscale=-1,xscale=.8]
   \node (a) at (0,0) {$\bullet$};
   \node (b) at (0,1) {$\bullet$};
   \node (c) at (0,2) {$\bullet$};
   \node (d) at (0,3) {$\bullet$};
   \draw (a) -- (b) -- (c) -- (d);
  
  \begin{scope}[xshift = 1.2cm]
   \node (a) at (0,0) {$\bullet$};
   \node (b) at (-.5,1) {$\bullet$};
   \node (c) at (.5,1) {$\bullet$};
   \node (d) at (0,2) {$\bullet$};
   \draw (a) -- (b) -- (d) -- (c) -- (a);
  \end{scope}

  \begin{scope}[xshift = 2.8cm]
   \node (a) at (0,0) {$\bullet$};
   \node (b) at (-.5,1) {$\bullet$};
   \node (c) at (-.5,2) {$\bullet$};
   \node (d) at (.5,2) {$\bullet$};
   \draw (a) -- (b) -- (c);
   \draw (a) -- (d);
  \end{scope}

  \begin{scope}[xshift = 4.2cm]
   \node (a) at (0,0) {$\bullet$};
   \node (b) at (0,1) {$\bullet$};
   \node (c) at (-.5,2) {$\bullet$};
   \node (d) at (.5,2) {$\bullet$};
   \draw (d) -- (b) -- (c);
   \draw (a) -- (b);
  \end{scope}

  \begin{scope}[xshift = 6cm]
   \node (a) at (0,0) {$\bullet$};
   \node (b) at (-1,1) {$\bullet$};
   \node (c) at (0,1) {$\bullet$};
   \node (d) at (1,1) {$\bullet$};
   \draw (b) -- (a) -- (c);
   \draw (a) -- (d);
  \end{scope}
    \end{tikzpicture}
\\ \hline
  2 & & & \begin{tikzpicture}[baseline=.3cm,xscale=.9]
	   \node (a) at (0,0) {$\bullet$};
	   \node (b) at (-.5,1) {$\bullet$};
	   \node (c) at (.5,1) {$\bullet$};
	   \draw (b) -- (a) -- (c);
	  \end{tikzpicture}
 & \begin{tikzpicture}[baseline=.8cm,xscale=.9]
  \begin{scope}
   \node (a) at (0,0) {$\bullet$};
   \node (b) at (-.5,1) {$\bullet$};
   \node (c) at (-.5,2) {$\bullet$};
   \node (d) at (.5,2) {$\bullet$};
   \draw (a) -- (b) -- (c);
   \draw (a) -- (d);
  \end{scope}

  \begin{scope}[xshift = 1.5cm]
   \node (a) at (0,0) {$\bullet$};
   \node (b) at (0,1) {$\bullet$};
   \node (c) at (-.5,2) {$\bullet$};
   \node (d) at (.5,2) {$\bullet$};
   \draw (d) -- (b) -- (c);
   \draw (a) -- (b);
  \end{scope}

  \begin{scope}[xshift = 3cm]
   \node (a) at (-.5,0) {$\bullet$};
   \node (b) at (.5,0) {$\bullet$};
   \node (c) at (-.5,1) {$\bullet$};
   \node (d) at (.5,1) {$\bullet$};
   \draw (c) -- (a);
   \draw (c) -- (b);
   \draw (d) -- (b);
  \end{scope}

  \begin{scope}[xshift = 4.5cm]
   \node (a) at (-.5,0) {$\bullet$};
   \node (b) at (.5,0) {$\bullet$};
   \node (c) at (-.5,1) {$\bullet$};
   \node (d) at (.5,1) {$\bullet$};
   \draw (c) -- (a);
   \draw (c) -- (b);
   \draw (d) -- (b);
   \draw (d) -- (a);
  \end{scope}
   \end{tikzpicture}
\\ \hline
  3& & & & \begin{tikzpicture}[baseline=.3cm,xscale=.9]
	    \node (a) at (0,0) {$\bullet$};
	    \node (b) at (-1,1) {$\bullet$};
	    \node (c) at (0,1) {$\bullet$};
	    \node (d) at (1,1) {$\bullet$};
	    \draw (b) -- (a) -- (c);
	    \draw (a) -- (d);
	   \end{tikzpicture}
\\ \hline
 \end{tabular}
 \caption{Connected finite $T_0$ spaces of cardinality at most $4$.}
 \label{reallysmallspaces}
\end{table}

\section{\texorpdfstring{$\card{\mxl(X)}\leq 3$}{|mxl(X)| <= 3}}
\label{sec-mleq3}

In this section, we assume that $\card{X}>1$ and 
$X$ is a connected minimal finite space, that is, 
$X$ is a connected finite $T_0$ space without beat points. 
In this case, $\mxl(X)\cap \mnl(X)=\emptyset$ and $\card{X}\geq 4$. 

Following Cianci-Ottina \cite{CianciOttina}, we use the following notations:

\begin{defn}
 We put 
  \begin{align*}
   \cB&=X-\mxl(X)-\mnl(X), &    l&=\card{\cB }, \\ 
   m&=\card{\mxl(X)}, &   n&=\card{\mnl(X)}, \\
   m'&=\card{\mxl(\cB)}, &    n'&=\card{\mnl(\cB)},
   \intertext{and for $x\in X$ and $a\in \mxl(X)$, we put }
  \alpha_x&=\card{\mxl(F_x)}=\card{\mxl(X)\cap F_x},\\
  \beta_x&=\card{\mnl(U_x)}=\card{\mnl(X)\cap U_x}, \\
   \gamma_a&=\card{U_a\cap \mxl(\cB)}
 \end{align*}
 Since $X$ does not have beat points, $\alpha_x\geq 2$ if $x\not\in \mxl(X)$ 
 and $\beta_x\geq 2$ if $x\not\in\mnl(X)$.
\end{defn}

Note that 
\begin{align*}
 \card{I(\mxl(\cB), \mxl(X))} &= \sum_{b\in \mxl(\cB)}\alpha_b = \sum_{a\in \mxl(X)}\gamma_a 
 \intertext{because}
 I(\mxl(\cB), \mxl(X))&=\Set{(b,a)\in \mxl(\cB)\times \mxl(X) | b\leq a} \\
 &=\bigcup_{b\in \mxl(\cB)} p_1^{-1}(b) = \bigcup_{b\in \mxl(\cB)} \{b\}\times \left(\mxl(X)\cap F_b\right) \\
 &=\bigcup_{a\in \mxl(X)} p_2^{-1}(a) = \bigcup_{a\in \mxl(X)} \left(U_a\cap \mxl(\cB)\right) \times \{a\}.
\end{align*}

We study the weak homotopy type of $X$ of $m\leq 3$. 

\begin{lem}
 \label{msthan2}
 If $m\leq 2$, then $X$ splits into smaller spaces.
\end{lem}
\begin{proof}
 If $m=1$, then $X$ has the maximum and is contractible.

 If $m=2$ and $\mxl(X)=\{a,b\}$, then  
 $X=U_a\cup U_b\simeq_w \bS(U_a \cap U_b)$ 
 and $U_a\cap U_b\simeq *$ or $\card{U_a\cap U_b}<\card{X}$.  
\end{proof}

\begin{lem}
 \label{misthreempistwo}
 If $m=3$ and $m'= 2$, 
 then $X$ splits into smaller spaces.
\end{lem}
\begin{proof}
 Since $\sum_{b\in \mxl(\cB)}\alpha_b\geq 2m'= 4 > 3 = 1\cdot m$, 
 there exists $a\in \mxl(X)$ such that $\gamma_a>1$, namely, $\gamma_a=2=m'$. 
 Therefore, we have
 \[
  X=U_a\cup \mxl(X) \cup \mnl(X)
 \]
 and by \cref{Uamxlmnlsplit}, $X$ splits into smaller spaces.
\end{proof}

\begin{lem}
 \label{misthreeandcontractible}
 If $m=3$ and there exist two points $a_0,a_2\in \mxl(X)$ such that 
 $U_{a_0}\cap U_{a_2}$ is homotopically trivial, 
 then $X$ splits into smaller spaces.
\end{lem}
\begin{proof}
 Suppose $\mxl(X)=\{a_0,a_1,a_2\}$.

 Since $U_{a_0} \cup U_{a_2}$ is a down set, by \cref{unionhtptriv}, we see that
 \begin{align*}
  U_{a_0} \cup U_{a_2}&\simeq_w \bS\left(U_{a_0} \cap U_{a_2}\right) \simeq_w * \\ 
  X&=U_{a_1}\cup \left(U_{a_0} \cup U_{a_2} \right) \\
  &\simeq_w \bS\left(U_{a_1}\cap \left(U_{a_0} \cup U_{a_2}\right)\right)  
 \end{align*}
 and $U_{a_1}\cap \left(U_{a_0} \cup U_{a_2}\right) \subsetneq U_{a_1}$.
\end{proof}

\begin{rem}
 In fact, we can show that 
 $U_{a_0}\cup U_{a_2}=X-\{a_1\}$, 
 $U_{a_1}\cap \left(U_{a_0} \cup U_{a_2}\right) = \hat{U}_{a_1}$ and 
 $X\simeq_w \bS(\hat{U}_{a_1})$.

 More generally, if $X$ is a connected minimal finite space and $a_0\in \mxl(X)$, 
 then 
 \[
 \bigcup_{\mathclap{a\in \mxl(X)-\{a_0\}}}U_{a} = X-\{a_0\}.
 \]

 Actually, since $a_0\not\in U_a$, we have 
 \[
 \bigcup_{\mathclap{a\in \mxl(X)-\{a_0\}}}U_{a} \subset X-\{a_0\}.
 \]
 On the other hand, we have 
 $X-\{a_0\}=\left(\mxl(X)-\{a_0\}\right)\cup \cB \cup \mnl(X)$,  
 and clearly, $\mxl(X)-\{a_0\}$ is contained in the left hand side. 

 If $b\in \mxl(\cB)$, then $\card{\hat{F}_b}\geq 2$ because 
 $b$ is not a beat point. Hence, there exists $a\in \mxl(X)-\{a_0\}$ such that $b<a$. 
 Therefore, $\cB$ is contained in the left hand side.

 If $c\in \mnl(X)$,  then $\hat{F}_c\ne \emptyset$ since $X$ is connected. 
 If $\hat{F}_c\cap \cB\ne\emptyset$, then $c$ is contained in the left hand side.
 If $\hat{F}_c\cap \cB=\emptyset$, then $\hat{F}_c\subset \mxl(X)$, 
 and since $\card{\hat{F}_c}\geq 2$, $c$ is contained in the left hand side.
\end{rem}

\begin{lem}
 \label{misthreeand1connected}
 Suppose $\mxl(X)=\{a_0,a_1,a_2\}$. 
 If $U_{a_0}\cap U_{a_1}$ is connected and 
 $\hat{U}_{a_2}$ is weak homotopy equivalent to 
 a simplicial complex of dimension at most $1$, 
 then $X$ splits into smaller spaces.
\end{lem}
\begin{proof}

 Note that
 $\widehat{C}_{a_2} = \widehat{U}_{a_2}$ 
 and $X-\{a_2\}=U_{a_0} \cup U_{a_1}$.

 Since $U_{a_0}\cap U_{a_1}$ is connected, 
 $\geop{U_{a_0} \cup U_{a_1}}\simeq \susp{\geop{U_{a_0} \cap U_{a_1}}} $ is simply connected. 
 Hence the inclusion
 \[
  \geop{\hat{U}_{a_2}}=\geop{\hat{C}_{a_2}}  \to \geop{X-\{a_2\}}=\geop{U_{a_0} \cup U_{a_1}}
 \]
 is null homotopic. Therefore, by \cref{weakgamma},  we have
 \[
  X \simeq_w \left(X-\{a_2\}\right) \vee \bS(\hat{U}_{a_2}).
 \]
\end{proof}

\begin{lem}
 \label{mandmpis3}
 If $m=m'=3$, 
 then $X$ splits into smaller spaces, or 
 $\mxl(X)\cup \mxl(\cB)$ is isomorphic to $S^1_3$.
 \begin{figure}[h]
  \centering
 \begin{tikzpicture}
  \foreach \i 
  [evaluate=\i as \j using int(2-\i)]
  in {0,1,2}{
  \node (a\j) at (\i,1) {$a_{\j}$};
  \node (b\i) at (\i,0) {$b_{\i}$};
  }
  \draw (a2) -- (b0) -- (a1) -- (b2) -- (a0) -- (b1) -- (a2);
 \end{tikzpicture}
  \caption{$S^1_3$}
  \label{33circle}
 \end{figure}
\end{lem}
\begin{proof}
 Since $\card{\mxl(\cB)}=m'=3$, if there exists an element $a\in \mxl(X)$ such that 
 $\gamma_a=\card{U_a\cap \mxl(\cB)}=3$,  
 then $X=U_{a}\cup \mxl(X)\cup \mnl(X)$ and $X$ splits into smaller spaces by \cref{Uamxlmnlsplit}. 

 Assume that $\gamma_a\leq 2$ for all $a\in \mxl(X)$. 
 Recall that $\alpha_b=\card{F_b\cap \mxl(X)}\geq 2$ for all $b\in \mxl(\cB)$. 
 Since 
 \[
  2\cdot 3\geq \sum_{a\in \mxl(X)}\gamma_a = \sum_{b\in \mxl(\cB)}\alpha_b\geq 2\cdot 3, 
 \]
 we see that 
 $\gamma_a = 2$ for all $a\in \mxl(X)$ and $\alpha_b=2$ for all $b\in \mxl(\cB)$. 
 Now, it is straightforward to see 
 that $\mxl(X)\cup \mxl(\cB)$ is isomorphic to $S^1_3$.

\end{proof}

\begin{lem}
 \label{mandmpis3andweakbeatpt}
 If $m=m'=3$ and 
 there exists $b\in \cB$ such that $\card{\hat{F}^{\cB}_b}=\card{\hat{F}_b\cap \cB}=2$, 
 then $X$ splits into smaller spaces.
\end{lem}
\begin{proof}
 By \cref{mandmpis3}, we may assume that 
 $\mxl(X)\cup \mxl(\cB) = S^1_3$ of \cref{33circle}.

 We show that $\hat{F}^{\cB}_b\subset \mxl(\cB)$. 
 Since $\hat{F}^{\cB}_b\ne \emptyset$, we have $b\not\in \mxl(\cB)$ 
 and $\hat{F}^{\cB}_b \cap \mxl(\cB)\ne \emptyset$. 

 Suppose that $\hat{F}^{\cB}_b\not\subset \mxl(\cB)$. 
 Then we have $\card{\hat{F}^{\cB}_b \cap \mxl(\cB)}=1$. 
 We may assume that $b_0\in F_b$ and $b_1,b_2\not\in F_b$.
 Suppose $\hat{F}^{\cB}_b=\{b',b_0\}$. 
 Since $b<b'\in \cB$ and $b'\not\in \mxl(\cB)$, 
 we have $b<b'<b_0$. 
 Since $b'\not< b_1,b_2$ and $b'$ is not a beat point, 
 we have $b'<a_0$.
 Therefore, we have $\hat{F}_b=\{b',b_0,a_0,a_1,a_2\}$ and 
 $b'=\min \hat{F}_b$. Then $b$ is a beat point, which contradicts to the assumption. 
 
  \begin{figure}[h]
  \centering
 \begin{tikzpicture}
  \foreach \i 
  [evaluate=\i as \j using int(2-\i)]
  in {0,1,2}{
  \node (a\j) at (\i,1) {$a_{\j}$};
  \node (b\i) at (\i,0) {$b_{\i}$};
  }
  \draw (a2) -- (b0) -- (a1) -- (b2) -- (a0) -- (b1) -- (a2);
  \node (b) at (0,-2) {$b$};
  \node (b') at (0,-1) {$b'$};
  \draw (b) -- (b') -- (b0);
 \end{tikzpicture}
   \caption{$\widehat{F}_b\not\subset \mxl(\cB)$.}
  \end{figure}

 Therefore, $\hat{F}^{\cB}_b\subset \mxl(\cB)$. 

 We show that $b$ is a weak beat point. 
 
 We may suppose  $\hat{F}^{\cB}_b=\{b_0,b_2\}$. 
 Then, since $\mxl(X)=\{a_0,a_1,a_2\}\subset F_b$, we have 
 \begin{align*}
  \hat{F}_b&=(\hat{F}_b \cap \cB) \cup \mxl(X) \\
  &=\{a_0,a_1,a_2,b_0,b_2\}
 \end{align*}
 which is contractible, and so $b$ is a weak beat point. 

 Therefore $X\simeq_w X-\{b\}$ hence $X$ splits into smaller spaces.
\end{proof}

\begin{lem}
 \label{mis3mpis4}
 If $m=3$ and $m'=4$, 
 then $X$ splits into smaller spaces, or 
 $\mxl(X)\cup \mxl(\cB)$ is isomorphic to 
 one of spaces of \cref{figmis3mpis4}.

\begin{figure}[h]
 \centering

 \begin{tikzpicture}
  \node (b0) at (-.5,0) {$b_0$};
  \foreach \i 
  [evaluate=\i as \j using int(\i+1)]  
  in {0,1,2}{
  \node (a\i) at (\j,1.5) {$a_{\i}$};
  \node (b\j) at (\j,0) {$b_{\j}$};
  \draw (b0) -- (a\i);
  }
  \draw (a0) -- (b1) -- (a1) -- (b3) -- (a2) -- (b2) -- (a0);

  \node at (1.5,-1) {(a)};

 \begin{scope}[xshift=5cm]
  \foreach \i 
  in {0,1,2}{
  \node (a\i) at (\i,1) {$a_{\i}$};
  \node (b\i) at (\i,0) {$b_{\i}$};
  }
  \node (b3) at (3,0) {$b_3$};
  \draw (a0) -- (b0) -- (a1) -- (b1) -- (a0);
  \draw (a0) -- (b2);
  \draw (a1) -- (b3);
  \draw (b2) -- (a2) -- (b3);

  \node at (1.5,-1) {(b)};
 \end{scope}
 \end{tikzpicture}
 \caption{$m=3$ and $m'=4$.}
 \label{figmis3mpis4}
\end{figure}
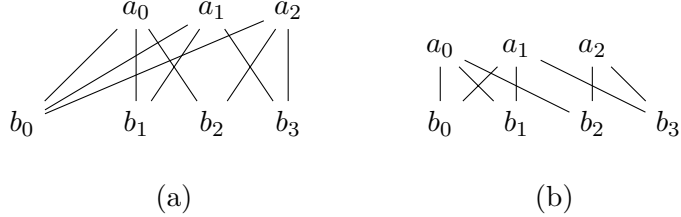
\end{lem}
\begin{proof}
 This can be shown similarly to \cref{mandmpis3}. 
 If there exists an element $a\in \mxl(X)$ such that 
 $\gamma_a=4$,  
 then $X=U_{a}\cup \mxl(X)\cup \mnl(X)$ and $X$ splits into smaller spaces. %by \cref{Uamxlmnlsplit}. 
 If $\gamma_a\leq 3$ for all $a\in \mxl(X)$, 
 then, 
 since
 \[
  3\cdot 3 \geq \sum_{a\in \mxl(X)}\gamma_a = \sum_{b\in \mxl(\cB)}\alpha_b\geq 2\cdot 4, % = 3+3+2,  %> 3+2+2, 
 \]
 we see that 
 % $\gamma_a=2$ or $3$ for all $a\in \mxl(X)$ and 
 % $\gamma_a=3$ for at least two of them, 
 % and that $\alpha_b\leq 3$ for all $b\in \mxl(\cB)$ and 
 % $\alpha_b=2$ for at least three of them.
 $\gamma_a=2$ or $3$ for all $a\in \mxl(X)$ and 
 $\gamma_a=2$ for at most one of them, 
 and that $\alpha_b=2$ or $3$ for all $b\in \mxl(\cB)$ and 
 $\alpha_b=3$ for at most one of them. 
 If there exists an element $a\in  \mxl(X)$ such that $\gamma_a=2$, 
 then $\alpha_b=2$ for all $b\in \mxl(\cB)$ and 
 we easily see that $\mxl(X)\cup \mxl(\cB)$ is isomorphic to the space (b) of \cref{figmis3mpis4}. 
 Otherwise, $\gamma_a=3$ for all $a$ and there exists one element $b\in \mxl(\cB)$ such that 
 $\alpha_b=3$, and we see that 
 $\mxl(X)\cup \mxl(\cB)$ is isomorphic to the space (a) of \cref{figmis3mpis4}. 
\end{proof}

\section{\texorpdfstring{$\card{\cB}\leq 5$}{|B|<= 5}}
\label{sec-lleq5}

We assume that $X$ is a connected minimal finite space of $\card{X}>1$ 
in this section, too, 
and we study the weak homotopy type of $X$ of $l=\card{\cB}\leq 5$.

\begin{lem}
 \label{liszero}
 If $l=0$, namely, $X=\mxl(X)\cup \mnl(X)$, 
 then $X$ is weak homotopy equivalent to a wedge of $S^1$'s.
\end{lem}
\begin{proof}
 In this case, $\ordcpx{X}$ is a connected $1$-dimensional simplicial complex.
\end{proof}

\begin{lem}
 \label{lisone}
 If $m'=1$, namely, if there exists $\max \cB$, 
 then $X$ splits into smaller spaces.
 
 In particular, if $l=1$,  then $X$ splits into smaller spaces.
\end{lem}
\begin{proof}
 Let $b=\max \cB $. Clearly, we have $X=U_b\cup \mxl(X)\cup \mnl(X)$, 
 and the result follows from \cref{Uamxlmnlsplit}.
\end{proof}

\begin{lem}
 \label{listwo}
 If there exists an element $b\in \cB $ such that  $\cB - U_b$ is a chain, 
 then $X$ splits into smaller spaces.
 
 In particular, if $l=2$,  then $X$ splits into smaller spaces.
\end{lem}
\begin{proof}
 If $\cB- U_b=\emptyset$, then $b=\max \cB$ and the result follows from \cref{lisone}.

 Assume that $\cB -  U_b\ne \emptyset$ and put
 \[
  C=\cB - U_b, \quad c_0 = \max C,\quad c_1=\min C, \quad U=\bigcup_{a\in F_b}U_a.
 \]
 By the assumption, $C$ is a nonempty finite chain. 

 If $C-U=\emptyset$, namely, if $C\subset U$, 
 then $c_0\in U$, and hence there exists an element $a\in F_b$ such that $c_0\in U_a$. 
 Since $c_0=\max C$, we have $C\subset U_a$, 
 and since $a\in F_b$, we have $U_b\subset U_a$. 
 Therefore, $\cB\subset U_b\cup C\subset U_a$ and 
 \[
  X=U_a\cup \mxl(X) \cup \mnl(X).
 \]
 Hence, $X$ splits into smaller spaces.

 If $C\cap U=\emptyset$, then  
 $c_1\not\in U$ and so $F_{c_1}\cap F_b=\emptyset$. 
 Hence $\left(U_b\cup F_{c_1}\right) \cap  \hat{F}_b=\emptyset$, 
 and since $b\in \cB$, $\hat{F}_b\ne \emptyset$, 
 and so $U_b\cup F_{c_1} \subsetneqq X$. 
 Since $\cB\subset U_b\cup C\subset U_b\cup F_{c_1}$, we have  
 \[
  X=U_b\cup F_{c_1} \cup \mxl(X) \cup \mnl(X).
 \]
 Therefore, $X$ splits into smaller spaces by \cref{UaFbmxlmnlsplit}.

 Suppose $C-U\ne\emptyset$ and $C\cap U\ne \emptyset$, and put  
 $d_1 = \min (C-U)$ and $d_0=\max(C\cap U)$.
 Since $d_0\in U$, there exists an element $a\in F_b$ such that $d_0\in U_a$. 
 Since $d_0=\max(C\cap U)$, we have $C\cap U\subset U_a$, 
 since $a\in F_b$, we have $U_b\subset U_a$, and since $d_1=\min (C-U)$, 
 we have $C-U\subset F_{d_1}$. 
 Therefore, $\cB\subset U_b\cup C\subset U_a\cup F_{d_1}$ and we see that
 \[
  X=U_a\cup F_{d_1} \cup \mxl(X) \cup \mnl(X).
 \]
 Since $d_1\not\in U$, 
 we have $F_{d_1}\cap F_b=\emptyset$ and so 
 $F_b-U_a\subset F_b\subset F_{d_1}^c$.
 Therefore, we have 
 \[
 F_b-U_a \subset U_a^c \cap F_{d_1}^c = \left(U_a\cup F_{d_1}\right)^c.
 \]
 Since $X$ is a minimal finite space and $b$ is not a maximal element, 
 by \cref{Fb-Uanonempty}, $F_b-U_a\ne\emptyset$ 
 and hence $U_a\cup F_{d_1} \subsetneq X$. 
 Therefore, by \cref{UaFbmxlmnlsplit}, $X$ splits into smaller spaces.
\end{proof}

\begin{coro}
 \label{mxlbistwo}
 If $l\leq 5$ and $m'\leq 2$, then $X$ splits into smaller spaces.
\end{coro}
\begin{proof}
 If $m'<2$, the result follows from \cref{liszero,lisone}. 

 Suppose $m'=2$ and $\mxl(\cB)=\{b_1,b_2\}$. 
 Since $\card{\cB}=l\leq 5$, 
 if $\card{U_{b_1}\cap \cB}\geq 3$, then $\cB-U_{b_1}$ is a chain, 
 and if $\card{U_{b_1}\cap \cB}\leq 2$, then $\cB-U_{b_2}$ is a chain. 
 The result follows from \cref{listwo}.
\end{proof}

\begin{lem}
 \label{mp3m5ormp5m3chain}
 Assume that all the connected components of $\cB$ are chains. 

 If $m'\leq3$ and $m\leq 5$, or $m'\leq 5$ and $m\leq 3$, 
 then $X$ splits into smaller spaces. 
\end{lem}
\begin{proof}
 The case $m\leq 2$ follows from \cref{msthan2}, 
 the case $m'=0$ follows from \cref{liszero}, 
 the case $m'=1$ follows from \cref{lisone}, 
 and the case $m'= 2$ follows from \cref{listwo}. 
 Hence, we may assume  $m'\geq 3$ and $m\geq 3$, 
 that is, we may assume $m'=3$ and $3\leq m\leq 5$, 
 or $3\leq m'\leq 5$ and $m=3$.
 In these cases, since
 \[
  \left(\sum_{b\in \mxl(\cB)} \alpha_{b}\right) - (m'-2)m 
  \geq 2m' - (m'-2)m =  4 -  (m'-2)(m-2) 
  % \geq 4-(3-2)(5-2) = 1
  >0, 
 \]
 there exists an element $a\in \mxl(X)$ such that 
 $\gamma_a=\card{U_a\cap \mxl(\cB)}> m'-2$.

 If $\gamma_a=m'$, then $\mxl(\cB)\subset U_a$ and 
 \[
  X=U_a \cup \mxl(X) \cup \mnl(X), 
 \]
 therefore $X$ splits into smaller spaces. 

 When $\gamma_a=m'-1$, suppose $\mxl(\cB)-U_a=\{b\}$ 
 and let $\cB_0$ be the connected component of $\cB$ containing $b$, 
 then $\cB_0-U_a$ is a nonempty chain. 
 Put $b_0=\min(\cB_0-U_a)$, then we have  
 \[
  X=U_a \cup F_{b_0} \cup \mxl(X) \cup \mnl(X) 
 \]
 and by \cref{bodyissimple}, $U_a \cup F_{b_0}$ is weak homotopy equivalent to a wedge of spheres, 
 hence $X$ splits into smaller spaces by \cref{UaFbmxlmnlsplit}. 
\end{proof}

\begin{coro}
 \label{listhree}
 If $l=3$ and one of the following holds, then $X$ splits into smaller spaces.
 \begin{enumerate}
  \item $m'\leq 2$.
  \item $m\leq 5$.
 \end{enumerate}
\end{coro}
\begin{proof}
 Part 1 follows from \cref{mxlbistwo}. 

 Consider part 2. Note that $\card{\cB}=l=3$.

 If  $\cB$ is connected, then $\cB$ has maximum or minimum. 
 If $\max\cB$ exists, then the result follows from \cref{lisone}. 
 If $\min\cB$ exists, then the opposite $X^o$ splits into smaller spaces, and so does $X$ 
 since $X\simeq_w X^o$.

 If $\cB$ is not connected, then all the connected components of $\cB$ are chains 
 and the result follows from  \cref{mp3m5ormp5m3chain}.
\end{proof}

\begin{coro}
 \label{card13andl3}
 If $\card{X}\leq 13$ and $l\leq 3$, then $X$ splits into smaller spaces.
\end{coro}
\begin{proof}
 By considering the opposite if necessary, 
 we may assume $m\leq n$. 
 We also may assume $m\geq 3$ by \cref{msthan2}.
 If $l\leq 2$, then the result follows from \cref{liszero,lisone,listwo}. 
 If $3\leq m\leq n$ and $l=3$, then we have 
 \begin{align*}
  13&\geq \card{X} = l+m+n \geq l+2m=3+2m.
 \end{align*}
 Therefore $m\leq 5$ and the result follows from \cref{listhree}.
\end{proof}

\begin{coro}
 \label{lisfour}
 If $l=4$ and one of the following holds, then $X$ splits into smaller spaces.
 In particular, if $l=4$ and $m\leq 3$, then $X$ splits into smaller spaces.
 \begin{enumerate}
  \item $m'\leq 2$.
  \item $m'=3$ and $m\leq 5$.
  \item $m'=4$ and $m\leq 3$.
 \end{enumerate}
\end{coro}
\begin{proof}
 Part 1 follows from \cref{mxlbistwo}. 
 Part 3 follows from \cref{mp3m5ormp5m3chain} since $\cB$ is an antichain if $l=m'$. 

 If $l=4$ and $m'=3$, then $\cB$ is isomorphic to one of spaces of \cref{l4m3B}.

 \begin{figure}[h]
  \centering 

  \begin{tikzpicture}
   \foreach \i in {0,1,2}{
   \node (mn\i) at (4*\i, 0) {$\bullet$};
   \foreach \j in {0,1,2} {
   \node (mx\i-\j) at (4*\i + \j ,1) {$\bullet$};
   }
   \draw[thick] (mn\i) -- (mx\i-0);
   }
   \draw[thick] (mn0) -- (mx0-1);
   \draw[thick] (mn0) -- (mx0-2);
   \draw[thick] (mn1) -- (mx1-1);
  \end{tikzpicture}
  \caption{$\cB$ for $l=4$ and $m'=3$.}
  \label{l4m3B}
 \end{figure}

 Since $n'\leq 2$ in the first two cases, by taking the opposite, 
 we see that $X$ splits into smaller spaces. 
 In the last case, every connected component of $\cB$ is a chain 
 and the result follows from \cref{mp3m5ormp5m3chain}.
\end{proof}

We consider the case $l=5$.

\begin{lem}
 \label{l5mp3m3}
 If $l=5$, $m'=3$ and $m=3$, then $X$ splits into smaller spaces.
\end{lem}
\begin{proof}
 Note that, if $n'=\card{\mnl(\cB)}\leq 2$, 
 by taking the opposite, 
 we see that $X$ splits into smaller spaces by \cref{mxlbistwo}. 

 If $\cB$ is connected, then $\mxl(\cB)\cap \mnl(\cB)=\emptyset$. Since $\card{\mxl(\cB)}=m'=3$, 
 $\card{\mnl(\cB)}\leq 2$ and  $X$ splits into smaller spaces.

 Suppose $\cB$ is not connected. 
 By \cref{mandmpis3}, we may assume that $\mxl(X)\cup \mxl(\cB)$ is isomorphic to $S^1_3$. 
 Since $\card{\mxl(\cB)}=3$, 
 the number of connected components is at most $3$. 

 \begin{enumerate}[label=(\alph*)]
  \item The case where $\cB$ has $3$ connected components.
	
	Since the cardinalities of each component are  
	\[
	 5=1+1+3 =1+2+2, 
	\]
	there exists a component which consists of a single element, say $\{b_0\}$.  
	Since $\mxl(X)\cup \mxl(\cB) \cong S^1_3$, 
	there exists an element $a_0\in \mxl(X)$ such that $b_0\not\in U_{a_0}$. 
	Then we have 
	\[
	X=U_{a_0}\cup F_{b_0} \cup \mxl(X) \cup \mnl(X) .
	\]
	Since  all the elements of $\cB-\mxl(\cB)$ are up beat points of $\cB$, 
	$X$ splits into smaller spaces by \cref{bodyissimple}.

	\begin{figure}[h]
	 \centering
	 \begin{tikzpicture}
	  \foreach \i in {0,1}{
	  \node (m\i) at (\i,1) {$\bullet$};
	  \node (mb\i) at (1+\i,0) {$\bullet$};
	  % \node (c\i) at (1+\i,-1) {$\bullet$};
	  }
	  \node (a0) at (2,1) {$a_0$};
	  \node (b0) at (0,0) {$b_0$};
	  \node (c0) at (2,-1) {$\bullet$};
	  \node (c1) at (2,-2) {$\bullet$};
	  \draw (m0) -- (b0) -- (m1) -- (mb1) -- (a0) -- (mb0) -- (m0);
	  \draw (c1) -- (c0) -- (mb1);

	  \begin{scope}[xshift=5cm]
	   \foreach \i in {0,1}{
	  \node (m\i) at (\i,1) {$\bullet$};
	  \node (mb\i) at (1+\i,0) {$\bullet$};
	  \node (c\i) at (1+\i,-1) {$\bullet$};
	  }
	  \node (a0) at (2,1) {$a_0$};
	  \node (b0) at (0,0) {$b_0$};
	  \draw (m0) -- (b0) -- (m1) -- (mb1) -- (a0) -- (mb0) -- (m0);
	  \draw (c0) -- (mb1) -- (c1);
	  \end{scope}
	  
	  \begin{scope}[xshift=10cm]
	   \foreach \i in {0,1}{
	   \node (m\i) at (\i,1) {$\bullet$};
	   \node (mb\i) at (1+\i,0) {$\bullet$};
	   \node (c\i) at (1+\i,-1) {$\bullet$};
	   }
	   \node (a0) at (2,1) {$a_0$};
	   \node (b0) at (0,0) {$b_0$};
	   \draw (m0) -- (b0) -- (m1) -- (mb1) -- (a0) -- (mb0) -- (m0);
	   \draw (c0) -- (mb0) (mb1) -- (c1);
	  \end{scope}
	 \end{tikzpicture}
	 \caption{$\mxl(X)\cup \cB$ with $3$ components for $l=5$, $m=m'=3$.}
	\end{figure}
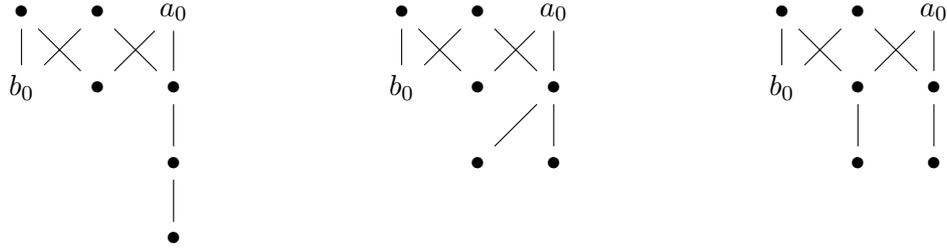

  \item The case where $\cB$ has $2$ connected components.
	
	The cardinalities of each component are  
	\[
	 5=1+4 =2+3.
	\]
	In the latter case, 
	$\card{\mnl(\cB)}=2$ and the result follows.
	Consider the former case. 
	Let $\cB_0=\{b_0\}$ be the component with a single element, 
	$\cB_1$ the component with $\card{\cB_1}=4$, 
	and $\mxl(\cB_1)=\{b_1,b_2\}$.
	Since $\cB_1$ is connected and $\card{\mxl(\cB_1)}=2$, we see that $\card{\mnl(\cB_1)}\leq 2$. 
	If $\card{\mnl(\cB_1)}=1$,  then $\card{\mnl(\cB)}=2$ and the result follows.
	Suppose $\card{\mnl(\cB_1)}=2$. 
	Since $\cB_1$ is connected, there exists an element $b\in \mnl(\cB_1)$ such that 
	$b_1,b_2\in F_b$ and, since $b_0\not\in F_b$, $\hat{F}^{\cB}_b=\{b_1,b_2\}$. 
	The result follows from \cref{mandmpis3andweakbeatpt}.
 \end{enumerate}
\end{proof}

\begin{lem}
 \label{l5mp4m3}
 If $l=5$, $m'=4$ and $m=3$, then $X$ splits into smaller spaces.
\end{lem}
\begin{proof}
  In this case, $\cB$ is isomorphic to one of spaces of \cref{Bl5mp4m3}.

 \begin{figure}[h]
  \centering 
  \begin{tikzpicture}
   \foreach \i in {0,1}{
   \node (mn\i) at (5*\i, 0) {$\bullet$};
   \foreach \j in {0,1,2,3} {
   \node (mx\i-\j) at (5*\i + \j ,1) {$\bullet$};
   }
   \draw[thick] (mn\i) -- (mx\i-0);
   }
   \draw[thick] (mn0) -- (mx0-1);
   \draw[thick] (mn0) -- (mx0-2);
   \draw[thick] (mn0) -- (mx0-3);
   \draw[thick] (mn1) -- (mx1-1);
   \draw[thick] (mn1) -- (mx1-2);

   \node at (1.5,-.5) {(a)};
   \node at (6.5,-.5) {(b)};

   \begin{scope}[yshift=-3cm]
   \foreach \j in {2,3} {
   \node (b\j) at (\j ,1) {$b_{\j}$};
   }
    \node (b) at (0,1) {$b$};
    \node (bp) at (1,1) {$b'$};
   \node (a1) at (0,0) {$c$}; 
   \draw[thick] (a1) -- (b);
   \draw[thick] (a1) -- (bp);
   \foreach \j in {0,1,2,3} {
   \node (mx2-\j) at (5 + \j ,1) {$\bullet$};
   }
    \node (mn2) at (5,0) {$\bullet$};
    \draw[thick] (mn2) -- (mx2-0);

   \node at (1.5,-.5) {(c)};
   \node at (6.5,-.5) {(d)};
   \end{scope}
  \end{tikzpicture}
  \caption{$\cB$ for $l=5$ and $m'=4$.}
  \label{Bl5mp4m3}
 \end{figure}

 In the cases (a) and (b), since $\card{\mnl{\cB}}\leq 2$, $X^o$ splits into smaller spaces by 
 \cref{mxlbistwo} and so does $X$.
 The case (d) follows from \cref{mp3m5ormp5m3chain}. 

 Consider the case (c). 
 Since $X$ is minimal and $m=\card{\mxl(X)}=3$, we have 
 $F_{b_2}\cap F_{b_3}\cap \mxl(X)\ne \emptyset$ 
 and hence, there exists an element $a_0\in \mxl(X)$ 
 such that $b_2,b_3\in U_{a_0}$. See \cref{Bl5mp4m3C}.

 \begin{figure}[h]
  \centering
  \begin{tikzpicture}
   \foreach \j in {2,3} {
   \node (b\j) at (\j ,1) {$b_{\j}$};
   }
    \node (b) at (0,1) {$b$};
    \node (bp) at (1,1) {$b'$};
   \node (a1) at (0,0) {$c$}; 
   \draw[thick] (a1) -- (b);
   \draw[thick] (a1) -- (bp);
   \node  (a0) at (3,2) {$a_0$};
   \draw[thick] (b2) -- (a0) -- (b3);
   % \draw[dashed] (a0) -- (b);
   \draw[dashed] (a0) -- (bp);
   \draw[dashed] (a0) -- (a1);
  \end{tikzpicture}
  \caption{The case (c).}
  \label{Bl5mp4m3C}
 \end{figure}
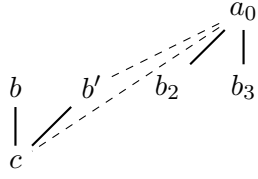

 We may assume that $\gamma_{a_0}=\card{U_{a_0}\cap\mxl(\cB)}=2$ or $3$ (see \cref{mis3mpis4}). 
 % If $\{b,b'\}\subset U_{a_0}$, then $X=U_{a_0}\cup \mxl(X) \cup \mnl(X)$ and 
 % the result follows. 

 % If $\{b,b'\}\cap  U_{a_0}=\emptyset$, 
 If $\gamma_{a_0}=2$, 
 then we have 
 \[
  X=U_{a_0}\cup F_{c} \cup \mxl(X) \cup \mnl(X). 
 \]
 We put
 \begin{align*}
  A_0&=\left(U_{a_0} - \mnl(X)\right)- U_{c} = \{a_0, b_2,b_3\}, & 
  A_1&=\left(U_{a_0} - \mnl(X)\right)\cap U_{c} = \emptyset \text{ or } \{c\},\\
  B_0&=\left(F_{c} - \mxl(X)\right)- F_{a_0} = \{b,b',c\}, &
  B_1&=\left(F_{c} - \mxl(X)\right)\cap F_{a_0} =\emptyset, 
 \end{align*}
 then $b_2$ and $b_3$ are up beat points of $U_{a_0}$, 
 $b$ and $b'$ are down beat points of $F_c$, 
 and $I(A_0,B_0)=\emptyset$.  
 If $A_1\ne \emptyset$, then $c=\max A_1$. 
 Therefore $X$ splits into smaller spaces by \cref{specialinterval-two-empty,specialinterval-one-empty}.

 % Finally, %if $b\not\in U_{a_0}$ and $b'\in U_{a_0}$, 
 If $\gamma_{a_0}=3$, we may assume
 that $b\not\in U_{a_0}$ and $b'\in U_{a_0}$, 
 % then we have 
 and we have
 \[
  X=U_{a_0}\cup F_{b} \cup \mxl(X) \cup \mnl(X).
 \]
 We put 
 \begin{align*}
  A_0&=\left(U_{a_0} - \mnl(X)\right)- U_{b} = \{a_0, b', b_2,b_3\}, &
  A_1&=\left(U_{a_0} - \mnl(X)\right)\cap U_{b} = \{c\},\\
  B_0&=\left(F_{b} - \mxl(X)\right)- F_{a_0} = \{b\}, &
  B_1&=\left(F_{b} - \mxl(X)\right)\cap F_{a_0} =\emptyset,
 \end{align*}
 then $b'$, $b_2$, and $b_3$ are up beat points of $U_{a_0}$, 
 $B_0\setminus \{b\}=\emptyset$, 
 $I(A_0,B_0)=\emptyset$, 
 and $c=\max A_1$. 
 Therefore $X$ splits into smaller spaces by \cref{specialinterval-one-empty}.
\end{proof}

\begin{coro}
 \label{lisfive}
 If $l=5$ and one of the following holds, then $X$ splits into smaller spaces.
 \begin{enumerate}
  \item $m'\leq 2$
  \item $m\leq 3$
 \end{enumerate}
\end{coro}
\begin{proof}
 Part 1 follows from \cref{mxlbistwo}. 
 For part 2, we may assume $m=3$ and $3\leq m' \leq 5$.
 The case $m'=3$ follows from \cref{l5mp3m3}, $m'=4$ from \cref{l5mp4m3}, 
 and $m'=5$ from \cref{mp3m5ormp5m3chain}.
\end{proof}

\begin{coro}
 \label{card11}
 If $\card{X}\leq 11$, then $X$ splits into smaller spaces.
\end{coro}
\begin{proof}
 We may assume that $3\leq m\leq n$. In this case, we have 
 \begin{align*}
  11&\geq \card{X} = l+m+n \geq l+6, 
 \end{align*}
 and hence $l\leq 5$. 

 By \cref{card13andl3}, we may assume that $l\geq 4$. 
 Hence we have
 \begin{align*}
  11&\geq \card{X} = l+m+n \geq l+2m\geq 4+2m 
 \end{align*}
 and so $m\leq 3$. 
 The result follows from \cref{lisfour,lisfive}.
\end{proof}

\section{\texorpdfstring{$\card{X}=12$}{|X|=12}}
\label{sec-card12}

In this section, we assume that $X$ is a connected minimal finite space of $\card{X}=12$ 
and show that $X$ splits into smaller spaces.
We need laborious case by case analysis.

\begin{prop}
 \label{card12}
 If $X$ is a connected minimal finite space with $\card{X}=12$, 
 then $X$ splits into smaller spaces.
\end{prop}
\begin{proof}
 We may assume that $3\leq m\leq n$. In this case, we have  
 \[
  12= \card{X} = l+m+n \geq 
 \begin{cases}
  l+6 \\
   l+2m 
  \end{cases}
 \]
 % \begin{align*}
 %  12= \card{X} = l+m+n &\geq l+6 \\
 %  12= \card{X} = l+m+n &\geq l+2m 
 % \end{align*}
 and hence 
 \begin{align*}
  l&\leq 6, \\
  m&\leq (12-l)/2.
 \end{align*}

 The case $l\leq 3$ follows from \cref{card13andl3}.

 The case $l=4$. 
 In this case, we have $m\leq 4$. 
 If $m\leq 3$ or $m=4$ and $m'\leq 3$, then 
 $X$ splits into smaller spaces by \cref{lisfour}.
 The remaining case is where $m=4$ and $m'=4$, 
 that is, $l=m=n=4$ and $\cB$ is an antichain. 
 We show this case in \cref{l4m4n4}. 

 If $l=5$, then $m\leq 3$ and $X$ splits into smaller spaces by \cref{lisfive}.

 If $l=6$, then  $m\leq 3$ and hence $m=n=3$. 
 We show this case in \cref{l6mn3}.
\end{proof}

 \begin{lem}
  \label{l4m4n4}
  If %$X$ is a connected minimal finite space such that 
  $l=m=n=4$ and $\cB$ is an antichain,  then $X$ splits into smaller spaces.
 \end{lem}

 \begin{proof}
  Note that $\mxl(\cB)=\cB$. 
  If there exists a maximal element $a\in \mxl(X)$ such that $\gamma_a=\card{U_a\cap \mxl(\cB)}\geq 3$, 
  then there exists an element $b\in \cB$ such that 
  \[
   X=U_a\cup F_b \cup \mxl(X) \cup \mnl(X).
  \]
  Therefore,  $X$ splits into smaller spaces by \cref{unchainsplit}. 
  
  Suppose $\gamma_a\leq 2$ for all $a\in \mxl(X)$. 

  Since $X$ does not have any beat point, $\alpha_b=\card{\mxl(X)\cap F_b}\geq 2$ for all $b\in \cB$. 
  Since 
  \[
  2\cdot 4 \geq \sum_{a\in \mxl(X)}\gamma_a = \sum_{b\in \cB}\alpha_b \geq 2\cdot 4 ,
  \]
  % \[
  % 2\cdot 4 \leq \sum_{b\in \cB}\alpha_b=\sum_{a\in \mxl(X)}\gamma_a \leq  2\cdot 4,
  % \]
  we have $\alpha_b=2$ for all $b\in \cB$ and $\gamma_a=2$ for all $a\in \mxl(X)$. 
  Therefore, $\mxl(X)\cup \cB$ is isomorphic to either $S^1_2\amalg S^1_2$ or $S^1_4$ 
  by \cref{4by4bipartite}. 
  Similarly, or by considering the opposite, 
  we may assume that $\cB\cup \mnl(X)$ is isomorphic to $S^1_2\amalg S^1_2$ or $S^1_4$.

  \begin{figure}[h]
   \centering
   \begin{tikzpicture}
    \foreach \i 
    [evaluate=\i as \j using int(2-\i)]
    in {0,1,2,3}{
    \node (a\i) at (2*\i,1) {$\bullet$};
    \node (b\i) at (2*\i,0) {$\bullet$};
    }
    \foreach \i 
    [evaluate=\i as \j using {int(mod(1+\i,4))}]
    in {0,1,2,3}{
    \draw (a\i) -- (b\i) -- (a\j);
    }

    \begin{scope}[yshift=2cm]
     \foreach \i 
    [evaluate=\i as \j using int(2-\i)]
    in {0,1,2,3}{
    \node (a\i) at (2*\i,1) {$\bullet$};
    \node (b\i) at (2*\i,0) {$\bullet$};
    }
    \foreach \i 
    [evaluate=\i as \j using {int(mod(1+\i,4))}]
    in {0,1,2,3}{
    \draw (a\i) -- (b\i);
    }
    \draw (a0) -- (b1);
    \draw (a1) -- (b0);
    \draw (a2) -- (b3);
    \draw (a3) -- (b2);
    \end{scope}

   \end{tikzpicture}
   \caption{  $S^1_2\amalg S^1_2$ and $S^1_4$}
  \end{figure}
  
  Consider the case where at least one of $\mxl(X) \cup \cB$ and $\cB\cup \mnl(X)$ 
  is isomorphic to $S^1_4$. 
  We may assume that $\cB\cup \mnl(X)$ is isomorphic to $S^1_4$. 
  In this case, we easily see that, for any $a\in \mxl(X)$, 
  $\hat{U}_a$ is contractible or homotopy equivalent to  $S^0$, 
  and $X-\{a\}$ is connected. Therefore,  
  $X$ splits into smaller spaces by \cref{weakgammamaximal}. 
  
  If both $\mxl(X) \cup \cB$ and $\cB\cup \mnl(X)$ are isomorphic to 
  $S^1_2\amalg S^1_2$, then we see that $X$ is isomorphic to the space of \cref{s12x4} 
  because $X$ is connected. 
  We see that, for any $a\in \mxl(X)$, 
  $\hat{U}_a$ is homotopy equivalent to  $S^0$ 
  and $X-\{a\}$ is connected. Therefore,  
  $X$ splits into smaller spaces by \cref{weakgammamaximal}.
  (In fact, one can easily see that $X\simeq_w \bigvee_5 S^1$.)

  \begin{figure}[h]
   \centering
   \begin{tikzpicture}
    \foreach \i 
    [evaluate=\i as \j using int(2-\i)]
    in {0,1,2,3}{
    \node (a\i) at (2*\i,2) {$\bullet$};
    \node (b\i) at (2*\i,1) {$\bullet$};
    % \node (c\i) at (2*\i,0) {$\bullet$};
    }
    \node (c0) at (0,-1) {$\bullet$};
    \node (c1) at (2,0) {$\bullet$};
    \node (c2) at (4,0) {$\bullet$};
    \node (c3) at (6,-1) {$\bullet$};

    \foreach \i 
    [evaluate=\i as \j using {int(mod(1+\i,4))}]
    in {0,1,2,3}{
    \draw (a\i) -- (b\i) -- (c\i);
    }
    \draw (a0) -- (b1);
    \draw (a1) -- (b0);
    \draw (a2) -- (b3);
    \draw (a3) -- (b2);
    \draw (c0) -- (b3);
    \draw (c3) -- (b0);
    \draw (c1) -- (b2);
    \draw (c2) -- (b1);
   \end{tikzpicture}
   \caption{ $X$ when $\mxl(X) \cup \cB$ and $\cB\cup \mnl(X)$ are $S^1_2\amalg S^1_2$. }
   \label{s12x4}
  \end{figure}

 \end{proof}

 \begin{lem}
  \label{l6mpmn3}
  If %$X$ is a connected minimal finite space with 
  $l=6$, $m=n=3$ and $m'=3$, then $X$ splits into smaller spaces.
 \end{lem}
 \begin{proof}
  If $n'=\card{\mnl(\cB)}\leq 2$, 
  then the opposite $X^o$ splits into smaller spaces by \cref{misthreempistwo,lisone} 
  and so does $X$. 
  By \cref{mandmpis3}, 
  we may assume that $\mxl(X)\cup \mxl(\cB)$ is identified with $S^1_3$  in \cref{33circle},  
  and so $\mxl(X)=\{a_0,a_1,a_2\}$ and $\mxl(\cB)=\{b_0,b_1,b_2\}$.

 Since $\card{\mxl(\cB)}=3$, 
 the number of connected components of $\cB$ is at most $3$. 

 \begin{enumerate}[label=(\alph*)]
  \item The case where $\cB$ has $3$ connected components.
	
	The cardinalities of each component are  
	\[
	6 = 1+1+4 = 1+2+3 = 2+2+2.
	\]
	The case $2+2+2$ follows from \cref{mp3m5ormp5m3chain} and 
	the case $1+2+3$ follows from \cref{bodyissimple} 
	as in the case (a) of the proof of \cref{l5mp3m3}. 

	Consider the case $1+1+4$. 
	Let $\cB_4$ be the connected component with $\card{\cB_4}=4$. 
	Since $\cB_4$ is connected, we have $\card{\mnl(\cB_4)}\leq 3$. 
	If $\card{\mnl(\cB_4)}=2$ or $3$, then the result follows from \cref{bodyissimple} 
	similarly to the case $1+2+3$ (see \cref{reallysmallspaces}).

	Suppose $\card{\mnl(\cB_4)}=1$. We may assume that $b_1\in \cB_4$. 
	Let $\cB_4=\{b_1,b_3,b_4,b_5\}$ and $b_4=\min\cB_4$. 
	In this case, we have $\mnl(\cB)=\{b_0,b_2,b_4\}$ 
	and hence $n=n'=3$, therefore we may assume that  
	$\mnl(\cB)\cup \mnl(X)$ is isomorphic to the one in \cref{3333-114} (a) 
	by \cref{mandmpis3}. 

  \begin{figure}[h]
  \centering
 \begin{tikzpicture}[scale=.85]
  \foreach \i 
  [evaluate=\i as \j using int(2-\i)]
  in {0,1,2}{
  \node (a\j) at (2*\i,1) {$a_{\j}$};
  \node (c\j) at (2*\i,-3) {$\bullet$};
  }
  \node (b0) at (0,-1) {$b_{0}$};
  \node (b1) at (2,0) {$b_{1}$};
  \node (b2) at (4,-1) {$b_{2}$};

  \node (b3) at (1,-1) {$b_3$};
  \node (b4) at (2,-2) {$b_4$};
  \node (b5) at (3,-1) {$b_5$};

  \draw (a2) -- (b0) -- (a1) -- (b2) -- (a0) -- (b1) -- (a2);
  \draw (c2) -- (b0) -- (c1) -- (b2) -- (c0) -- (b4) -- (c2);
  % \draw (a1) -- (b3) -- (c1) -- (b5) -- (a1);
  % \draw (b1) -- (b3) -- (b4) -- (b5) -- (b1);

  \node at (2, -4) {(a)};

  \begin{scope}[xshift=5cm]
  \foreach \i 
  [evaluate=\i as \j using int(2-\i)]
  in {0,1,2}{
  \node (a\j) at (2*\i,1) {$a_{\j}$};
  \node (c\j) at (2*\i,-3) {$\bullet$};
  }
  \node (b0) at (0,-1) {$b_{0}$};
  \node (b1) at (2,0) {$b_{1}$};
  \node (b2) at (4,-1) {$b_{2}$};

  \node (b3) at (1,-1) {$b_3$};
  \node (b4) at (2,-2) {$b_4$};
  \node (b5) at (3,-1) {$b_5$};

  \draw (a2) -- (b0) -- (a1) -- (b2) -- (a0) -- (b1) -- (a2);
  \draw (c2) -- (b0) -- (c1) -- (b2) -- (c0) -- (b4) -- (c2);
  % \draw (a1) -- (b3) -- (c1) -- (b5) -- (a1);
  \draw (b1) -- (b3) -- (b4) -- (b5) -- (b1);
   
  \node at (2, -4) {(b)};
  \end{scope}

  \begin{scope}[xshift=10cm]
  \foreach \i 
  [evaluate=\i as \j using int(2-\i)]
  in {0,1,2}{
  \node (a\j) at (2*\i,1) {$a_{\j}$};
  \node (c\j) at (2*\i,-3) {$\bullet$};
  }
  \node (b0) at (0,-1) {$b_{0}$};
  \node (b1) at (2,0) {$b_{1}$};
  \node (b2) at (4,-1) {$b_{2}$};

  \node (b3) at (1,-1) {$b_3$};
  \node (b4) at (2,-2) {$b_4$};
  \node (b5) at (3,-1) {$b_5$};

  \draw (a2) -- (b0) -- (a1) -- (b2) -- (a0) -- (b1) -- (a2);
  \draw (c2) -- (b0) -- (c1) -- (b2) -- (c0) -- (b4) -- (c2);
  \draw (a1) -- (b3) -- (c1) -- (b5) -- (a1);
  \draw (b1) -- (b3) -- (b4) -- (b5) -- (b1);
   
  \node at (2, -4) {(c)};
  \end{scope}
 \end{tikzpicture}
   \caption{$1+1+4$}
   \label{3333-114}
  \end{figure}
	
	% If $b_3$ and $b_5$ are comparable,  then we have 
	% $\card{\hat{F}^{\cB}_{b_3}}=\card{\{b_1,b_5\}}=2$ or 
	% $\card{\hat{F}^{\cB}_{b_5}}=\card{\{b_1,b_3\}}=2$, 
	% and the result follows from \cref{mandmpis3andweakbeatpt}.
	% (Actually, )

	Since $\cB_4$ is connected and has maximum and minimum, 
	$\cB_4$ is either a chain or isomorphic to the one in 
	\cref{3333-114} (b), but since $X$ does not have any beat point, 
	$\cB_4$ is not a chain.

	Now, since $\{b_1,a_0,a_2\}\subset \hat{F}_{b_3}\subset \{b_1,a_0,a_1,a_2\}$ 
	and $b_3$ is not a beat point, we have $b_3\prec a_1$, and  
	proceeding similarly, we see that $X$ is isomorphic to the space \cref{3333-114} (c). 
	Since $U_{a_0}\cap U_{a_2}=U_{b_1}\simeq *$, 
	$X$ splits into smaller spaces by \cref{misthreeandcontractible}. 
	In fact, we easily see that 
	\begin{align*}
	 \hat{U}_{a_1}&\simeq_w S^1\vee S^1 \vee S^1, \\
	 X\simeq_w \bS(\hat{U}_{a_1}) &\simeq_w S^2\vee S^2 \vee S^2.
	\end{align*}
%   \begin{figure}[h]
%    \centering
%  \begin{tikzpicture}
%   \foreach \i 
%   [evaluate=\i as \j using int(2-\i)]
%   in {0,1,2}{
%   % \node (a\j) at (2*\i,1) {$a_{\j}$};
%   \node (c\j) at (2*\i,-3) {$c_{\j}$};
%   }
%   \node (b0) at (0,-1) {};
%   \node (b1) at (2,0) {$b_1$};
%   \node (b2) at (4,-1) {};

%   \node (b3) at (1,-1) {$b_3$};
%   \node (b4) at (2,-2) {$b_4$};
%   \node (b5) at (3,-1) {$b_5$};

%   % \draw (b0) -- (a1) -- (b2);
%   \draw (c0) -- (b4) -- (c2);
%   \draw (b3) -- (c1) -- (b5);
%   \draw (b3) -- (b4) -- (b5) -- (b1) -- (b3);

%   \node at (2,-4) {$U_{a_0}\cap U_{a_2}$};

%   \begin{scope}[xshift=6cm]
%   \foreach \i 
%   [evaluate=\i as \j using int(2-\i)]
%   in {0,1,2}{
%   % \node (a\j) at (2*\i,1) {$a_{\j}$};
%   \node (c\j) at (2*\i,-3) {$c_{\j}$};
%   }
%   \node (b0) at (0,-1) {$b_0$};
%   % \node (b1) at (2,0) {$b_1$};
%   \node (b2) at (4,-1) {$b_2$};

%   \node (b3) at (1,-1) {$b_3$};
%   \node (b4) at (2,-2) {$b_4$};
%   \node (b5) at (3,-1) {$b_5$};

%   % \draw (b0) -- (a1) -- (b2);
%   \draw (c0) -- (b4) -- (c2) -- (b0) -- (c1) -- (b2) -- (c0);
%   \draw (b3) -- (c1) -- (b5);
%   \draw (b3) -- (b4) -- (b5);

%   \node at (2,-4) {$\hat{U}_{a_1}$};
%   \end{scope}

%  \end{tikzpicture}

%   \end{figure}

  \item The case where $\cB$ has $2$ connected components.
	
	The cardinalities of each component are  
	\[
	6 = 1+5 = 2+4 = 3+3.
	\]
	In the case $3+3$, there exists an element $b\in \cB$ such that $\card{\hat{F}^{\cB}_{b}}=2$, 
	and the result follows from \cref{mandmpis3andweakbeatpt}.

	In the case $2+4$, let $\cB_4$ be the connected component with $\card{\cB_4}=4$. 
	Since $\card{\mxl(\cB_4)}=2$, we have $\card{\mnl(\cB_4)}\leq 2$. 
	If $\card{\mnl(\cB_4)}=1$, then $\card{\mnl(\cB)}=2$ and hence the result follows.  
	If $\card{\mnl(\cB_4)}=2$, then we see that 
	there exists an element $b\in \cB$ such that $\card{\hat{F}^{\cB}_{b}}=2$ (see \cref{reallysmallspaces}), 
	and the result follows. 

	Consider the case $1+5$, and let $\cB_5$ be the connected component with $\card{\cB_5}=5$. 
	Since $\card{\mxl(\cB_5)}=2$, we have $\card{\mnl(\cB_5)}\leq 3$. 
	If $\card{\mnl(\cB_5)}=1$, then $\card{\mnl(\cB)}=2$, 
	and if $\card{\mnl(\cB_5)}=3$, then 
	there exists an element $b\in \cB$ such that $\card{\hat{F}^{\cB}_{b}}=2$, 
	and hence the result follows. 

	Suppose $\card{\mnl(\cB_5)}=2$. We may assume that $b_0\not\in \cB_5$. 
	Let $\cB_5=\{b_1,b_2, c, d_1,d_2\}$ and $\mnl(\cB_5)=\{d_1,d_2\}$. 
	We have $\mnl(\cB)=\{b_0,d_1,d_2\}$ and hence $n=n'=3$, therefore we may assume that 
	$\mnl(\cB)\cup \mnl(X)$ is isomorphic to the one in \cref{3333-15} (a).

  \begin{figure}[h]
  \centering
 \begin{tikzpicture}[scale=.85]
  \foreach \i 
  [evaluate=\i as \j using int(2-\i)]
  in {0,1,2}{
  \node (a\j) at (2*\i,1) {$a_{\j}$};
  \node (c\j) at (2*\i,-3) {$c_{\j}$};
  }
  \node (b0) at (0,-1) {$b_{0}$};
  \node (b1) at (2,0) {$b_{1}$};
  \node (b2) at (4,0) {$b_{2}$};

  \node (c) at (1,-1) {$c$};
  \node (d1) at (2,-2) {$d_1$};
  \node (d2) at (4,-2) {$d_2$};

  \draw (a2) -- (b0) -- (a1) -- (b2) -- (a0) -- (b1) -- (a2);
  \draw (c2) -- (b0) -- (c1) -- (d2) -- (c0) -- (d1) -- (c2);
  % \draw (a1) -- (b3) -- (c1) -- (b5) -- (a1);
  % \draw (b1) -- (b3) -- (b4) -- (b5) -- (b1);

  \node at (2,-4) {(a)};

\begin{scope}[xshift=5cm]
   \foreach \i 
  [evaluate=\i as \j using int(2-\i)]
  in {0,1,2}{
  \node (a\j) at (2*\i,1) {$a_{\j}$};
  \node (c\j) at (2*\i,-3) {$c_{\j}$};
  }
  \node (b0) at (0,-1) {$b_{0}$};
  \node (b1) at (2,0) {$b_{1}$};
  \node (b2) at (4,0) {$b_{2}$};

  \node (c) at (1,-1) {$c$};
  \node (d1) at (2,-2) {$d_1$};
  \node (d2) at (4,-2) {$d_2$};

  \draw (a2) -- (b0) -- (a1) -- (b2) -- (a0) -- (b1) -- (a2);
  \draw (c2) -- (b0) -- (c1) -- (d2) -- (c0) -- (d1) -- (c2);
  \draw (b1) -- (c) -- (d1);

  \node at (2,-4) {(b)};
\end{scope}

\begin{scope}[xshift=10cm]
   \foreach \i 
  [evaluate=\i as \j using int(2-\i)]
  in {0,1,2}{
  \node (a\j) at (2*\i,1) {$a_{\j}$};
  \node (c\j) at (2*\i,-3) {$c_{\j}$};
  }
  \node (b0) at (0,-1) {$b_{0}$};
  \node (b1) at (2,0) {$b_{1}$};
  \node (b2) at (4,0) {$b_{2}$};

  \node (c) at (1,-1) {$c$};
  \node (d1) at (2,-2) {$d_1$};
  \node (d2) at (4,-2) {$d_2$};

  \draw (a2) -- (b0) -- (a1) -- (b2) -- (a0) -- (b1) -- (a2);
  \draw (c2) -- (b0) -- (c1) -- (d2) -- (c0) -- (d1) -- (c2);
  \draw (d2) -- (b1) -- (c) -- (d1) -- (b2);

  \node at (2,-4) {(c)};
\end{scope}

\begin{scope}[xshift=2cm, yshift=-6cm]
   \foreach \i 
  [evaluate=\i as \j using int(2-\i)]
  in {0,1,2}{
  \node (a\j) at (2*\i,1) {$a_{\j}$};
  \node (c\j) at (2*\i,-3) {$c_{\j}$};
  }
  \node (b0) at (0,-1) {$b_{0}$};
  \node (b1) at (2,0) {$b_{1}$};
  \node (b2) at (4,0) {$b_{2}$};

  \node (c) at (1,-1) {$c$};
  \node (d1) at (2,-2) {$d_1$};
  \node (d2) at (4,-2) {$d_2$};

  \draw (a2) -- (b0) -- (a1) -- (b2) -- (a0) -- (b1) -- (a2);
  \draw (c2) -- (b0) -- (c1) -- (d2) -- (c0) -- (d1) -- (c2);
  \draw (d2) -- (b1) -- (c) -- (d1) -- (b2);
  \draw (c1) -- (c) -- (a1);
  \draw (c1) -- (b2);
  \draw (d2) -- (a1);
 
  \node at (2,-4) {(d)};
\end{scope}

\begin{scope}[xshift=8cm,yshift=-6cm]
  \foreach \i 
  [evaluate=\i as \j using int(2-\i)]
  in {0,1,2}{
  \node (a\j) at (2*\i,1) {$a_{\j}$};
  \node (c\j) at (2*\i,-3) {$c_{\j}$};
  }
  \node (b0) at (0,-1) {$b_0$};
  \node (b1) at (2,0) {$b_{1}$};
  \node (b2) at (4,0) {$b_2$};

  \node (c) at (1,-1) {$c$};
  \node (d1) at (2,-2) {$d_1$};
  \node (d2) at (4,-2) {$d_2$};

  \draw[dotted] (b0) -- (a2) -- (b1) -- (a0) -- (b2);
  \draw[dotted] (c2) -- (b0) -- (c1);
  \draw (d2) -- (c0) -- (d1) -- (c2);
  \draw (d2) -- (b1) -- (c) -- (d1);
  \draw[dotted] (d1) -- (b2);
  \draw (c1) -- (c);
  \draw[dotted] (c1) -- (b2);

  \node at (2,-4) {(e)};
\end{scope}

 \end{tikzpicture}
   \caption{1+5}
   \label{3333-15}
  \end{figure}

	Since $c\not\in \mnl(\cB_5)\cup \mxl(\cB_5)$ and $\hat{F}^{\cB}_c\subset \{b_1,b_2\}$, 
	we have $0<\card{\hat{F}^{\cB}_c}\leq 2$. 
	We may assume that $\card{\hat{F}^{\cB}_c}=1$  
	and, by considering the opposite, $\card{\hat{U}^{\cB}_c}=1$. 
	We may assume that $\hat{F}^{\cB}_c=\{b_1\}$ and $\hat{U}^{\cB}_c=\{d_1\}$. 

	If $d_2<b_2$, then, since $\cB_5$ is connected, we have  $d_2<b_1$ or $d_1<b_2$,  
	and hence $\card{\hat{F}^{\cB}_{d_2}}=2$ or $\card{\hat{U}^{\cB}_{b_2}}=2$, 
	therefore $X$ splits into smaller spaces. 

	Suppose $d_2\not<b_2$. 
	Since $\cB_5$ is connected, we have $d_2<b_1$ and $d_1<b_2$. 
	Since $c$, $b_2$, and $d_2$ are not beat points, 
	we have $c<a_1$, $c> c_1$, $b_2>c_1$ and $d_2<a_1$. 
	Therefore $X$ is isomorphic to the space of \cref{3333-15} (d). 
	Since $U_{a_0}\cap U_{a_2} =U_{b_1}\simeq *$, 
	$X$ splits into smaller spaces by \cref{misthreeandcontractible}.
	In fact, we easily see that 
	\begin{align*}
	 \hat{U}_{a_1}&\simeq_w S^1\vee S^1 \vee S^1, \\
	 X\simeq_w \bS(\hat{U}_{a_1}) &\simeq_w S^2\vee S^2 \vee S^2.
	\end{align*}

	%   \begin{figure}[h]
	%   \centering
	%  \begin{tikzpicture}
	%   \foreach \i 
	%   [evaluate=\i as \j using int(2-\i)]
	%   in {0,1,2}{
	%   \node (a\j) at (2*\i,1) {$a_{\j}$};
	%   \node (c\j) at (2*\i,-3) {$c_{\j}$};
	%   }
	%   \node (b0) at (0,-1) {$b_{0}$};
	%   \node (b1) at (2,0) {$b_{1}$};
	%   \node (b2) at (4,0) {$b_{2}$};
	
	%   \node (c) at (1,-1) {$c$};
	%   \node (d1) at (2,-2) {$d_1$};
	%   \node (d2) at (4,-2) {$d_2$};
	
	%   \draw[dotted] (a2) -- (b0) -- (a1) -- (b2) -- (a0) -- (b1) -- (a2);
	%   \draw (c2) -- (b0) -- (c1) -- (d2) -- (c0) -- (d1) -- (c2);
	%   \draw[dotted] (d2) -- (b1) -- (c);
	%   \draw (c) -- (d1) -- (b2);
	%   \draw[dotted] (c) -- (a1) -- (d2);
	%   \draw (c) -- (c1) -- (b2);
	   %  \end{tikzpicture}
	    %   \end{figure}

  \item The case where $\cB$ is connected.
	We have $\card{\mnl(\cB)}\leq 3$ 
	and we may assume that $\card{\mnl(\cB)}=3$, that is, $m=m'=n=n'=3$.  
	In this case, we may assume that 
	$\mxl(X)\cup \mxl{\cB}$ and $\mnl(X)\cup \mnl(\cB)$ 
	are isomorphic to $S^1_3$ as in \cref{3333-1} (a).
	We may also assume that $\card{\hat{F}^{\cB}_b}\ne 2$ and $\card{\hat{U}^{\cB}_b}\ne 2$ 
	for all $b\in \cB$, therefore $\cB$ is isomorphic to one of graphs in \cref{3by3bipartite}. 

	If there exists an element $b\in \cB$ such that $\card{\hat{F}^{\cB}_b}=1$, 
	say, $\card{\hat{F}^{\cB}_{d_0}}=1$, then 
	$\cB$ is isomorphic to the one in \cref{3333-1} (b). 
	Since $d_0$ is not a beat point and $d_0\not< b_0,b_2$, we see that $d_0\prec a_1$. 
	Proceeding similarly, we see that $X$ is isomorphic to the space of \cref{3333-1} (c).  
	Since $U_{a_0}\cap U_{a_2}=U_{b_1}\simeq *$ (see \cref{3333-1} (d)), 
	$X$ splits into smaller spaces by \cref{misthreeandcontractible}
	(In fact, $X\simeq_wS^2\vee S^2 \vee S^2$). 

	If $\card{\hat{F}^{\cB}_b}=3$ for all $b\in B$, 
	then $X$ is isomorphic to the space of \cref{3333-1} (e).
	Since $U_{a_0}\cap U_{a_2}=U_{b_1}\simeq *$ (see \cref{3333-1} (f)), 
	$X$ splits into smaller spaces by \cref{misthreeandcontractible}
	(In fact, $X\simeq_wS^3$). 

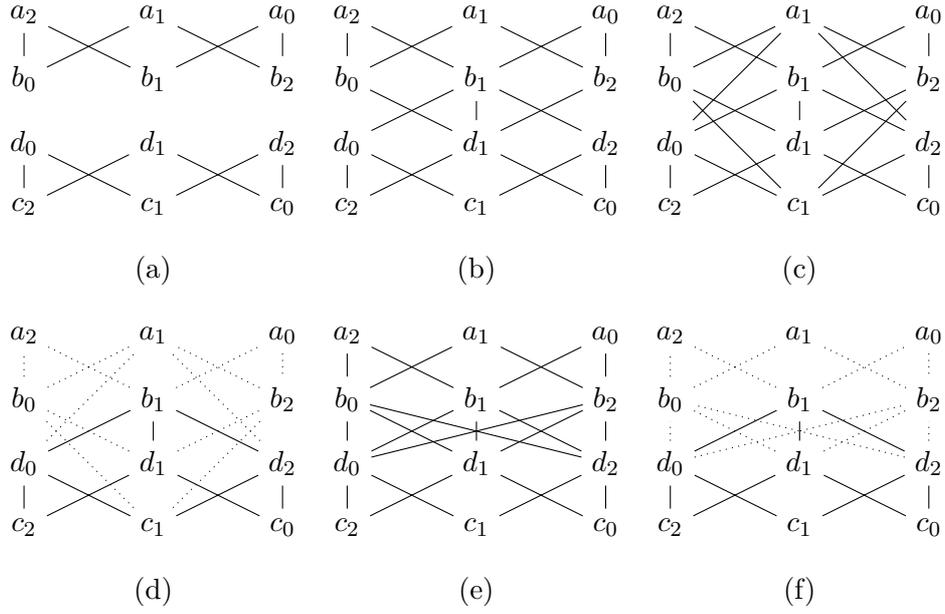
\begin{figure}[h]
 \centering
 \begin{tikzpicture}[scale=.85]
  \foreach \i 
  [evaluate=\i as \j using int(2-\i)]
  in {0,1,2}{
  \node (a\j) at (2*\i,3) {$a_{\j}$};
  \node (b\i) at (2*\i,2) {$b_{\i}$};
  \node (d\i) at (2*\i,1) {$d_{\i}$};
  \node (c\j) at (2*\i,0) {$c_{\j}$};
  }
  \draw (b0) -- (a2) -- (b1) -- (a0) -- (b2) -- (a1) -- (b0);
  \draw (d0) -- (c2) -- (d1) -- (c0) -- (d2) -- (c1) -- (d0);
  
  \node at (2,-1) {(a)};

  \begin{scope}[xshift=5cm]
  \foreach \i 
  [evaluate=\i as \j using int(2-\i)]
  in {0,1,2}{
  \node (a\j) at (2*\i,3) {$a_{\j}$};
  \node (b\i) at (2*\i,2) {$b_{\i}$};
  \node (d\i) at (2*\i,1) {$d_{\i}$};
  \node (c\j) at (2*\i,0) {$c_{\j}$};
  }
  \draw (b0) -- (a2) -- (b1) -- (a0) -- (b2) -- (a1) -- (b0);
  \draw (d0) -- (c2) -- (d1) -- (c0) -- (d2) -- (c1) -- (d0);
  \draw (d0) -- (b1);
  \draw (d1) -- (b0);
  \draw (d1) -- (b1);
  \draw (d1) -- (b2);
  \draw (b1) -- (d2);

  \node at (2,-1) {(b)};
  \end{scope}

  \begin{scope}[xshift=10cm]
  \foreach \i 
  [evaluate=\i as \j using int(2-\i)]
  in {0,1,2}{
  \node (a\j) at (2*\i,3) {$a_{\j}$};
  \node (b\i) at (2*\i,2) {$b_{\i}$};
  \node (d\i) at (2*\i,1) {$d_{\i}$};
  \node (c\j) at (2*\i,0) {$c_{\j}$};
  }
  \draw (b0) -- (a2) -- (b1) -- (a0) -- (b2) -- (a1) -- (b0);
  \draw (d0) -- (c2) -- (d1) -- (c0) -- (d2) -- (c1) -- (d0);
  \draw (d0) -- (b1);
  \draw (d1) -- (b0);
  \draw (d1) -- (b1);
  \draw (d1) -- (b2);
  \draw (b1) -- (d2);
  \draw (d0) -- (a1);
  \draw (b0) -- (c1);
  \draw (d2) -- (a1);
  \draw (b2) -- (c1);

  \node at (2,-1) {(c)};
  \end{scope}

  \begin{scope}[yshift=-5cm]
  \foreach \i 
  [evaluate=\i as \j using int(2-\i)]
  in {0,1,2}{
  \node (a\j) at (2*\i,3) {$a_{\j}$};
  \node (b\i) at (2*\i,2) {$b_{\i}$};
  \node (d\i) at (2*\i,1) {$d_{\i}$};
  \node (c\j) at (2*\i,0) {$c_{\j}$};
  }
  \draw[dotted] (b0) -- (a2) -- (b1) -- (a0) -- (b2) -- (a1) -- (b0);
  \draw (d0) -- (c2) -- (d1) -- (c0) -- (d2) -- (c1) -- (d0);
  \draw (d0) -- (b1);
  \draw[dotted] (d1) -- (b0);
  \draw (d1) -- (b1);
  \draw[dotted] (d1) -- (b2);
  \draw (b1) -- (d2);
  \draw[dotted] (d0) -- (a1);
  \draw[dotted] (b0) -- (c1);
  \draw[dotted] (d2) -- (a1);
  \draw[dotted] (b2) -- (c1);

  \node at (2,-1) {(d)};
  \end{scope}

  \begin{scope}[xshift=5cm,yshift=-5cm]
     \foreach \i 
  [evaluate=\i as \j using int(2-\i)]
  in {0,1,2}{
  \node (a\j) at (2*\i,3) {$a_{\j}$};
  \node (b\i) at (2*\i,2) {$b_{\i}$};
  \node (d\i) at (2*\i,1) {$d_{\i}$};
  \node (c\j) at (2*\i,0) {$c_{\j}$};
  }
  \draw (b0) -- (a2) -- (b1) -- (a0) -- (b2) -- (a1) -- (b0);
  \draw (d0) -- (c2) -- (d1) -- (c0) -- (d2) -- (c1) -- (d0);
  \draw (d0) -- (b0);
  \draw (d0) -- (b1);
  \draw (d0) -- (b2);
  \draw (d1) -- (b0);
  \draw (d1) -- (b1);
  \draw (d1) -- (b2);
  \draw (b1) -- (d2);
  \draw (d2) -- (b0);
  \draw (d2) -- (b2);

   \node at (2,-1) {(e)};
  \end{scope}

  \begin{scope}[xshift=10cm,yshift=-5cm]
     \foreach \i 
  [evaluate=\i as \j using int(2-\i)]
  in {0,1,2}{
  \node (a\j) at (2*\i,3) {$a_{\j}$};
  \node (b\i) at (2*\i,2) {$b_{\i}$};
  \node (d\i) at (2*\i,1) {$d_{\i}$};
  \node (c\j) at (2*\i,0) {$c_{\j}$};
  }
  \draw[dotted] (b0) -- (a2) -- (b1) -- (a0) -- (b2) -- (a1) -- (b0);
  \draw (d0) -- (c2) -- (d1) -- (c0) -- (d2) -- (c1) -- (d0);
  \draw[dotted] (d0) -- (b0);
  \draw (d0) -- (b1);
  \draw[dotted] (d0) -- (b2);
  \draw[dotted] (d1) -- (b0);
  \draw (d1) -- (b1);
  \draw[dotted] (d1) -- (b2);
  \draw (b1) -- (d2);
  \draw[dotted] (d2) -- (b0);
  \draw[dotted] (d2) -- (b2);

   \node at (2,-1) {(f)};
  \end{scope}

  \end{tikzpicture}
 \caption{$l=6$, $m=m'=n=n'=3$, and $\cB$ is connected.}
 \label{3333-1}
\end{figure}

 \end{enumerate}
\end{proof}

 \begin{lem}
  \label{l6mn3mp4}
  If %$X$ is a connected minimal finite space with 
  $l=6$, $m=n=3$, and $m'=4$, then $X$ splits into smaller spaces.
 \end{lem}

\begin{proof}
 If $n'=\card{\mnl(\cB)}\leq 3$, 
 then the opposite $X^o$ splits into smaller spaces by \cref{misthreempistwo,lisone,l6mpmn3} 
 and so does $X$. 

 % Since $\card{\mxl(\cB)}=m'=4$, if there exists an element $a\in \mxl(X)$ such that 
 % $\gamma_a=\card{U_a\cap \mxl(\cB)}=4$,  
 % then $X=U_{a}\cup \mxl(X)\cup \mnl(X)$ and the result follows from \cref{Uamxlmnlsplit}. 
 % Therefore, we may assume that $\gamma_a\leq 3$ for all $a\in \mxl(X)$. 
 % In this case, 
 % since
 % \[
 %  \sum_{a\in \mxl(X)}\gamma_a = \sum_{b\in \mxl(\cB)}\alpha_b\geq 2\cdot 4 = 3+3+2,  %> 3+2+2, 
 % \]
 % we see that 
 % $\gamma_a\geq 2$ for all $a\in \mxl(X)$ and 
 % $\gamma_a=3$ for at least two of them.

 By \cref{mis3mpis4}, we may assume that $\mxl(X)\cup \mxl(\cB)$ is (a) or (b) of \cref{figmis3mpis4-re}.

\begin{figure}[h]
 \centering

 \begin{tikzpicture}
  \node (b0) at (-.5,0) {$b_0$};
  \foreach \i 
  [evaluate=\i as \j using int(\i+1)]  
  in {0,1,2}{
  \node (a\i) at (\j,1.5) {$a_{\i}$};
  \node (b\j) at (\j,0) {$b_{\j}$};
  \draw (b0) -- (a\i);
  }
  \draw (a0) -- (b1) -- (a1) -- (b3) -- (a2) -- (b2) -- (a0);

  \node at (1.5,-1) {(a)};

 \begin{scope}[xshift=5cm]
  \foreach \i 
  in {0,1,2}{
  \node (a\i) at (\i,1) {$a_{\i}$};
  \node (b\i) at (\i,0) {$b_{\i}$};
  }
  \node (b3) at (3,0) {$b_3$};
  \draw (a0) -- (b0) -- (a1) -- (b1) -- (a0);
  \draw (a0) -- (b2);
  \draw (a1) -- (b3);
  \draw (b2) -- (a2) -- (b3);

  \node at (1.5,-1) {(b)};
 \end{scope}
 \end{tikzpicture}
 \caption{$\mxl(X)\cup \mxl(\cB)$ for $m=3,m'=4$.}
 \label{figmis3mpis4-re}
\end{figure}

 % for all 
 % $b\in \mxl(\cB)$, there exists an element $a\in \mxl(X)$ such that 
 % $b< a$ and $\gamma_a\geq 3$. 
 % In fact, there exists an element $a\in \mxl(X)$ such that $\gamma_a\geq 3$. 
 % If $\gamma_a=4$, then $b<a$ for all $b\in \mxl(\cB)$. 
 % If $\gamma_a\leq 3$ for all $a\in \mxl(X)$, then, since $3+2+2<2\cdot 4$, 
 % there exist at least two elements such that $\gamma_a=3$. 
 % Suppose $\mxl(X)=\{a_0,a_1,a_2\}$ and  $\gamma_{a_0}=\gamma_{a_1}=3$. 
 % For any $b\in \mxl(\cB)$, since $\card{\{a_0,a_1,a_2\}\cap F_b}=\alpha_b\geq 2$,  
 % we have $\{a_1,a_2\}\cap F_b \ne\emptyset$; hence $b\leq a_1$ or $b\leq a_2$.

 Since $\card{\mxl(\cB)}=4$, 
 the number of connected components of $\cB$ is at most $4$. 

 \begin{enumerate}[label=(\alph*)]
  \item The case where $\cB$ has $4$ connected components.

	Note that each component has maximum since $\card{\mxl(\cB)}=4$. 
	The cardinalities of each component are  
	\[
	6 = 1+1+1+3 =1+1+2+2. 
	\]
	The case $1+1+2+2$ follows from \cref{mp3m5ormp5m3chain}.

	In the case $1+1+1+3$, let $b$ be the maximum element of the 
	component with $3$ elements. 
	% Since $\card{\{a_0,a_1,a_2\}\cap F_{b}}=\alpha_{b}\geq 2$,  
	% we have $\{a_0,a_1\}\cap F_{b} \ne\emptyset$; hence $b\leq a_0$ or $b\leq a_1$.
	% $a_0\in \mxl(X)$ an element such that $b<a_0$ and $\gamma_{a_0}\geq 3$. 
	% If $\gamma_{a_0}=4$, 
	% then $X=U_{a_0}\cup \mxl(X)\cup \mnl(X)$ and the result follows from \cref{Uamxlmnlsplit}. 
	% If $\gamma_{a_0}=3$, then 
	Note that $b\leq a_0$ or $b\leq a_1$.
	% We may assume $b\leq a_0$. 
	% Since $\gamma_{a_0}=3$, $\mxl(\cB) - U_{a_0}$ consists of a single element, say, $\{b'\}$. 
	% We see that 
	If $b\leq a_0$, then $\mxl(\cB) - U_{a_0}=\{b_3\}$, and we see that
	$X=U_{a_0}\cup F_{b_3}\cup \mxl(X)\cup \mnl(X)$ 
	and the assumptions of \cref{bodyissimple} hold, therefore the result holds. 
	The case $b\leq a_1$ is similar.

  \item The case where $\cB$ has $3$ connected components.

	% Note that $n'=\card{\mnl(\cB)}\geq 3$ since each component has a minimal element.
	The possible cardinalities of each component are  
	\[
	   6=1+1+4 =1+2+3 =2+2+2,
	\]
	but, since $\card{\mxl(\cB)}=4$, the last case does not occur. 

	In the case $1+2+3$, we see that $n'=3$ and the result follows.

	In the case $1+1+4$, 
	we see that $n'=\card{\mnl(\cB)}=3$ or $4$. 
	We have to consider the case $n'=4$. 
	Let $\cB_4$ be the connected component with $\card{\cB_4}=4$. 
	We have $\card{\mxl(\cB_4)}=2$, and  
	in the case $n'=4$, we have 
	$\card{\mnl(\cB_4)}=2$. We put $\mnl(\cB_4)=\{b_4,b_5\}$. 
	Since $\cB_4$ is connected, there exists an element $b\in \mnl(\cB_4)$ such that 
	$\widehat{F}^{\cB_4}_{b}=\cB_4\cap \widehat{F}_b=\mxl(\cB_4)$ (see \cref{reallysmallspaces}).  
	% We may assume that $\mxl(\cB_4)\subset F_{b_4}$ (see \cref{reallysmallspaces}). 
	In the case where $\mxl(X)\cup \mxl(\cB)$ is (b) of \cref{figmis3mpis4-re}, 
	if $\mxl(\cB_4)\ne \{b_0,b_1\}$, then 
	we see that $\widehat{F}_{b}\simeq *$, that is, $b$ is a weak beat point. 
	Hence $X\simeq_w X-\{b\}$ and the result follows.
	In the case where $\mxl(X)\cup \mxl(\cB)$ is (a) of \cref{figmis3mpis4-re}, 
	if $b_0\not\in \mxl(\cB_4)$, 
	then 
	we see that %$\widehat{F}_{b}\simeq *$, that is, 
	$b$ is a weak beat point
	and the result follows. 
	By the symmetry, we may consider the case where $\mxl(\cB_4)=\{b_0,b_1\}$.

	Suppose $\mxl(\cB_4)=\{b_0,b_1\}$.
	By considering the opposite, we may suppose that $\mxl(X) \cup \mxl(\cB)$ and 
	$\mnl(X)\cup \mnl(\cB)$ are those of the spaces in \cref{114np4}.

	% \begin{figure}[h]
	%  \centering

	%  \begin{tikzpicture}[yscale=-1]
	%   \node (b0) at (-.5,0) {$b_4$};
	%   \node (b1) at (1,0) {$b_5$};
	%   \node (b2) at (2,0) {$b_2$};
	%   \node (b3) at (3,0) {$b_3$};

	%   \foreach \i 
	%   [evaluate=\i as \j using int(\i+1)]  
	%   in {0,1,2}{
	%   \node (a\i) at (\j,1.5) {$c_{\i}$};
	%   \draw (b0) -- (a\i);
	%   }
	%   \draw (a0) -- (b1) -- (a1) -- (b3) -- (a2) -- (b2) -- (a0);
	  
	%   \node at (1.5,2) {(a)};
	  
	%   \begin{scope}[xshift=5cm]
	%    \foreach \i 
	%    in {0,1,2}{
	%    \node (a\i) at (\i,1) {$c_{\i}$};
	%    }
	%    \node (b3) at (3,0) {$b_3$};
	%    \node (b0) at (0,0) {$b_4$};
	%    \node (b1) at (1,0) {$b_5$};
	%    \node (b2) at (2,0) {$b_2$};
	%    \draw (a0) -- (b0) -- (a1) -- (b1) -- (a0);
	%    \draw (a0) -- (b2);
	%    \draw (a1) -- (b3);
	%    \draw (b2) -- (a2) -- (b3);
	   
	%    \node at (1.5,2) {(d)};
	%   \end{scope}
	%  \end{tikzpicture}
	%  \caption{$n=3,n'=4,\mnl(X)\cup \mnl(\cB)$.}
	%  \label{figmis3mpis4-op}
	% \end{figure}

	\begin{figure}[h]
	 \centering 
	  \begin{tikzpicture}
	   
	   \node (b0) at (0,.5) {$b_{0}$};
	   \node (b1) at (1,.5) {$b_{1}$};
	   \node (b2) at (2,0) {$b_{2}$};
	   \node (b3) at (3,0) {$b_{3}$};
	   \node (b4) at (0,-.5) {$b_{4}$};
	   \node (b5) at (1,-.5) {$b_{5}$};

	   \foreach \i 
	   in {0,1,2}{
	   \node (a\i) at (\i,1.5) {$a_{\i}$};
	   \node (c\i) at (\i,-1.5) {$c_{\i}$};
	   }

	   \draw (a0) -- (b0) -- (a1) -- (b1) -- (a0);
	   \draw (a0) -- (b2);
	   \draw (a1) -- (b3);
	   \draw (b2) -- (a2) -- (b3);

	   \draw (c0) -- (b4) -- (c1) -- (b5) -- (c0);
	   \draw (c0) -- (b2);
	   \draw (c1) -- (b3);
	   \draw (b2) -- (c2) -- (b3);

	   % \draw[blue] (b0) -- (b4) -- (b1) -- (b5);

	   \draw[dashed] (-.25,-.75) rectangle (1.25,.75);
	   % \node at (-.75,0) {$\cB_4$};

	   \node at (1.5,-2.5) {(a)};

	   \begin{scope}[xshift=4cm]
	    \node (b0) at (0,.5) {$b_{0}$};
	    \node (b1) at (1,.5) {$b_{1}$};
	    \node (b2) at (2,0) {$b_{2}$};
	    \node (b3) at (3,0) {$b_{3}$};

	    \node (b4) at (0,-.5) {$b_{4}$};
	    \node (b5) at (1,-.5) {$b_{5}$};

	    \foreach \i 
	    [evaluate=\i as \j using int(\i+1)]  
	    in {0,1,2}{
	    \node (a\i) at (\i,1.5) {$a_{\i}$};
	    \node (c\i) at (\j,-1.5) {$c_{\i}$};
	    }

	    \draw (a0) -- (b0) -- (a1) -- (b1) -- (a0);
	    \draw (a0) -- (b2);
	    \draw (a1) -- (b3);
	    \draw (b2) -- (a2) -- (b3);

	    \draw (c0) -- (b5) -- (c1) -- (b3) -- (c2) -- (b2) -- (c0);
	    \draw (c0) -- (b4);
	    \draw (c1) -- (b4);
	    \draw (c2) -- (b4);

	   \draw[dashed] (-.25,-.75) rectangle (1.25,.75);
 
	   \node at (1.5,-2.5) {(b)};

	   \end{scope}

	   \begin{scope}[xshift=8cm]
	    \node (b0) at (0,.5) {$b_{0}$};
	    \node (b1) at (1,.5) {$b_{1}$};
	    \node (b2) at (2,0) {$b_{2}$};
	    \node (b3) at (3,0) {$b_{3}$};

	    \node (b4) at (0,-.5) {$b_{4}$};
	    \node (b5) at (1,-.5) {$b_{5}$};

	    \foreach \i 
	    [evaluate=\i as \j using int(\i+1)]  
	    in {0,1,2}{
	    \node (a\i) at (\j,1.5) {$a_{\i}$};
	    \node (c\i) at (\j,-1.5) {$c_{\i}$};
	    \draw (b0) -- (a\i);
	    \draw (b4) -- (c\i);
	    }
	    \draw (a0) -- (b1) -- (a1) -- (b3) -- (a2) -- (b2) -- (a0);
	    \draw (c0) -- (b5) -- (c1) -- (b3) -- (c2) -- (b2) -- (c0);
	    % \draw (b0) -- (b4) -- (b1) -- (b5);
	    % \draw (b5) -- (a2); 
	    % \draw (c2) -- (b0);

	   \draw[dashed] (-.25,-.75) rectangle (1.25,.75);

	    % \draw[very thick, cyan, opacity=.4] (c0) -- (b4) -- (c1) -- (b5) -- (b0) -- (c2) -- (b2) -- (c0) -- (b5);
	    % \draw[very thick, cyan, opacity=.4] (b4) -- (c2) -- (b3) -- (c1);
	    
	   \node at (1.5,-2.5) {(c)};

	   \end{scope}

	  \end{tikzpicture}
	 \caption{$1+1+4$, $n'=4$.}
	 \label{114np4}
	\end{figure}
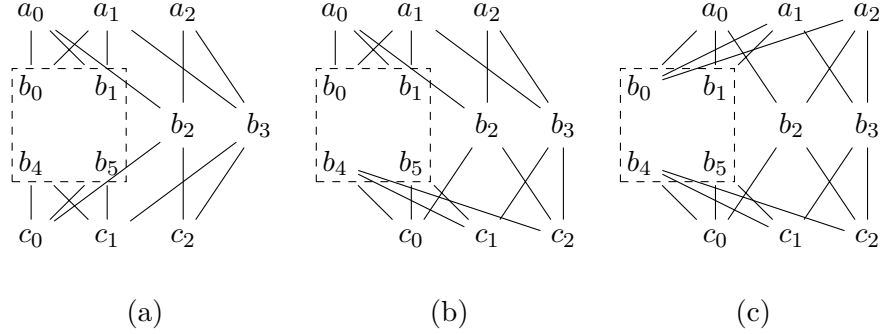

	We use \cref{misthreeand1connected}.

	Note that, in any case, we have $U_{a_0}=X-\{a_1,a_2,b_3\}$, $U_{a_1}=X-\{a_0,a_2, b_2\}$, and
	$U_{a_0}\cap U_{a_1}=X-\mxl(X)-\{b_2,b_3\}=\cB_4 \cup \mnl(X)$, 
	and in the cases (a) and (b), 
	$\{b_2,b_3\}\cup \mnl(X) \subset \widehat{U}_{a_2}\subset \{b_2,b_3,b_4,b_5\}\cup \mnl(\cB)$ and so $\hgt\left(\widehat{U}_{a_2}\right)=1$, 
	therefore $\widehat{U}_{a_2}$ is weak homotopy equivalent to a simplicial complex of dimension at most $1$.

	Consider the case (a). If $c_2<b_0$ or $c_2<b_1$, then $U_{a_0}\cap U_{a_1}$ is connected, 
	and hence $X$ splits into smaller spaces by \cref{misthreeand1connected}. 
	Otherwise, we see that $\widehat{F}_{c_2}=\{b_2,b_3\}\cup \mxl(X)\simeq *$, namely, $c_2$ is a weak beat point, and the result follows.

	In the case (b), $U_{a_0}\cap U_{a_1}$ is connected and the result follows.

	In the case (c), $U_{a_0}\cap U_{a_1}$ is connected. 

	Note that $U_{b_2}\cup U_{b_3}\simeq *$ 
	and 
	$\{b_0,b_2,b_3\}\cup \mnl(X) \subset \widehat{U}_{a_2}\subset \{b_0,b_2,b_3,b_4,b_5\}\cup \mnl(X)$. 

	If $b_5\not<b_0$, then $b_4<b_0$ and we see that $b_0$ is a down beat point, which contradicts the minimality of $X$, therefore, $b_5 < b_0$.

	If $\{b_4,b_5\}\subset U_{b_0}$, then we have 
	$\widehat{U}_{a_2}=U_{b_0}\cup U_{b_2} \cup U_{b_3}$ and 
	$U_{b_0}\cap \left(U_{b_2} \cup U_{b_3}\right) = \mnl(X)$. 
	Therefore, $\widehat{U}_{a_2}\simeq_w \bS \mnl(X) \simeq_w S^1\vee S^1$.

	If $b_4\not< b_0$, then we have $b_4<b_1$ and $b_5<b_0, b_1$. 
	Since $b_4$ is not an up beat point, we have $b_4<a_2$, 
	and since $b_0$ is not a down beat point, we have $c_2<b_0$.
	We see that $b_5$ is an up beat point of $\widehat{U}_{a_2}$, 
	and hence $\widehat{U}_{a_2}$ is homotopy equivalent to a space of height $1$. 
	See \cref{fig-l6mn3mp4-c-c}.

	Therefore, $X$ splits into smaller spaces by \cref{misthreeand1connected}.

	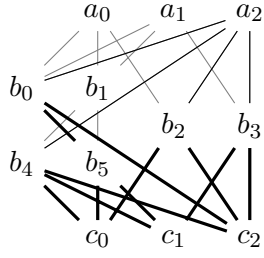
\begin{figure}[h]
	 \centering

	  \begin{tikzpicture}

	   \node (b0) at (0,.5) {$b_{0}$};
	   \node (b1) at (1,.5) {$b_{1}$};
	   \node (b2) at (2,0) {$b_{2}$};
	   \node (b3) at (3,0) {$b_{3}$};

	   \node (b4) at (0,-.5) {$b_{4}$};
	   \node (b5) at (1,-.5) {$b_{5}$};

	   \foreach \i 
	   [evaluate=\i as \j using int(\i+1)]  
	   in {0,1,2}{
	   \node (a\i) at (\j,1.5) {$a_{\i}$};
	   \node (c\i) at (\j,-1.5) {$c_{\i}$};
	   \draw[thin,gray] (b0) -- (a\i);
	   \draw[thin,gray] (b4) -- (c\i);
	   }
	   \draw[thin,gray] (a0) -- (b1) -- (a1) -- (b3) -- (a2) -- (b2) -- (a0);

	   \draw[thin,gray] (c0) -- (b5) -- (c1) -- (b3) -- (c2) -- (b2) -- (c0);
	   
	   \draw[thin,gray] (b0) -- (b5) -- (b1) -- (b4);
	   \draw[thin,gray] (b4) -- (a2); 
	   \draw[thin,gray] (c2) -- (b0);

	   \draw (a2) -- (b0);
	   \draw (a2) -- (b2);
	   \draw (a2) -- (b3);
	   \draw (a2) -- (b4);

	   \draw[very thick] (c0) -- (b4) -- (c1) -- (b5) -- (b0) -- (c2) -- (b2) -- (c0) -- (b5);
	   \draw[very thick] (b4) -- (c2) -- (b3) -- (c1);

	  \end{tikzpicture}
	 \caption{$\widehat{U}_{a_2}$.}
	 \label{fig-l6mn3mp4-c-c}
	\end{figure}

  \item The case where $\cB$ has $2$ connected components.
	The cardinalities of each component are  
	\[
	6=1+5 =2+4 =3+3 
	\]
	and we see that $n'\leq 3$. 
  \item The case where $\cB$ is connected. In this case, we see that $n'\leq 2$.
 \end{enumerate}

\end{proof}

\begin{lem}
 \label{l6mn3mp5}
 If %$X$ is a connected minimal finite space with 
 $l=6$, $m=n=3$, and $m'=5$, then $X$ splits into smaller spaces.
\end{lem}
\begin{proof}
 If $\cB$ has $5$ connected components, then $X$ splits into smaller spaces by \cref{mp3m5ormp5m3chain}. 
 Otherwise, we see that $n'\leq 4$ and hence $X^o$ splits into smaller spaces by \cref{lisone,misthreempistwo,l6mpmn3,l6mn3mp4}, 
 and so does $X$.
\end{proof}

 \begin{lem}
  \label{l6mn3Bantichain}
  If %$X$ is a connected minimal finite space with 
  $l=6$, $m=n=3$, and $\cB$ is an antichain, then $X$ splits into smaller spaces.
 \end{lem}
 \begin{proof}
  We rely on a result of Cianci-Ottina \cite{CianciOttina}. 
  Let
  \[
  \cR=\Set{(a,b,x)\in \mxl(X)\times \mnl(X)\times \cB | a\not\geq x \text{ and } b\not\leq x } 
  \]
  and $r=\card{\cR}$.
  Note that $r=\sum_{x\in \cB } (m-\alpha_x)(n-\beta_x)$, 
  and, for $(a,b)\in \mxl(X)\times \mnl(X)$ and $x\in X$,  $(a,b,x)\in \cR$ if and only if $x\not\in U_a\cup F_b\cup \mxl(X) \cup \mnl(X)$. 
  By considering the projection $\cR \to \mxl(X)\times \mnl(X)$, one sees that 
  if $r<mn$, then there exist elements $a\in \mxl(X)$ and $b\in \mnl(X)$ such that 
  $X=U_a \cup F_b \cup \mxl(X) \cup \mnl(X)$.

  In our case, since $2\leq \alpha_x,\beta_x\leq 3$ for all $x\in \cB$, we see that 
  \[
   r=\sum_{x\in \cB}(3-\alpha_x)(3-\beta_x)\leq \sum_{x\in \cB}(3-2)(3-2)=\card{\cB}=6<3\cdot 3 =mn.
  \]
  Therefore, $X$ splits into smaller spaces by \cref{unchainsplit}.
 \end{proof}

 \begin{rem}
  As one sees, if $l\leq 8$, $m=n=3$, and $\cB$ is an antichain, then $X$ splits into smaller spaces.
 \end{rem}

 \begin{coro}
  \label{l6mn3}
  If %$X$ is a connected minimal finite space with 
  $l=6$ and $m=n=3$, then $X$ splits into smaller spaces.
 \end{coro}
 \begin{proof}
  We have $1\leq m'\leq 6$. The case $m'=1$ follows from \cref{lisone} and the case $m'=2$ follows from \cref{misthreempistwo}.
 \end{proof}

\section{Proof of the main theorem}
\label{sec-proof}

\begin{thm}
 If $X$ is a connected finite space with $\Card{X}\leq 12$ or $\Card{X}=13$ and $l\leq 3$, then 
 $X$ has a weak homotopy type of a wedge of spheres.
\end{thm}
\begin{proof}
 Induction on $\card{X}$. 
 
 If $\card{X}=1$, then $X$ is a wedge of $0$-copies of spheres.

 Suppose $\card{X}>1$ and the result holds for spaces whose cardinalities are smaller than $\card{X}$. 

 If $X$ is not $T_0$, then $X$ is homotopy equivalent to its maximal quotient poset, 
 whose cardinality is smaller than $X$. 
 If $X$ is $T_0$ and has a beat point $b$, then $X\simeq X-\{b\}$. 
 If $X$ is minimal, then, by \cref{card11,card12,card13andl3}, $X$ splits into smaller spaces. 
 Since each wedge summand is a wedge of spheres by the induction hypothesis, so does $X$.
\end{proof}

\appendix

\section{A proof of \cref{suspensionlemma}}
\label{a-suspensionlemma}

We give a proof of \cref{suspensionlemma} by induction on $l+m$.

If $l=m=1$, then we have
\begin{align*}
 K&=L_1\cup M_1 \\
 &\simeq \susp (L_1\cap M_1) 
\end{align*}
and, since $K$ is connected, $L_1\cap M_1\ne \emptyset$. 
Therefore $ n=1-1-1+1=0$ and the result holds in this case.

Suppose $l+m>2$. We may assume $l>1$.
Put $L'=\coprod_{i>1}L_i$. 
Let $K_1,\dots,K_k$ be the connected components of $L'\cup M$.
Since $L_i$ and $M_j$ are connected, we see that 
there exist decompositions 
\begin{align*}
 \{2,\dots,l\}&=\coprod_{i=1}^k I_i &
 \{1,\dots,m\}&=\coprod_{i=1}^k J_i 
\end{align*}
and
\[
  K_i=\bigcup_{s\in I_i}L_s \cup \bigcup_{t\in J_i}M_t 
\]
We put 
\begin{align*}
 L^i &= \bigcup_{s\in I_i}L_s, \\ 
 M^i&=\bigcup_{t\in J_i}M_t, \\
 n_1^i&=\card{\Set{j\in J_i| L_1\cap M_j \ne \emptyset}} - 1.
\end{align*}
Note that, if $i\ne j$, then 
$L_s\cap M_t=\emptyset$ for all $s\in I_i$ and $t\in J_j$ 
because $K_i\cap K_j=\emptyset$.

Since $K_i=L^i\cup M^i$ and $\card{I_i}+\card{J_i}\leq l-1+m <l+m$, 
by the induction hypothesis, 
we have
\begin{align*}
 K_i &\simeq \left(\bigvee_{\substack{s\in I_i \\ t\in J_i \\ L_s\cap M_t\ne \emptyset}}
 \susp (L_s\cap M_t) \right) \vee \left(\bigvee_{n_i} S^1\right), \\
 n_i&=\card{\Set{(s,t)\in I_i\times J_i | L_s\cap M_t\ne\emptyset }} - \card{I_i} - \card{J_i} +1.
\end{align*}
We have
\begin{align*}
 K=L_1\cup L'\cup M &= L_1 \cup \left(\coprod K_i\right), \\
 L_1\cap \left(\coprod K_i\right) &=\coprod (L_1 \cap K_i),  \\
 L_1\cap K_i 
 &= L_1 \cap \left(L^i\cup M^i\right) 
 % \\   &
 = L_1 \cap M^i
 % = \coprod_{j\in J_i} (L_1 \cap M_j)
 = \coprod_{\substack{j\in J_i \\ L_1\cap M_j\ne \emptyset}} (L_1 \cap M_j),
\end{align*}
and, since $K$ is connected, $L_1\cap K_i \ne \emptyset$. 
Since $L_1\simeq *$
and the inclusion $\coprod (L_1\cap M_j) \to K_i$ is null homotopic 
because $L_1 \cap M_j \to M_j \subset  K_i$ is null homotopic and $K_i$ is connected, we have 
\begin{align*}
 K&\simeq K/L_1 \\ 
 &\cong \frac{\coprod K_i}{\coprod (L_1\cap K_i)} \\
 &=\bigvee_i \frac{K_i}{L_1\cap K_i} \\
 &=\bigvee_i \frac{K_i}{\coprod\limits_{\substack{j\in J_i \\ L_1\cap M_j\ne\emptyset}} (L_1\cap M_j)} \\
 &\simeq \bigvee_i\left(K_i \vee \susp \left(\coprod_{\substack{j\in J_i \\ L_1\cap M_j\ne\emptyset}} (L_1\cap M_j)\right)\right) 
\end{align*}
We have
\begin{align*}
 \bigvee_i \susp \left(\coprod_{\substack{j\in J_i \\ L_1\cap M_j\ne\emptyset}} (L_1\cap M_j)\right) 
 &\simeq \bigvee_i\left(\bigvee_{\substack{j\in J_i \\ L_1\cap M_j\ne\emptyset}} \susp(L_1\cap M_j) \vee 
 \bigvee_{n_1^i}S^1\right)
 \\
 &= \left(\bigvee_{\substack{j \\ L_1\cap M_j\ne\emptyset}} \susp(L_1\cap M_j) \right) \vee 
 \left(  \bigvee_{\sum n_1^i} S^1\right),
 \\
 \bigvee_i K_i
 &\simeq  \bigvee_i\left(
 \left(\bigvee_{\substack{s\in I_i \\ t\in J_i \\ L_s\cap M_t\ne\emptyset}} \susp (L_s\cap M_t) \right) 
 \vee \left(\bigvee_{n_i} S^1\right)
 \right) \\
 &= \left(\bigvee_{\substack{i>1, j \\ L_i\cap M_j\ne\emptyset}} \susp (L_i\cap M_j) \right) 
 \vee \left(\bigvee_{\sum n_i} S^1\right),
\end{align*}
and
\begin{align*}
 \sum n_1^i&=\sum_{i=1}^k \left(\card{\Set{j\in J_i| L_1\cap M_j \ne \emptyset}} - 1\right) \\
 &=\card{\Set{j\in J| L_1\cap M_j \ne \emptyset}} - k\\ 
 \sum n_i&=\sum_{i=1}^k
 \left(\card{\Set{(s,t)\in I_i\times J_i | L_s\cap M_t\ne\emptyset }} - \card{I_i} - \card{J_i} +1\right)\\ 
 &=\card{\Set{(i,j) | i>1, \; L_i\cap M_j\ne\emptyset }} - (l-1) - m +k\\ 
 \sum n_i + \sum n_1^i
 &=\card{\Set{(i,j) | L_i\cap M_j\ne\emptyset }} - l + 1 - m =n. 
\end{align*}
Therefore we have 
\[
   K\simeq \left(\bigvee_{\substack{i,j \\ L_i\cap M_j\ne\emptyset}} \susp \left(L_i\cap M_j\right)  \right)  \vee \left(\bigvee_{n} S^1\right)
\]
as desired. \qed

\subsection*{Acknowledgements}
The authors would like to thank T.~Matsushita and K.~Tanaka for helpful comments.

\noindent{Department of Mathematical Sciences, 
University of the Ryukyus\\
Nishihara-cho, Okinawa 903-0213, 
Japan}

\end{document}